\documentclass[a4paper,11pt]{amsart}
\usepackage{amssymb,amscd,amsxtra,xypic}
\usepackage[all]{xy}
\usepackage{array}
\usepackage{nicematrix}
\usepackage{textcomp}
\usepackage{bm}
\usepackage{mathtools}
\usepackage{longtable}
\usepackage{color, colortbl}
\usepackage{multirow}
\usepackage{booktabs}
\usepackage{nicematrix}
\usepackage{braket}

\definecolor{lightgray}{gray}{0.90}

\usepackage[pagebackref=true, linktocpage=true]{hyperref}
\hypersetup{%
bookmarksnumbered=true,%
colorlinks=true,%
linkcolor=blue,%
citecolor=blue,%
urlcolor=blue,%
setpagesize=false,%
pdftitle={},%
pdfauthor={Takuzo Okada}}

\theoremstyle{plain}
\newtheorem{Thm}{Theorem}[section]
\newtheorem{Lem}[Thm]{Lemma}
\newtheorem{Cor}[Thm]{Corollary}
\newtheorem{Prop}[Thm]{Proposition}
\newtheorem{Conj}[Thm]{Conjecture}

\theoremstyle{definition}
\newtheorem{Def}[Thm]{Definition}

\newtheorem{Rem}[Thm]{Remark}
\newtheorem*{Ack}{Acknowledgments}

\numberwithin{equation}{section}

\newcommand{\Proj}{\operatorname{Proj}}
\newcommand{\prt}{\partial}
\newcommand{\Sing}{\operatorname{Sing}}
\newcommand{\Spec}{\operatorname{Spec}}
\newcommand{\Cl}{\operatorname{Cl}}

\newcommand{\Pic}{\operatorname{Pic}}
\newcommand{\Int}{\operatorname{Int}}
\newcommand{\bNE}{\operatorname{\overline{NE}}}

\newcommand{\bMov}{\overline{\operatorname{Mov}}}

\newcommand{\Bs}{\operatorname{Bs}}

\newcommand{\mult}{\operatorname{mult}}

\newcommand{\wt}{\operatorname{wt}}

\newcommand{\rank}{\operatorname{rank}}
\newcommand{\ord}{\operatorname{ord}}
\newcommand{\Supp}{\operatorname{Supp}}

\newcommand{\red}{\mathrm{red}}
\renewcommand{\wt}{\operatorname{wt}}

\newcommand{\coeff}{\operatorname{coeff}}

\newcommand{\QI}{\mathrm{QI}}

\newcommand{\ntimes}{\! \times \!}
\newcommand{\br}{\mathrm{br}}
\newcommand{\lct}{\operatorname{lct}}
\newcommand{\omult}{\operatorname{omult}}

\newcommand{\vol}{\operatorname{vol}}

\newcommand{\mbA}{\mathbb{A}}

\newcommand{\mbC}{\mathbb{C}}

\newcommand{\mbF}{\mathbb{F}}
\newcommand{\mbG}{\mathbb{G}}

\newcommand{\mbP}{\mathbb{P}}
\newcommand{\mbQ}{\mathbb{Q}}

\newcommand{\mbT}{\mathbb{T}}
\newcommand{\mbU}{\mathbb{U}}

\newcommand{\mbZ}{\mathbb{Z}}

\newcommand{\mcF}{\mathcal{F}}

\newcommand{\mcM}{\mathcal{M}}

\newcommand{\mfm}{\mathfrak{m}}

\newcommand{\msF}{\mathsf{F}}
\newcommand{\msG}{\mathsf{G}}

\newcommand{\msI}{\mathsf{I}}

\newcommand{\msi}{\mathsf{i}}

\newcommand{\msp}{\mathsf{p}}
\newcommand{\msq}{\mathsf{q}}

\newcommand{\ratmap}{\dashrightarrow}

\newcommand{\bmu}{\boldsymbol{\mu}}

\makeatletter
\def\imod#1{\allowbreak\mkern10mu({\operator@font mod}\,\,#1)}
\makeatother

\title[K-stability and birational rigidity of Fano 3-fold WCIs]{K-stability and birational rigidity of Fano 3-fold weighted complete intersections}
\author[Takuzo Okada]{Takuzo Okada}
\address{Faculty of Mathematics, Kyushu University, Fukuoka 819-0395, Japan}
\email{tokada@math.kyushu-u.ac.jp}
\date{}

\begin{document}

\begin{abstract}
We show that any birationally rigid well-formed and quasismooth Fano 3-fold weighted complete intersection is K-stable. 
\end{abstract}

\maketitle

\tableofcontents

\section{Introduction} \label{sec:intro}

Throughout the paper, we work over the complex number field $\mbC$.
By a \textit{Fano variety}, we mean a normal projective $\mbQ$-factorial variety with only terminal singularities and with ample anticanonical divisor.
We consider two notions for Fano varieties: birational rigidity and K-stability.
Birational rigidity means rigidity under birational modifications into terminal Mori fiber spaces. 
More precisely, a Fano variety of Picard number $1$ is \textit{birationally rigid} if it is the unique terminal Mori fiber space in its birational equivalence class up to isomorphism.
K-(poly)stability of a Fano variety is equivalent to the existence of a K\"{a}hler--Einstein metric (see \cite[Theorem~1.6]{LXZ22}).
Although these two notions have completely different origins, they are both related to singularities of suitable divisors on Fano varieties.
In particular, we have the following expectation:

\begin{Conj}[{\cite[Conjecture~1.9]{KOWalpha}}]
\label{conj:BRKst}
A birationally rigid Fano variety is K-stable.
\end{Conj}

The main aim of this paper is to confirm Conjecture~\ref{conj:BRKst} for well-formed and quasismooth Fano 3-fold weighted complete intersections (WCIs, for short).
The class group of such a Fano 3-fold $X$ is isomorphic to $\mbZ$ and the positive integer $\iota_X$ such that $-K_X = \iota_X A$, where $A$ is the positive generator of the class group, is called the \textit{index} of $X$.
Well-formed and quasismooth Fano 3-fold WCIs are classified: the codimension of well-formed and quasismooth Fano 3-fold WCIs is bounded by $3$ and such Fano $3$-fold WCIs of codimension $1$, $2$ and $3$ consist of $130$, $125$ and $1$ families, respectively.
See \cite[Tables~5 and 6]{IF}, \cite[Tables~1 and 2]{BS1}, \cite[Tables~1 and 2]{BS2} (see also \cite[Tables~5-8]{DGO}) for the lists.
The classification of birationally rigid well-formed and quasismooth Fano 3-fold WCIs is completely settled for hypersurfaces and is almost complete for WCIs of higher codimension.

We have a clear picture for weighted hypersurfaces.

\begin{Thm}[{\cite{CP17}, \cite{ACP21}}]
Let $X$ be a well-formed and quasismooth Fano $3$-fold weighted hypersurface.
Then, $X$ is birationally rigid if and only if $\iota_X = 1$.
\end{Thm}

\begin{Thm}[{\cite{CO}}]
Let $X$ be a well-formed and quasismooth Fano $3$-fold weighted hypersurface of index $1$.
Then $X$ is K-stable.
\end{Thm}

Smooth complete intersections of three quadrics in $\mbP^6$ are the only well-formed and quasismooth Fano $3$-fold WCI of codimension $3$.
None of them are birationally rigid and all of them are K-stable (see \cite[Theorem~C]{AZ23}). 

We explain the codimension $2$ case.
It is proved in \cite{IP96} that a general complete intersection of a quadric and a cubic in $\mbP^5$ is birationally rigid.
The $84$ families of well-formed and quasismooth Fano $3$-fold WCIs of index $1$ other than complete intersections of a quadric and a cubic are investigated in \cite{OkadaI}, and it is proved that general members of the specific $18$ families listed in Table~\ref{table:BRcodim2} are birationally rigid and no member of the remaining $66$ families is birationally rigid.
After that the generality assumptions imposed in \cite{OkadaI} for the $18$ families are dropped in \cite{AZ16}.
Finally, it is proved in \cite{DG} that no well-formed and quasismooth Fano $3$-fold WCI of codimension $2$ and index $> 1$ is birationally rigid.
As a summary, we obtain the following:

\begin{Thm}[{\cite{IP96}, \cite{OkadaI}, \cite{AZ16}, \cite{DG}}]
Let $X$ be a well-formed and quasismooth Fano $3$-fold WCI of codimension $2$ other than a complete intersection of a quadric and a cubic in $\mbP^5$.
Then, $X$ is birationally rigid if and only if the index of $X$ is $1$ and it belongs to one of the specific $18$ families listed in \emph{Table~\ref{table:BRcodim2}}.
Moreover, a general smooth complete intersection of a quadric and a cubic in $\mbP^5$ is birationally rigid.
\end{Thm}

K-stability of smooth complete intersections of a quadric and a cubic in $\mbP^5$ is proved in \cite[Theorem~1.3]{Zhuang20} (without any generality assumptions).
K-stability of birationally rigid Fano $3$-fold WCIs of codimension $2$ and index $1$ is proved in \cite{KOWalpha} under generality assumptions.
The following is the main theorem of this paper in which we succeeded in dropping the generality assumptions.

\begin{Thm}
\label{thm:Kst}
Let $X$ be a well-formed and quasismooth Fano $3$-fold weighted complete intersection of codimension $2$.
If $X$ is birationally rigid, then $X$ is K-stable.
\end{Thm}

\begin{Cor}
\emph{Conjecture~\ref{conj:BRKst}} is true for well-formed and quasismooth Fano $3$-fold weighted complete intersections.
\end{Cor}

The proof of Theorem~\ref{thm:Kst} relies on the classification of maximal centers (see Definition~\ref{def:maxsing}) on birationally rigid well-formed and quasismooth Fano $3$-fold WCIs of codimension $2$ given in \cite{AZ16}, which is the core part of the proof of birational rigidity in \cite{AZ16}.
However, we have identified some insufficiencies in the arguments given in \cite{AZ16}, which are incorporated in several remarks in \S \ref{sec:pfbirrig}.
To overcome this, we have included a complete and uniform proof of \cite[Theorem~1.1]{AZ16} in \S \ref{sec:pfbirrig} to obtain Theorem~\ref{thm:BRcodim2} below.
Note that \cite{AZ16} introduced a different approach to proving rigidity.
Although that method could perhaps be extended to the missing cases, we depart from this approach and use the traditional methods of dealing with centers of maximal singularity to keep our arguments uniform.

\begin{Thm}[{$=$ Theorem~\ref{thm:birrigid}, cf.\  \cite[Theorem~1.1]{AZ16}}]
\label{thm:BRcodim2}
Let $X$ be a well-formed and quasismooth Fano $3$-fold WCI of codimension $2$ and index $1$ that belongs to one of the specific $18$ families listed in \emph{Table~\ref{table:BRcodim2}}.
Then, $X$ is birationally rigid.
\end{Thm}
 
We explain the structure of this paper briefly.
In \S \ref{sec:prelim}, we give basic definitions and recall some preliminary results that are necessary in this paper.
This includes some methods of excluding maximal centers (see \S \ref{sec:maxsing}), and some methods for computing local alpha invariants (see \S \ref{sec:methodalpha}) and local delta invariants (see \S \ref{sec:methoddelta}).
Moreover, some known results on the classification of maximal centers are encapsulated in \S \ref{sec:known}.
In \S \ref{sec:pfbirrig}, we give a complete proof of Theorem~\ref{thm:BRcodim2}.
The rest of the sections are devoted to the proof of Theorem~\ref{thm:Kst}.
In \S \ref{sec:alphasmpt} and \S \ref{sec:alphasingpt}, we give a lower bound $\alpha_{\msp} (X) \ge 1/2$ for smooth points $\msp$ and some singular points $\msp$, respectively.
Note that none of these points is a maximal center.
In \S \ref{sec:delta}, we prove the inequality $\delta_{\msp} (X) > 1$ for the singular points that are not treated in \S \ref{sec:alphasingpt}.
Theorem~\ref{thm:Kst} follows from the above results combined with Theorem~\ref{thm:localSZ}, which is the local version of the result of Stibitz--Zhuang.

\begin{Ack}
The author would like to thank Hamid Abban and Francesco Zucconi for discussions on birational rigidity of Fano $3$-fold WCIs of codimension $2$ and index $1$.
The author also would like to thank Yuuki Morita for his interest in this work and for informing the author about the reference \cite{Ishii}.
The author was partially supported by JSPS KAKENHI Grant Number JP24K00519.
\end{Ack}

\section{Preliminaries}
\label{sec:prelim}

Throughout the paper, by a \textit{$\mbQ$-Fano variety}, we mean a normal projective variety with only klt singularities and with ample anticanonical divisor.
By a \textit{Fano variety}, we mean a $\mbQ$-Fano variety in the Mori category, that is, it is a normal projective $\mbQ$-factorial variety with only terminal singularities and with ample anticanonical divisor.

\subsection{Basics on weighted projective varieties}
\label{sec:WCI}

Let $\mbP \coloneq \mbP (a_0, \dots, a_N)$ be a weighted projective space with homogeneous coordinates $x_0, \dots, x_N$ of weights $a_0, \dots, a_N$, respectively, and let $R \coloneq \mbC [x_0, \dots, x_N]$ be the graded polynomial ring with the grading given as above.
We may assume $\gcd \{a_0, \dots, a_N\} = 1$.

\begin{Def}
Let $X$ be a closed subscheme of $\mbP$ defined by a homogeneous ideal $I \subset R$.
We set
\[
C_X \coloneq \Spec (R/I)
\]
and call it the \textit{affine cone} of $X$.
Let $\pi \colon \mbA^{N+1} \setminus \{O\} \to \mbP$ be the natural projection, where $O$ is the origin.
For a point $\msp \in X$, we say that $X$ is \textit{quasismooth at $\msp$} if $C_X$ is smooth along $\pi^{-1} (\msp)$.
We say that $X$ is \textit{quasismooth} if it is quasismooth at any point of $X$, that is, $C_X$ is smooth away from the origin.
\end{Def}

\begin{Def}
We say that $\mbP$ is \textit{well-formed} if the greatest common divisor of any $N$ of the weights $a_0, \dots, a_N$ is $1$. 
We say that a subscheme $X \subset \mbP$ is \textit{well-formed} if $\mbP$ is well-formed and $\operatorname{codim}_X (X \cap \Sing (\mbP)) \ge 2$.
\end{Def}

\begin{Def}
Let $X \subset \mbP$ be a closed subscheme defined by a homogeneous ideal $I \subset R$ and let $f_1, \dots, f_m$ be generators of $I$.
We set
\[
J_X \coloneq \left(\frac{\prt f_i}{\prt x_j} \right)_{1 \le i \le m, 0 \le j \le N},
\] 
and call it the \textit{Jacobian matrix} of $X$ (with respect to $f_1, \dots, f_m)$.
\end{Def}

The matrix $J_X$ is usually thought of as the Jacobian matrix of $C_X$ and, for a point $\msq \in C_X$, $C_X$ is smooth at $\msq$ if and only if $\rank J_X (\msq) = N - n$, where $n = \dim X$.
For a point $\msp \in X$, it is clear that $\rank J_X (\msq)$ does not depend on the choice of a point $\msq \in \pi^{-1} (\msp)$, and thus we denote it by $\rank J_X (\msp)$.
Note that $X$ is quasismooth at $\msp$ if and only if $\rank J_X (\msp) = N - n$.

\begin{Def}
Let $X \subset \mbP$ be a subscheme.
For homogeneous polynomials $f_1, \dots, f_m \in R$, we define
\[
\begin{split}
(f_1 = \cdots = f_m = 0) &\coloneqq \Proj R/(f_1, \dots, f_m), \\
(f_1 = \cdots = f_m = 0)_X &\coloneqq (f_1 = \cdots = 0) \cap X,
\end{split}
\]
where the above intersection is scheme-theoretic.
For a homogeneous polynomial $f \in R$, we define
\[
(f \ne 0) \coloneq \mbP \setminus (f = 0) \quad \text{and} \quad (f \ne 0)_X  \coloneq (f \ne 0) \cap X.
\]
\end{Def}

\begin{Def}
Let $h \in \mbC [t_1, \dots, t_m]$ be a polynomial in variables $t_1, \dots, t_m$ and let $\phi = t_1^{i_1} \cdots t_m^{i_m}$ be a monomial.
We denote by $\coeff_h (\phi) \in \mbC$ the coefficient of the monomial $\phi$ in $h$.
We write $\phi \in h$ if $\coeff_h (\phi) \ne 0$. 
\end{Def}

\begin{Def}
Let $X \subset \mbP$ be a subscheme.
By a \textit{quasihyperplane} of $X$, we mean a hypersurface of $X$ cut by an equation $h (x_0, \dots, x_N) = 0$ where $h \in R$ is a homogeneous polynomial such that $x_i \in h$ for some $0 \le i \le N$.
\end{Def}

\begin{Lem}
\label{lem:normqhyp}
Let $X = (f_1 = \cdots = f_m = 0) \subset \mbP$ be a quasismooth and well-formed weighted complete intersection of dimension at least $3$ such that none of $f_i$ is quasilinear.
Then, any quasihyperplane of $X$ is normal.
\end{Lem}

\begin{proof}
Let $H$ be a quasihyperplane of $X$.
By \cite{Ishii}, we have $\dim \Sing (H) = 0$.
Note that in \cite{Ishii} the result is proved for a hypersurface in a smooth complete intersection in $\mbP^N$.
However the same proof works in our setting. 

The weighted projective space $\mbP$ is covered by a standard open subset $U_i \cong \mbA^{N+1}/\bmu_{a_i}$.
On the open subset $U_i$, the quasihyperplane $H \cap U_i$ is the quotient of a subscheme $\tilde{H}_i \subset \mbA^{N+1}$ that is a complete intersection in $\mbA^{N+1}$, hence $\tilde{H}_i$ is Cohen-Macaulay.
By the above argument, $\tilde{H}_i$ is regular in codimension $1$.
Thus $\tilde{H}_i$ is normal and so is its quotient $H \cap U_i$ for any $i$.
\end{proof}

We recall the notion of isolating sets.
Let $X$ be a normal variety embedded in $\mbP$ as a closed subscheme.

\begin{Def}
Let $\msp \in X$ be a point.
We say that a set of homogeneous polynomials $\{g_1, \dots, g_m\}$, where $g_i \in R$ for $1 \le i \le m$, \textit{isolates} $\msp$ if the point $\msp$ is an isolated component of $(g_1 = \cdots = g_m = 0)_X$.
In this case we also say that $\{g_1, \dots, g_m\}$ is a \textit{$\msp$-isolating set of maximum degree $d$}, where $d \coloneq \max \Set{ \deg (g_i) | 1 \le i \le m}$. 
\end{Def}

\begin{Def}
Let $\msp \in V$ be a germ of a cyclic quotient singularity and let $q_{\msp} \colon \breve{V} \to V$ be the corresponding quotient morphism.
For an effective $\mbQ$-divisor $D$ on $V$, we define
\[
\omult_{\msp} (D) \coloneq \mult_{\breve{\msp}} q_{\msp}^*D,
\]
where $\breve{\msp} \in \breve{V}$ is the preimage of $\msp$, and call it the \textit{orbifold multiplicity} of $D$ at $\msp$.
\end{Def} 

\begin{Lem}
\label{lem:isoltc}
Let $\msp \in X$ be a point, $Z_1, \dots, Z_k$ irreducible closed subsets of $X$ such that $\msp \in Z_i$ and $\dim Z_i > 0$ for any $i$, and let $g_1, \dots, g_m \in R$ be homogeneous polynomials.
Suppose that $X$ is quasismooth at $\msp$ and that $\{g_1, \dots, g_m\}$ isolates $\msp$.
We set $G_i \coloneq (g_i = 0)_X$, and we set
\[
\mu \coloneq \min \Set{ \frac{\omult_{\msp} (G_i)}{\deg g_i} | 1 \le i \le m}.
\]
Then, there exists an effective $\mbQ$-divisor $T \sim_{\mbQ} A$ such that $\omult_{\msp} (T) \ge \mu$ and $T$ does not contain any $Z_i$ in its support.
In particular, there exists an effective $\mbQ$-divisor $T' \sim_{\mbQ} d A$, where
\[
d \coloneq \max \Set{ \deg (g_i) | 1 \le i \le m},
\]
such that $\omult_{\msp} (T') \ge 1$ and $T'$ does not contain any $Z_i$ in its support.
\end{Lem}

\begin{proof}
The first assertion is \cite[Lemma~3.14]{KOWsuperrigid} (see also Remark~\ref{rem:KOWLem}).
The second assertion follows by simply taking $T' \coloneq d T$ since $\mu \ge 1/d$.
\end{proof}

\begin{Rem}
\label{rem:KOWLem}
The result \cite[Lemma~3.14]{KOWsuperrigid} does not hold as stated.
We need to assume that any $Z_i$ passes through the point $\msp$ as in Lemma~\ref{lem:isoltc}.
With this corrected assumption, the proof works without any change.
\end{Rem}

\begin{Lem}
\label{lem:isolomult}
Let $\msp \in X$ be a point and $D$ an effective $\mbQ$-divisor on $X$.
Suppose that $X$ is quasismooth at $\msp$, there is an effective $\mbQ$-divisor $S$ passing through $\msp$ such that $D$ and $S$ do not share any component, and that there is a $\msp$-isolating set of maximum degree $d$.
Then, 
\[
\omult_{\msp} (D) \le \omult_{\msp} (D) \omult_{\msp} (S) \le  r_{\msp} d (D \cdot S \cdot A),
\] 
where $r_{\msp}$ denotes the index of the cyclic quotient singularity $\msp \in X$.
\end{Lem}

\begin{proof}
We see that $D \cdot S$ is an effective $1$-cycle and write $D \cdot S = \Gamma + \Delta$, where any component of $\Gamma$ passes through $\msp$ and no component of $\Delta$ passes through $\msp$.
By Lemma~\ref{lem:isoltc}, we can take an effective $\mbQ$-divisor $T' \sim_{\mbQ} d A$ such that $\omult_{\msp} (T') \ge 1$ and $T'$ does not contain any component of $\Gamma$.
Then we have
\[
\omult_{\msp} (D) \omult_{\msp} (S) \le r_{\msp} (\Gamma \cdot T') = r_{\msp} ((D \cdot S \cdot T') - (\Delta \cdot T')) \le r_{\msp} (D \cdot S \cdot T')
\] 
since $T'$ is ample and $\Delta$ is effective.
The assertion follows since $T' \sim d A$.
\end{proof}

\subsection{Maximal extractions and centers}
\label{sec:maxsing}

Let $X$ be a Fano variety of Picard number $1$.
In this paper, by an \textit{extremal divisorial contraction} $\varphi \colon W \to V$, we mean the extremal divisorial contraction of a $K_W$-negative extremal ray in the Mori category, that is, $V$ and $W$ are both $\mbQ$-factorial and have only terminal singularities.

\begin{Def}
\label{def:maxsing}
For a linear system $\mcM$ on $X$, we denote by $n (\mcM)$ the positive rational number such that $\mcM \sim_{\mbQ} - n (\mcM) K_X$.
Let $\varphi \colon Y \to X$ be an extremal divisorial contraction.
We say that $\varphi$ is a \textit{Sarkisov extraction} if there is a Sarkisov link (in the Mori category) initiated by $\varphi$.
We say that $\varphi$ is a \textit{maximal extraction} if there is a movable linear system $\mcM$ such that $\ord_E (\mcM) > n (\mcM) a_X (E)$.
We say that $\varphi$ is a \textit{strong maximal extraction} if there is a movable linear system $\mcM$ on $X$ such that 
\[
\frac{1}{n (\mcM)} > c (X, \mcM) = \frac{a_X (E)}{\ord_E (\mcM)},
\]
where $c (X, \mcM)$ is the canonical threshold of the pair $(X, \mcM)$.

A subvariety $\Gamma \subset X$ is said to be a \textit{Sarkisov center} (resp.\ \textit{maximal center}, resp.\ \textit{strong maximal center}) if there is a Sarkisov extraction (resp.\ maximal extraction, resp.\ strong maximal extraction) $\varphi \colon Y \to X$ whose center is $\Gamma$.
We say that a subvariety $\Gamma \subset X$ is a \textit{weak maximal center} if there is a movable linear system $\mcM$ on $X$ such that $\Gamma$ is a center of non-canonical singularities of the pair $(X, \frac{1}{n (\mcM)} \mcM)$.
\end{Def}

\begin{Rem}
For a movable linear system $\mcM$ on $X$, the following implications are known (see \cite[Lemma~2.4]{Okada3mfs}):
\[
\begin{split}
\text{Strong maximal extraction} \ &\Longrightarrow \ \text{Sarkisov extraction} \\
\ &\Longrightarrow \ \text{Maximal extraction}.
\end{split}
\]
For a subvariety $\Gamma \subset X$, the following implications are known:
\[
\begin{split}
\text{Strong maximal center} \ &\Longrightarrow \ \text{Sarkisov center} \\
\ &\Longrightarrow \ \text{Maximal center} \\
\ &\overset{(*)}{\Longrightarrow} \ \text{Weak maximal center}
\end{split}
\]
The reverse implication of $(*)$ does not hold in general (see \cite[Remark~2.16]{KOP}).
However, the reverse implication of $(*)$ holds true if the center $\Gamma$ is a point (see \cite[Lemma~2.7]{KOP}).
\end{Rem}

We can characterize maximal extractions by the mobility of anticanonical divisor.

\begin{Lem}
\label{lem:exclmcmob}
Let $X$ be a Fano variety with Picard number $1$ and let $\varphi \colon Y \to X$ be an extremal divisorial contraction.
Then, $\varphi$ is a maximal extraction if and only if $-K_Y \in \Int \bMov (Y)$.
\end{Lem}

\begin{proof}
Let $E$ be the exceptional divisor of $\varphi$ and we set $a \coloneq a_X (E)$.

Suppose that $\varphi$ is a maximal extraction.
Then, there is a movable linear system $\mcM$ such that $m \coloneq \ord_E (\mcM) > n a$, where $n \coloneq n (\mcM)$.
Let $D \in \mcM$ be a general member.
Then $\varphi_*^{-1} D \sim_{\mbQ} - n \varphi^*K_X - m E$ is a movable divisor on $Y$ and thus
\[
- m K_Y \sim_{\mbQ} (m-na) (-\varphi^*K_X) + a \varphi_*^{-1} D \in \Int \bMov (Y)
\]
since $-\varphi^*K_X$, $\varphi_*^{-1} D$ are movable and not proportional to each other in $\Cl (Y)$, and $m - n a > 0$, $a > 0$.

Suppose that $-K_Y \in \Int \bMov (Y)$.
Then the linear system
\[
\mcM_Y \coloneq \left|n(-K_Y - \varepsilon E)\right| = \left|n\left(- \varphi^*K_X - (\varepsilon + a) E\right)\right|
\] 
is movable for a sufficiently small rational number $\varepsilon > 0$ and a sufficiently divisible integer $n > 0$.
We set $\mcM \coloneq \varphi_*\mcM_Y \subset \left|-nK_X\right|$.
The linear system $\mcM$ is clearly movable and, by the construction, we have $\ord_E \mcM = n (\varepsilon + a) > n a_E (X)$.
This shows that $\varphi$ is a maximal extraction for the pair $(X, \frac{1}{n} \mcM)$.
This proves the assertion.
\end{proof}

\begin{Rem}
For an extremal divisorial contraction $\varphi \colon Y \to X$, we have the following implications:
\[
\begin{split}
\text{$\varphi$ is a maximal extraction} \ &\Longleftrightarrow \  -K_Y \in \Int \bMov (Y)  \\
\ &\Longrightarrow \
(-K_Y)^2 \in \Int \bNE (Y).
\end{split}
\]
\end{Rem}

\begin{Rem}
Let $X$ be a Fano $3$-fold of Picard number $1$ and let $\msp \in X$ be a terminal cyclic quotient singularity.
Then, there is a unique extremal divisorial contraction $\varphi \colon Y \to X$ with center $\msp$, which is called the \textit{Kawamata blowup} at $\msp \in X$.
Thus, $\msp$ is a maximal center if and only if the Kawamata blowup is a maximal extraction. 
\end{Rem}

We recall several methods for excluding a given divisorial contraction as a maximal extraction and we slightly generalize some of them.

\begin{Lem}[{\cite[Lemma~2.16]{OkadaII}}]
\label{lem:exclNE}
Let $\varphi \colon Y \to X$ be an extremal divisorial contraction.
If $(-K_Y)^2 \notin \bNE (Y)$, then $\varphi$ is not a maximal extraction.
\end{Lem}

\begin{Lem}[{\cite[Lemma~2.18]{OkadaII}}]
\label{lem:exclnumprop}
Let $\varphi \colon Y \to X$ be an extremal divisorial contraction centered at a point $\msp \in X$ with exceptional divisor $E$.
Assume that there are surfaces $S$ and $T$ on $Y$ with the following properties.
\begin{enumerate}
\item $S \sim_{\mbQ} a B + d E$ and $T \sim_{\mbQ} b B + e E$ for some integers $a, b, d, e$ such that $a, b > 0$, $0 \le e < a_X (E) b$ and $ae-bd \ge 0$.
\item The intersection $\Gamma \coloneq S \cap T$ is a $1$-cycle whose support consists of irreducible and reduced curves which are numerically proportional to each other.
\item $(T \cdot \Gamma) \le 0$.
\end{enumerate}
Then, $\varphi$ is not a maximal extraction.
\end{Lem}

\begin{Lem}[{\cite[Lemma~2.19]{OkadaII}}]
\label{lem:exclnegdef}
Let $\varphi \colon Y \to X$ be an extremal divisorial contraction with exceptional divisor $E$.
Suppose that there is an effective divisor $S \sim_{\mbQ} b B + e E$ with $b > 0$ and $e \ge 0$ on $Y$ and a normal surface $T \ne E$ on $Y$ such that the support of the $1$-cycle $S|_T$ consists of curves on $T$ whose intersection form is negative-definite.
Then $\varphi$ is not a maximal extraction.
\end{Lem}

\begin{Lem}[{\cite[Lemma~2.20]{OkadaII}}]
\label{lem:excl:infinitecurve}
Let $\varphi \colon Y \to X$ be an extremal divisorial contraction with exceptional divisor $E$.
Suppose that there are infinitely many irreducible and reduced curves $C_{\lambda}$ on $Y$ such that $(-K_Y \cdot C_{\lambda}) \le 0$.
Then $\varphi$ is not a maximal extraction.
\end{Lem}

\begin{proof}
In \cite[Lemma~2.20]{OkadaII}, an additional assumption $(E \cdot C_{\lambda}) > 0$ is imposed.
We shall show that this assumption is redundant and automatically satisfied.
 
Let $C_{\lambda} \subset Y$ be as in the statement.
We can write $K_Y = \varphi^*K_X + a E$.
If $C_{\lambda}$ is contracted by $\varphi$, then $(E \cdot C_{\lambda}) < 0$ since $-E$ is $\varphi$-ample.
This is impossible since $0 \leq (K_Y \cdot C_{\lambda}) = a (E \cdot C_{\lambda})$.
Hence $C_{\lambda}$ is not contracted by $\varphi$ and this implies that $(- \varphi^*K_X \cdot C_{\lambda}) > 0$.
It follows that
\[
(E \cdot C_{\lambda}) = a^{-1} (K_Y \cdot C_{\lambda}) + a^{-1} (-\varphi^* K_X \cdot C_{\lambda}) > 0.
\]
Therefore, by \cite[Lemma~2.20]{OkadaII}, $\varphi$ is not a maximal extraction.
\end{proof}

\subsection{Methods for estimating local alpha invariant}
\label{sec:methodalpha}

We recall the definition of (local) alpha invariants and then explain simple methods for estimating local alpha invariants briefly.

Let $X$ be a $\mbQ$-Fano variety.
For a divisor $D$ on $X$, we denote by $|D|_{\mbQ}$ the set of effective $\mbQ$-divisors which are $\mbQ$-linearly equivalent to $D$.

\begin{Def}
For an effective $\mbQ$-divisor $D$ on $X$ and a point $\msp \in X$, we define the \textit{log canonical threshold} of $(X; D)$ and the \textit{log canonical threshold} of $(X; D)$ \textit{at} $\msp$ to be the numbers
\[
\begin{split}
\lct (X; D) &\coloneq \sup \Set{ c \in \mbQ_{\ge0} | \text{$(X, cD)$ is log canonical}}, \\
\lct_{\msp} (X; D) &\coloneq \sup \Set{ c \in \mbQ_{\ge 0} | \text{$(X, cD)$ is log canonical at $\msp$}},
\end{split}
\]
respectively.
The \textit{alpha invariant} of $X$ and the \textit{alpha invariant} of $X$ \textit{at} $\msp$ are defined as
\[
\begin{split}
\alpha (X) &\coloneq \sup \Set{ c \in \mbQ_{\ge 0} | \text{$(X, cD)$ is log canonical for any $D \in \left|-K_X\right|_{\mbQ}$}}, \\
&= \inf \Set{ \lct (X; D) | D \in \left| -K_X \right|_{\mbQ}}, \\
\alpha_{\msp} (X) &\coloneq \sup \Set{ c \in \mbQ_{\ge 0} | \text{$(X, cD)$ is log canonical at $\msp$ for any $D \in \left|-K_X\right|_{\mbQ}$}}, \\
&= \inf \Set{ \lct_{\msp} (X; D) | D \in \left|-K_X\right|_{\mbQ}}.
\end{split}
\]
\end{Def}

\begin{Rem}
\label{rem:alphairred}
For a point $\msp$ on a $\mbQ$-Fano variety $X$, we have an equality
\[
\alpha_{\msp} (X) =  \inf \Set{ \lct_{\msp} (X; D) | \text{$D \in \left|-K_X\right|_{\mbQ}$ is an irreducible $\mbQ$-divisor}},
\] 
where an \textit{irreducible $\mbQ$-divisor} means a $\mbQ$-divisor whose support is irreducible (see \cite[Remark~2.6]{KOWalpha}).
\end{Rem}

\begin{Rem}
\label{rem:lctomult}
The easiest way to obtain a lower bound for $\lct_{\msp} (X; D)$ is by (orbifold) multiplicity.
Suppose that $\msp \in X$ is either a smooth point or a cyclic quotient singular point.
Then we have
\[
\lct_{\msp} (X; D) \ge \frac{1}{\omult_{\msp} (D)}
\]
for any effective $\mbQ$-divisor $D$ (see \cite[Lemma~3.2]{KOWalpha}).
The combination of this simple method and Lemma~\ref{lem:isolomult} is quite useful in many cases.
\end{Rem}

\subsection{Methods for estimating local delta invariant}
\label{sec:methoddelta}

We give definitions of (local) delta invariants, and then explain several methods for estimating local delta invariants that we need in this paper.

\begin{Def}
Let $X$ be a $\mbQ$-Fano variety of dimension $n$.
For a prime divisor $E$ over $X$, we define
\[
S_X (E) \coloneq \frac{1}{(-K_X)^n} \int_0^{\infty} \vol (-f^*K_X - t E) d t,
\]
where $f \colon \tilde{X} \to X$ is a projective birational morphism such that $E$ is a prime divisor on $\tilde{X}$, and
\[
\delta_X (E) \coloneq \frac{A_X (E)}{S_X (E)},
\] 
where $A_X (E)$ is the log discrepancy of $X$ at $E$.
The \textit{global delta invariant} of $X$ is defined as
\[
\delta (X) \coloneq \inf \Set{ \delta_X (E) | \text{$E$ is a prime divisor over $X$}}.
\]
For a closed point $\msp \in X$, the \textit{local delta invariant} of $X$ at $\msp$ is defined as
\[
\delta_{\msp} (X) \coloneq \inf \Set{ \delta_X (E) | \text{$E$ is a prime divisor over $X$ such that $\msp \in C_X (E)$}}.
\]
\end{Def}

\begin{Thm}[{\cite[Corollary~2.10]{CO}}]
\label{thm:localSZ}
Let $X$ be a Fano variety of Picard number $1$ and let $\Sigma$ be a set of closed points of $X$ including all the zero dimensional maximal centers of $X$.
If $\delta_{\msp} (X) > 1$ for any closed point $\msp \in \Sigma$ and if $\alpha_{\msq} (X) \ge 1/2$ for any closed point $\msq \in X \setminus \Sigma$, then $X$ is K-stable.
\end{Thm}

Proving $\alpha_{\msp} (X) \ge 1/2$ is relatively easy, hence the key part is to prove $\delta_{\msp} (X) > 1$ for a closed point that is a maximal center.
This will be done by Abban--Zhuang method (for weighted complete intersections) that we will explain below.

Let $X$ be a well-formed and quasismooth weighted complete intersection of index $1$ and $A \coloneq -K_X$ the ample generator of $\Cl (X) \cong \mbZ$.
Let $\msp \in X$ be a point, which is either a smooth point or a terminal quotient singular point.
We denote by $r_{\msp}$ the index of the singularity $\msp \in X$.
Note that $\msp \in X$ is a smooth point if and only if $r_{\msp} = 1$.

\begin{Def}
\label{def:typeI}
A \textit{flag of type} $\mathrm{I}$ is a flag $\msp \in Z \subset Y \subset X$, where $Y$ is a normal hypersurface of $X$ and $Z$ is the scheme-theoretic intersection $Y \cap H$ for some hypersurface $H$ of $X$ such that $Z$ is irreducible and reduced, and the pairs $(X, Y)$ and $(Y, Z)$ are both plt in a neighborhood of $\msp$.
\end{Def}

Let $l_Y$ and $l_H$ be positive integers such that $Y \in |l_Y A|$ and $H \in |l_H A|$.

\begin{Prop}[{\cite[Proposition~3.2]{CO}}]
\label{prop:typeI}
With the notation and assumptions as above, let $\msp \in Z \subset Y \subset X$ be a flag of type $\mathrm{I}$.
Then,
\[
\delta_{\msp} (X) \ge
\min \left\{ 4 l_Y, \ 4 l_H, \ \frac{4}{r_{\msp} l_Y l_H (A^3)} \right\}.
\]
\end{Prop}

\begin{proof}
The definition of type $\mathrm{I}$ flag is the same as that in \cite[Definition~3.1]{CO}, with the following slight generalization that the pairs $(X, Y)$ and $(Y, Z)$ are only required to be locally plt around $\msp$.
The proof of this proposition follows the same arguments as those in \cite[Proposition~3.2]{CO} in view of \cite[Theorem~11.14]{Fujita23} that ensures that the Abban--Zhuang method works in this setting.
\end{proof}

\begin{Rem}
Let $\msp \in Z \subset Y \subset X$ be as in Definition~\ref{def:typeI}.
Then the condition that $(X, Y)$ and $(Y, Z)$ are both plt at $\msp$ follows if both $Y$ and $Z$ are quasismooth at $\msp$.
\end{Rem}

\begin{Def}
A flag of type $\mathrm{IIa}$ is a flag $\msp \in \Gamma \subset Y \subset X$ with the following properties: $Y$ is a normal hypersurface in $X$ such that the pair $(X, Y)$ is plt around $\msp$, $\Gamma$ is an irreducible and reduced curve on $Y$ such that the pair $(Y, \Gamma)$ is plt around $\msp$ and there is a hypersurface $H$ on $X$ such that $H|_Y = \Gamma + \Delta$, where $\Delta \ne \Gamma$ is an irreducible and reduced curve on $Y$ such that $(\Delta^2) < 0$.
\end{Def}

Let $l_Y$ and $l_H$ be positive integers such that $Y \in |l_Y A|$ and $H \in |l_H A|$.
We set
\[
\lambda \coloneq (\Gamma^2), \ \mu \coloneq - (\Delta^2), \ \nu \coloneq (\Gamma \cdot \Delta).
\]

\begin{Prop}[{\cite[Proposition~3.8]{CO}}]
\label{prop:typeII}
With the notation and assumptions as above, let $\msp \in Z \subset Y \subset X$ be a flag of type $\mathrm{IIa}$.
Then,
\[
\delta_{\msp} (X) \ge
\min \left\{ 4 l_Y, \ \frac{1}{S (V_{\bullet, \bullet}^Y; \Gamma)}, \ \frac{1}{r_{\msp} S (W_{\bullet, \bullet, \bullet}^{Y, \Gamma}; \msp)} \right\},
\]
where
\[
\begin{split}
S (V_{\bullet, \bullet}^Y; \Gamma) &= \frac{3}{(A^3)} \int_0^{\frac{1}{l_Y}} \left(\int_0^{\frac{\nu - \mu}{\nu} t (u)} \left( (2 \nu + \lambda - \mu) t (u)^2 - 2 (\nu + \lambda) t (u) v + \lambda v^2 \right) d v \right. \\
& \hspace{3cm} \left.+ \int_{\frac{\nu - \mu}{\nu} t (u)}^{t (u)} \frac{\nu^2 + \lambda \mu}{\mu} (t (u) - v)^2 d v \right) d u, \\
S (W_{\bullet, \bullet, \bullet}^{Y, \Gamma}; \msp) &= \frac{3}{(A^3)} \int_0^{l_Y} \left( \int_0^{\frac{\nu-\mu}{\nu} t (u)} ((\nu+\lambda) t(u) - \lambda v)^2 d v  \right. \\
& \left. \hspace{2cm} + \int_{\frac{\nu-\mu}{\nu} t (u)}^{t(u)} \left( \frac{\nu^2 + \lambda \mu}{\mu} \right)^2 (t(u) - v)^2 d v \right) d u + F_{\msp} (W_{\bullet, \bullet, \bullet}^{Y, \Gamma}),
\end{split}
\]
with $t (u) = (1 - l_Y u)/l_H$ and
\[
F_{\msp} (W_{\bullet, \bullet, \bullet}^{Y, \Gamma}) = \frac{1}{4 l_Y l_H^3 (A^3)} \cdot \frac{\mu(\nu^2 + \lambda \mu)}{\nu^2} \cdot \ord_{\msp} (\Delta|_{\Gamma}).
\]
\end{Prop}

\begin{proof}
The definition of type $\mathrm{IIa}$ flag is the same as that in \cite[Definition~3.1]{CO}, with the following slight generalization that the pairs $(X, Y)$ and $(Y, \Gamma)$ are only required to be locally plt around $\msp$.
The proof of this proposition follows the same arguments as those in \cite[Proposition~3.8]{CO} in view of \cite[Theorem~11.14]{Fujita23} that ensures that the Abban--Zhuang method works in this setting.
\end{proof}

\subsection{Known results on Fano $3$-fold WCIs of codimension $2$}
\label{sec:known}

We explain some known results on the classification of maximal centers and computations of alpha invariants of well-formed and quasismooth Fano $3$-fold WCIs of codimension $2$ and index $1$.

We set 
\[
\begin{split}
\msI^*_{\br} &\coloneq \{8, 14, 20, 24, 31, 37, 45, 47, 51, 59, 60, 64, 71, 75, 78, 84, 85\}, \\
\msI_{\br} &\coloneq \msI^*_{\br} \cup \{1\},
\end{split}
\]
where family \textnumero~$1$ is the family of complete intersections of type $(2, 3)$ in $\mbP^5$.
It is proved in \cite{IP96} (resp.\ \cite{OkadaI}) that general members of family \textnumero~$1$ (resp.\ families indexed by $\msI^*_{\br}$) are birationally rigid.
Moreover, it is also proved in \cite{OkadaI} that no member belonging to the other $85 - 19 = 66$ families is birationally rigid.

\begin{Rem}
\label{rem:knownBR}
We explain some known results on the classification of maximal centers on members of the families parametrized by $\msI^*_{\br}$.

The following results are proved in \cite{OkadaI} without any generality assumptions:
\begin{itemize}
\item For any member of each of the families indexed by $\msI^*_{\br}$, no curve is a maximal center: This is proved in \cite[Corollary~6.2]{OkadaI}.
\item For any member of each of the families indexed by $\msI^*_{\br}$, no smooth point is a maximal center: This is proved in \cite[Proposition~7.6]{OkadaI} if $\msi = 8$ and in \cite[Proposition~7.5]{OkadaI} otherwise.
\item For any member of each of the families indexed by $\msI^*_{\br}$ and for its singular point without the mark $\heartsuit$ in the fourth column of Table~\ref{table:BRcodim2}, the singular point is either excluded as a maximal center or the Kawamata blowup initiates a Sarkisov-self link: This is proved in \cite[Theorem~5.5, Propositions~8.4 and 8.11]{OkadaI} except for some special members of family \textnumero~71 and the singular point of type $\frac{1}{5} (1, 2, 3)$.
This special case is also covered (see \cite[Remark~8.2 and Example~8.23]{OkadaI}).
\end{itemize}
As a consequence, it is proved in \cite{OkadaI} that any well-formed and quasismooth member of family \textnumero~$\msi$, where 
\[
\msi \in \{8, 14, 24, 45, 60, 64, 75, 76, 78, 84, 85\} \subset \msI_{\br}^*,
\] 
is birationally rigid since there is no singular point with the mark $\heartsuit$ for members of these $11$ families.

The following result is proved under some generality assumptions:
\begin{itemize}
\item For a member of each of the families indexed by $\msI^*_{\br}$ and for its singular point with the mark $\heartsuit$ in the fourth column of Table~\ref{table:BRcodim2}, the singular point is either excluded as a maximal center or the Kawamata blowup initiates a Sarkisov-self link under some generality assumptions that are made explicit: The generality assumption is indicated in the fourth column of the tables in \cite[Section~9]{OkadaI} (see also \cite[Condition~3.1 and \S 9]{OkadaI}).
\end{itemize}
Therefore, to complete the proof of Theorem~\ref{thm:BRcodim2}, we need to consider families \textnumero~$\msi$, where
\[
\msi \in \{20, 31, 37, 47, 51, 59, 71\} \subset \msI^*_{\br},
\]
and the singular point $\msp \in X$ with the mark $\heartsuit$ in Table~\ref{table:BRcodim2}, and need to show that either $\msp$ is not a maximal center or there is a Sarkisov-self link initiated by the Kawamata blowup of $X$ at $\msp$ in the case when the generality assumptions are not satisfied.
\end{Rem}

\begin{Rem}
\label{rem:KstNo60}
Let $X$ be a member of family \textnumero~$60$.
Then, it is proved that $\alpha_{\msp} (X) \ge 1$ in \cite[Theorem~1.5]{KOWalpha}.
Note that this result holds true for any quasismooth member $X$ of this family since the generality assumption \cite[Conditions~2.14 and 2.15]{KOWalpha} is vacuous (except for quasismoothness of $X$) for family \textnumero~$60$.
\end{Rem}

We fix some notation on Fano $3$-fold WCIs of codimension $2$ and index $1$.
Let $X = X_{d_1, d_2} \subset \mbP (a_0, \dots, a_5)$ be a well-formed and quasismooth Fano $3$-fold weighted complete intersection of index $1$ defined by two homogeneous polynomials of degree $d_1$ and $d_2$.
Unless otherwise specified, we assume that $a_0 \le \dots \le a_5$ and we use $x, y, z, s, t$ and $w$ as the homogeneous coordinates of weight $a_0, a_1, a_2, a_3, a_4$ and $a_5$, respectively.
\begin{itemize}
\item We denote by $\msF_1$ and $\msF_2$ the homogeneous polynomials of degree $d_1$ and $d_2$ in variables $x, y, z, s, t, w$, respectively, that define $X$.
\item For a variable $v \in \{x, y, z, s, t, w\}$, we set $H_v \coloneq (v = 0)_X$, $U_v \coloneq X \setminus H_v = (v \ne 0)_X$ and let $\msp_v$ be the point of $\mbP (a_0, \dots, a_5)$ at which only the coordinate $v$ does not vanish: for example, $\msp_x = (1\!:\!0\!:\!0\!:\!0\!:\!0\!:\!0)$.
\item A \textit{weighted complete intersection curve} (\textit{WCI curve}, for short) of type $(c_1, c_2, c_3, c_4)$ in $\mbP (a_0, \dots, a_5)$ is an irreducible and reduced curve defined by four homogeneous polynomials of degree $c_1, c_2, c_3$ and $c_4$.
\item For a subscheme $S$ of $X$ and a subset $\Gamma \subset S$, we define $\Sing_{\Gamma} (S) \coloneq \Gamma \cap \Sing (S)$.
\item We set $A \coloneq -K_X$, which is the positive generator of $\Cl (X) \cong \mbZ$.
\item Suppose that we are given a Kawamata blowup $\varphi \colon Y \to X$ at a singular point $\msp$.
In this case, we set $B \coloneq -K_Y \sim \varphi^*A - \frac{1}{r} E$, where $E$ is the exceptional divisor of $\varphi$ and $r$ is the index of the cyclic quotient singularity $\msp \in X$.
For a curve or a divisor $\Xi$ on $X$, we set $\tilde{\Xi} \coloneq \varphi_*^{-1} \Xi$.
Note that if $\msp \in X$ is of type $\frac{1}{r} (1, a, r-a)$, where $1 \le a < r$ is coprime to $r$, then $\varphi$ is the weighted blowup with weight $\frac{1}{r} (1, a, r-a)$ and we have
\[
(E^3) = \frac{r^2}{a (r-a)}.
\]
For a homogeneous polynomial $h \in \mbC [x, y, z, s, t, w]$, we define $\ord_E (h) \coloneq \ord_E (\varphi^*D)$, where $D \coloneq (h = 0)_X$.
\end{itemize}

\section{Proof of birational rigidity}
\label{sec:pfbirrig}

The aim of this section is to give a complete and uniform proof of Theorem~\ref{thm:BRcodim2}, that is, we prove that any well-formed and quasismooth member of $18$ families indexed by $\msI^*_{\br}$ is birationally rigid.

In the following subsections, we consider a member $X$ of family \textnumero~$\msi$ with
\[
\msi \in \{20, 31, 37, 47, 51, 59\} \subset \msI_{\br}^*
\]
and its singular point $\msp \in X$ with the mark $\heartsuit$ in Table~\ref{table:BRcodim2}.
The aim of this section is to remove generality assumptions and prove that either $\msp$ is not a maximal center or there is a Sarkisov self-link initiated by the Kawamata blow-up of $\msp \in X$.
Combining with the earlier work \cite{OkadaI}, this proves the following (see Remark~\ref{rem:knownBR}).

\begin{Thm}
\label{thm:birrigid}
Let $X$ be a well-formed and quasismooth Fano $3$-fold weighted complete intersection that is a member of family \textnumero~$\msi$, where $\msi \in \msI_{\br}^*$.
Then $X$ is birationally rigid.
\end{Thm}

This follows from Propositions~\ref{prop:No20excl2C1}, \ref{prop:No20excl2C2}, \ref{prop:No31excl3C1}, \ref{prop:No31excl3C2}, \ref{prop:No31excl3C3}, \ref{prop:No31excl3C4}, \ref{prop:No37excl2C1}, \ref{prop:No37excl2C2}, \ref{prop:No47excl3C1}, \ref{prop:No47excl3C2}, \ref{prop:No51excl4}, \ref{prop:No59excl4} and \ref{prop:No71excl4} below.

Throughout the following subsections, $\varphi \colon Y \to X$ is the Kawamata blow-up of $X$ at the point $\msp$ with exceptional divisor $E$.

\subsection{Family \textnumero~20 and $\frac{1}{2} (1, 1, 1)$ points}

Let $X = X_{6, 8} \subset \mbP (1, 2, 2, 3, 3, 4)$ be a member of family \textnumero~20 and let $\msp \in X$ be a $\frac{1}{2} (1, 1, 1)$ point.
The generality assumption in \cite{OkadaI} imposed for $\msp \in X$ is that the base locus of the pencil $|\mfm_{\msp} (2 A)|$ is irreducible. 
We assume that $\Bs |\mfm_{\msp} (2A)|$ is not irreducible. 
Without loss of generality, we may assume that $\msp = \msp_z$ by replacing coordinates.
Then, $|\mfm_{\msp} (2A)|$ is generated by $x^2$ and $y$.

\subsubsection{Case: $z w \in \msF_1$}
\label{sec:No20excl2C1}

We assume that $z w \in \msF_1$.

\begin{Lem}
\label{lem:No20excl2C1eq}
The defining polynomials of $X$ can be written as 
\[
\begin{split}
\msF_1 &= z w + s t + w f_2 + f_6, \\
\msF_2 &= z^3 y + z^2 (\delta x s + \varepsilon x t + g_4) + z (\alpha s^2 + \beta t^2 + \gamma s t + w g_2 + g_6) \\
&\hspace{6cm} + w^2 + w h_4 + g_8,
\end{split}
\]
where $\alpha, \beta, \gamma, \varepsilon \in \mbC$, $f_2, f_6, g_2, g_4 \in \mbC [x, y]$ and $g_6, g_8, h_4 \in (x, y) \subset \mbC [x, y, s, t]$ with $g_6, g_8 \in (x^2, y)$.
\end{Lem}

\begin{proof}
We have $z^3 y \in \msF_2$ and $w^2 \in \msF_2$ by the quasismoothness of $X$.
Possibly rescaling coordinates, we may assume $\coeff_{\msF_1} (z w) = \coeff_{\msF_2} (z^3 y) = \coeff_{\msF_2} (w^2) = 1$.
Replacing $\msF_2$ by $\msF_2 - \coeff_{\msF_2} (z^2 w) z \msF_1$, we may assume $z^2 w \notin \msF_2$.
Then, we can write
\[
\begin{split}
\msF_1 &= z w + w f_2 + f'_6, \\
\msF_2 &= z^3 y + z^2 g'_4 + z (w g_2 + g'_6) + w^2 + w h_4 + g_8,
\end{split}
\]
where $f_2, f'_6, g_2, g'_4, g'_6, h'_4 \in \mbC [x, y, s, t]$.
It is clear that $f_2$ and $g_2$ do not involve variables $s$ and $t$, that is, $f_2, g_2 \in \mbC [x, y]$.
We can write $g'_4 = \delta x s + \varepsilon x t + g_4$, where $g_4 \in \mbC [x, y]$.
By the quasismoothness of $X$, $f'_6 (0, 0, s, t)$ is a quadratic form in $s, t$ of rank $2$, and hence we may assume $f'_6 = s t + f_6$, where $f_6 \in \mbC [x, y]$, by replacing $s$ and $t$ suitably.
We can write $g'_6 = \alpha s^2 + \beta t^2 + \gamma s t + g_6$ with $g_6 \in (x, y) \subset \mbC [x, y, s, t]$.
We see that $h_4 \in (x, y)$ (resp.\ $g_8 \in (x^2, y)$) since there are no monomials of degree $4$ (resp.\ degree $7$ and $8$) in variables $s, t$.
There are no monomials of degree $5$ in variables $s, t$, the degree $6$ monomials in variables $s, t$ are $s^2, t^2, s t$ and $g_6$ does not contain any of them.
It follows that $g_6 \in (x^2, y) \subset \mbC [x, y, s, t]$ and we obtain the desired form of defining polynomials.
\end{proof}

We choose homogeneous coordinates as in Lemma~\ref{lem:No20excl2C1eq}.
We set $S \coloneq (x = 0)_X$ and let $T \in |\mfm_{\msp} (2 A)|$ be a general member.

\begin{Lem}
\label{lem:No20excl2C1sing}
The following assertions hold.
\begin{enumerate}
\item We have $\alpha \beta = 0$.
\item If $\alpha = \gamma = 0$, then $\delta \ne 0$.
\item If $\beta = \gamma = 0$, then $\varepsilon \ne 0$.
\end{enumerate}
\end{Lem}

\begin{proof}
We have
\[
S \cap T = (x = y = 0)_X = (z w + s t = z (\alpha s^2 + \beta t^2 + \gamma s t) + w^2 = 0).
\]
If $\alpha \beta \ne 0$, then $S \cap T$ is an irreducible and reduced curve.
By assumption, we have $\alpha \beta = 0$ and the assertion (1) is proved.

Suppose that $\alpha = \gamma = 0$.
We have $y s^2 \in g_8$ by the quasismoothness of $X$ at $\msp_s$ and set $\theta \coloneq \coeff_{g_8} (y s^2) \ne 0$.
We see that $\delta \ne 0$ because otherwise $X$ is not quasismooth at $(0\!:\!0\!:\!\sqrt[3]{-\theta}\!:\!1\!:\!0\!:\!0) \in X$ which is a contradiction.
This proves the assertion (2).
The assertion (3) follows similarly.
\end{proof}

We can choose $x, s, t$ as local orbifold coordinates of $X$ at $\msp$ and $\ord_E (x, s, t) = \frac{1}{2} (1, 1, 1)$. 
By the equations $\msF_1 = \msF_2 = 0$, we have $\ord_E (y, w) = \frac{1}{2} (2, 2)$ since at least one of $z s^2, z t^2, z s t, z^2 s x, z^2 t x$ appear in $\msF_2$ with nonzero coefficient by Lemma~\ref{lem:No20excl2C1sing}. 
We have $\tilde{S} \sim \varphi^*A - \frac{1}{2} E = B$ and $\tilde{T} \sim 2 \varphi^*A - E = 2 B$.

\begin{Prop}
\label{prop:No20excl2C1}
The following assertions hold.
\begin{enumerate}
\item If exactly one of $\alpha, \beta$ is zero and $\gamma \ne 0$, then there are infinitely many irreducible and reduced curves that intersect $B$ nonpositively.
\item Otherwise, the support of the $1$-cycle $\tilde{S}|_{\tilde{T}}$ consists of curves whose intersection matrix is negative-definite.
\end{enumerate}
In particular, $\msp$ is not a maximal center. 
\end{Prop}

\begin{proof}
Note first that $T$ is a normal surface by Lemma~\ref{lem:normqhyp}.
We have
\[
S \cap T = (x = y = 0)_X = (z w + s t = z (\alpha s^2 + \beta t^2 + \gamma s t) + w^2 = 0).
\]
Without loss of generality, we may assume $\alpha = 0$.
In this case $S \cap T$ contains $\Gamma \coloneq (x = y = t = w = 0)$.

By setting $z = 1$ and by eliminating $w = - s t$ in terms of the equation $\msF_1 (0, 0, z, s, t, w) = 0$, we see that $S \cap T \cap U_z$ is isomorphic to the quotient of the plane curve $C \subset \mbA^2_{s, t}$ defined by the equations
\[
\msF_2 (0, 0, 1, s, t, - s t) = t (\beta t + \gamma s + s^2 t) = 0.
\]
Any curve contained in $S \cap T$ appears as a component of this affine curve since $S \cap T \cap (z = 0) = \{\msp_s, \msp_t\}$.
The curve $\Gamma \cap U_z$ corresponds to the quotient of the curve $(t = 0) \subset \mbA^2_{s, t}$.
From this, we obtain the following.

\begin{enumerate}
\item[(i)] If $\beta \ne 0$ and $\gamma \ne 0$, then $\beta t + \gamma s + s^2 t = 0$ is an irreducible and reduced curve, and hence 
\[
S|_T = \Gamma + \Xi,
\] 
where $\Xi$ is the component of $S \cap T$ other than $\Gamma$.
Note that $\Xi$ is an irreducible and reduced curve.
Note also that $\Gamma$ and $\Xi$ intersect at $\msp_s$ since $S \cap T$ is not quasismooth at $\msp_s$ and $\Gamma$ is quasismooth at $\msp_s$.
\item[(ii)] If $\beta \ne 0$ and $\gamma = 0$, then $\beta t + \gamma s +  s^2 t = t (\beta + s^2) = 0$, and hence 
\[
S|_T = 2 \Gamma + \Xi,
\] 
where $\Xi$ is the component of $S \cap T$ other than $\Gamma$.
Note that $\Xi$ is an irreducible and reduced curve since $S \cap T \cap U_t$ is an irreducible and reduced curve and $S \cap T \cap (t = 0) = \Gamma$, set-theoretically.
Note also that $\Gamma$ intersects $\Xi$ at $(0\!:\!0\!:\!1\!:\!\sqrt{-\beta}\!:\!0\!:\!0)$.
\item[(iii)] If $\beta = 0$ and $\gamma \ne 0$, then $S|_T$ is the sum of three distinct irreducible and reduced curves, but we do not make use of this fact (since their intersection matrix is only negative-semidefinite).
\item[(iv)] If $\beta = \gamma = 0$, then $\beta t + \gamma s + s^2 t = s^2 t$, and hence 
\[
S|_T = 2 \Gamma + 2 \Delta,
\] 
where $\Delta \coloneq (x = y = s = w = 0)$.
\end{enumerate}

In any case, we have $(A \cdot \Gamma) = 1/6$ and the curve $\tilde{\Gamma}$ intersects $E$ transversally at $1$ smooth point of $Y$, and hence $(E \cdot \tilde{\Gamma}) = 1$.
We have 
\[
(B \cdot \tilde{\Gamma}) = (A \cdot \Gamma) - \frac{1}{2} (E \cdot \tilde{\Gamma}) = - \frac{1}{3}.
\]

(i) Suppose that $\beta \ne 0$ and $\gamma \ne 0$, and let $S|_T = \Gamma + \Xi$ be as in (i).
We have $(A \cdot \Xi) = (A \cdot S \cdot T) - (A \cdot \Gamma) = 1/2$. 
We have $\tilde{S}|_{\tilde{T}} = \tilde{\Gamma} + \tilde{\Xi}$ since $\tilde{S} \cap \tilde{T} \cap E$ does not contain a curve.
Then we obtain
\[
(E \cdot \tilde{\Xi}) = (E \cdot \tilde{S} \cdot \tilde{T}) - (E \cdot \tilde{\Gamma}) = \frac{1}{2} (E^3) - 1 = 1. 
\]
This shows that $(B \cdot \tilde{\Xi}) = 0$.
By the observation given in (i), we have $m \coloneq (\tilde{\Gamma} \cdot \tilde{\Xi}) > 0$ since $\Gamma$ intersects $\Xi$ at a point other than $\msp$.
By taking intersection numbers of $\tilde{S}|_{\tilde{T}} = \tilde{\Gamma} + \tilde{\Xi}$ and $\tilde{\Gamma}$, $\tilde{\Xi}$, we have
\[
(\tilde{\Gamma}^2) = -m - \frac{1}{3}, \quad
(\tilde{\Xi}^2) = - m.
\]
It follows that
\[
\begin{pmatrix}
(\tilde{\Gamma}^2) & (\tilde{\Gamma} \cdot \tilde{\Xi}) \\
(\tilde{\Gamma} \cdot \tilde{\Xi}) & (\tilde{\Xi}^2)
\end{pmatrix} =
\begin{pmatrix}
- m - \frac{1}{3} & m \\
m & - m
\end{pmatrix}
\]
and it is negative-definite since $m > 0$.

(ii) Suppose that $\beta \ne 0$ and $\gamma = 0$.
Let $S|_T = 2 \Gamma + \Xi$ be as in (ii).
We have $(A \cdot \Xi) = (A \cdot S \cdot T) - 2 (A \cdot \Gamma) = 1/3$. 
We have $\tilde{S}|_{\tilde{T}} = 2 \tilde{\Gamma} + \tilde{\Xi}$ since $\tilde{S} \cap \tilde{T} \cap E$ does not contain a curve.
Then we obtain
\[
(E \cdot \tilde{\Xi}) = (E \cdot \tilde{S} \cdot \tilde{T}) - 2 (E \cdot \tilde{\Gamma}) = \frac{1}{2} (E^3) - 2 = 0. 
\]
This shows $(B \cdot \tilde{\Xi}) = 1/3$.
By the observation given in (ii), the curves $\Gamma$ and $\Xi$ have an intersection point $\msp' = (0\!:\!0\!:\!1\!:\!\sqrt{-\beta}\!:\!0\!:\!0) \ne \msp$.
It follows that $m \coloneq (\tilde{\Gamma} \cdot \tilde{\Xi}) \ge 1$. 
By taking intersection numbers of $\tilde{S}|_{\tilde{T}} = 2 \tilde{\Gamma} + \tilde{\Xi}$ and $\tilde{\Gamma}$, $\tilde{\Xi}$, we have
\[
(\tilde{\Gamma}^2) = - \frac{1}{2} m - \frac{1}{6}, \quad
(\tilde{\Xi}^2) = \frac{1}{3} - 2 m.
\]
It follows that
\[
\begin{pmatrix}
(\tilde{\Gamma}^2) & (\tilde{\Gamma} \cdot \tilde{\Xi}) \\
(\tilde{\Gamma} \cdot \tilde{\Xi}) & (\tilde{\Xi}^2)
\end{pmatrix} =
\begin{pmatrix}
- \frac{m}{2} - \frac{1}{6} & m \\
m & \frac{1}{3} - 2 m
\end{pmatrix}
\]
and it is negative-definite since $m \ge 1$.

(iii) Suppose that $\beta = 0$ and $\gamma \ne 0$.
We set $\msF'_2 \coloneq \msF_2 - \gamma z \msF_1$.
We have
\[
\msF'_2 = z^2 (z y - \gamma w + \delta x s + \varepsilon x t + g_4) + z (w g_2 + g_6 - \gamma w f_2 - \gamma f_6) + w^2 + w h_4 + g_8.
\]
We set $\xi \coloneq z y - \gamma w + \delta x s + \varepsilon x t + g_4$.
By the equation $\msF'_2 = 0$, we have $\ord_E (\xi) = \frac{4}{2}$.
We set $S_{\lambda} \coloneq (\xi - \lambda x^4 = 0)_X$ and $T_{\mu} \coloneq (y - \mu x^2 = 0)_X$.
The equations $y = \mu x^2$ and $\xi = \lambda x^4$ are the same as 
\begin{equation}
\label{eq:No20excl2C1-1}
\begin{cases}
y = \mu x^2 , \\
w = \gamma^{-1} (- \lambda x^4 + \mu z x^2 + \delta x s + \varepsilon x t + g_4 (x, \mu x^2, s, t)).
\end{cases}
\end{equation}
We see that $\bar{g}_4 \coloneq g_4 (x, \mu x^2, s, t)$ is divisible by $x$ and we set 
\[
\zeta \coloneq \gamma^{-1} (-\lambda x^3 + \mu z x + \delta s + \varepsilon t + \bar{g}_4/x) \in \mbC [x, s, t],
\]
so that the second equation in \eqref{eq:No20excl2C1-1} becomes $w = x \zeta$.
For a polynomial $\phi = \phi (x, y, s, t, w)$, we set 
\[
\bar{\phi} \coloneq \phi (x, \mu x^2, s, t, x \zeta) \in \mbC [x, s, t].
\]
Then,
\[
\overline{\msF'_2} = \lambda x^4 z^2 + z (x \zeta (\bar{g}_2 - \gamma \bar{f}_2) + \bar{g}_6 - \gamma \bar{f}_6) + x^2 \zeta^2 + x \zeta \bar{h}_4 + \bar{g}_8.
\]
Note that $\overline{\msF'_2}$ is divisible by $x^2$ by Lemma~\ref{lem:No20excl2C1eq}, but it is not divisible by $x^3$ since $s^2 y, t^2 y \in g_8$ by the quasismoothness of $X$ at $\msp_s$ and $\msp_t$.
We write $\overline{\msF'_2} = x^2 \msG$.
Then, by eliminating the variables $y = \mu x^2$, $w = x \zeta$,  the scheme $S_{\lambda} \cap T_{\mu}$ is isomorphic to the complete intersection
\[
(x \zeta (z + \bar{f}_2) + s t + \bar{f}_6 = x^2 \msG_2 = 0) \subset \mbP (1_x, 2_z, 3_s, 3_t),
\]
where 
\[
\overline{\msF}_1 = x \zeta (z + \bar{f}_2) + s t + \bar{f}_6.
\]
The component $(\overline{\msF}_1 = x = 0)$ is the union $\Gamma \cup \Delta$, where $\Delta \coloneq (x = y = s = w = 0)$, and neither $\Gamma$ nor $\Delta$ is contained in the other component $(\overline{F}_1 = \msG = 0)$ since $(x = \overline{F}_1 = \msG = 0)$ is a finite set of points.
Then, we can write $S_{\lambda} \cdot T_{\mu} = 2 \Gamma + 2 \Delta + \Xi_{\lambda, \mu}$, where $\Xi_{\lambda, \mu}$ is an effective $1$-cycle on $X$ such that $\Gamma, \Delta \not\subset \Supp (\Xi_{\lambda, \mu})$.
We have $\tilde{S}_{\lambda} \sim 4 \varphi^*A - \frac{4}{2} E = 4 B$ and $\tilde{T}_{\mu} \sim 2 \varphi^*A - \frac{2}{2} E = 2 B$.
We have $(A \cdot \Delta) = 1/6$ and $(E \cdot \tilde{\Delta}) = 1$.
Hence $(B \cdot \tilde{\Delta}) = -1/3$ and 
\[
(B \cdot \tilde{\Xi}_{\lambda, \mu}) = (B \cdot \tilde{S}_{\lambda} \cdot \tilde{T}_{\mu}) - 2 (B \cdot \tilde{\Gamma}) - 2 (B \cdot \tilde{\Delta}) = -\frac{4}{3} + \frac{2}{3} + \frac{2}{3} = 0.
\]
Thus, there is at least one irreducible component of $\Xi_{\lambda, \mu}$ that intersects $B$ nonpositively.
For distinct pairs $(\lambda, \mu), (\lambda', \mu') \in \mbC^2$, we have $(S_{\lambda} \cap T_{\mu}) \cap (S_{\lambda'} \cap T_{\mu'}) = (x = y = \xi = 0)_X = \Gamma \cup \Delta$ set-theoretically, and hence there are infinitely many irreducible and reduced curves on $Y$ that intersect $B$ nonpositively.

(iv) Suppose that $\beta = \gamma = 0$.
Let $S|_T = 2 \Gamma + 2 \Delta$ be as in (iv).
We have $(A \cdot \Delta) = 1/6$ and $(E \cdot \tilde{\Delta}) = 1$.
Hence, $(B \cdot \tilde{\Delta}) = -1/3$.
We have $\Gamma \cap \Delta = \{\msp\}$ and the intersection points $E \cap \tilde{\Gamma}$ and $E \cap \tilde{\Delta}$ is distinct.
It follows that $(\tilde{\Gamma} \cdot \tilde{\Delta}) = 0$.
By taking intersection numbers of $B|_{\tilde{T}} \sim \tilde{S}|_{\tilde{T}} = 2 \tilde{\Gamma} + 2 \tilde{\Delta}$, we have
\[
\begin{pmatrix}
(\tilde{\Gamma}^2) & (\tilde{\Gamma} \cdot \tilde{\Delta}) \\
(\tilde{\Gamma} \cdot \tilde{\Delta}) & (\tilde{\Delta}^2)
\end{pmatrix} =
\begin{pmatrix}
- \frac{1}{6} & 0 \\
0 & - \frac{1}{6}
\end{pmatrix},
\]
which is negative-definite.

By Lemmas~\ref{lem:exclnegdef} and \ref{lem:excl:infinitecurve}, $\msp$ is not a maximal center.
\end{proof}

\subsubsection{Case: $z w \notin \msF_1$}
\label{sec:No20excl2C2}

We assume that $z w \notin \msF_1$.

\begin{Lem}
\label{lem:No20excl2C2eq}
The defining polynomials of $X$ can be written as
\[
\begin{split}
\msF_1 &= z^2 y + z f_4 + s t + f_6, \\
\msF_2 &= z^2 w + z (\alpha s^2 + \beta t^2 + g_6) + w^2 + g_8,
\end{split}
\]
where $\alpha, \beta \in \mbC$ and $f_i, g_i \in (x, y) \subset \mbC [x, y, s, t, w]$.
\end{Lem}

\begin{proof}
We have $z^2 y \in \msF_1$, $z^2 w \in \msF_2$ and $w^2 \in \msF_2$ by the quasismoothness of $X$.
By replacing coordinates, the defining polynomials of $X$ can be written as
\[
\begin{split}
\msF_1 &= z^2 y + z f_4 + f'_6, \\
\msF_2 &= z^2 w + z  g'_6 + w^2 + g_8,
\end{split}
\]
where $f_4, f'_6, g'_6, g_8 \in \mbC [x, y, s, t, w]$ with $w^2 \notin g_8$.
Note that $f_4, g_8 \in (x, y)$ since the degree $4$ and $8$ monomials in variables $s, t, w$ are $w$ and $w^2$, respectively, and $w \notin f_4$ and $w^2 \notin g_8$.
By the quasismoothness of $X$, $f'_6 (0, 0, s, t, 0)$ is a rank $2$ quadratic form in variables $s$ and $t$.
We may assume that $f'_6 (0, 0, s, t, 0) = s t$ after replacing $s$ and $t$ suitably.
Then we can write $f'_6 = s t + f_6$ with $f_6 \in (x, y)$.
We can write $g'_6 = \alpha s^2 + \beta t^2 + \gamma s t + g_6$, where $g_6 \in (x, y)$.
Replacing $\mathrm{F}_2$ by $\msF_2 - \gamma z \msF_1$ and then replacing $w \mapsto w + \gamma y z$, we may assume $\gamma = 0$ and we obtain the desired defining polynomials.
\end{proof}

We choose homogeneous coordinates as in Lemma~\ref{lem:No20excl2C2eq}.
We can choose $x, s, t$ as local orbifold coordinates of $X$ at $\msp$ and $\ord_E (x, s, t) = \frac{1}{2} (1, 1, 1)$.
By the equations $\msF_1 = \msF_2 = 0$, we have $\ord_E (y) = \frac{2}{2}$ and
\[
\ord_E (w) =
\begin{dcases}
\frac{2}{2}, & \text{if either $\alpha \ne 0$ or $\beta \ne 0$}, \\
\frac{4}{2}, & \text{if $\alpha = \beta = 0$}.
\end{dcases}.
\]
Let $S \in |\mfm_{\msp} (2 A)|$ be a general member and set $T \coloneq (x = 0)_X$.
Note that $\tilde{S} \sim 2 \varphi^*A - E = 2 B$ and $\tilde{T} \sim \varphi^*A - \frac{1}{2} E = B$.

\begin{Prop}
\label{prop:No20excl2C2}
The following assertions hold.
\begin{enumerate}
\item If $\alpha \beta \ne 0$, then the intersection $\Xi \coloneq \tilde{S} \cap \tilde{T}$ is the union of $2$ irreducible and reduced curves that are numerically proportional to each other, and $(T \cdot \Xi) < 0$.
\item If exactly one of $\alpha$ and $\beta$ is zero, then the intersection matrix of $\tilde{T}|_{\tilde{S}}$ is negative-definite.
\item If $\alpha = \beta =0$, then there are infinitely many curves on $Y$ that intersect $B$ nonpositively.
\end{enumerate}
In particular, $\msp$ is not a maximal center.
\end{Prop}

\begin{proof}
We have
\[
S \cap T = (x = y = s t = z^2 w + z (\alpha s^2 + \beta t^2) + w^2 = 0) = \Gamma \cup \Delta,
\]
where 
\[
\begin{split}
\Gamma &\coloneq  (x = y = s = z^2 w + \beta z t^2 + w^2 = 0), \\
\Delta &\coloneq (x = y = t = z^2 w + \alpha z s^2 + w^2 = 0).
\end{split}
\]
The curves $\Gamma$ and $\Delta$ are clearly reduced and they are irreducible if and only if $\beta \ne 0$ and $\alpha \ne 0$, respectively.
Moreover, $\Gamma$ (resp.\ $\Delta$) contains a WCI curve of type $(1, 2, 3, 4)$ if and only if $\beta = 0$ (resp.\ $\alpha = 0$).

(1)
Suppose that $\alpha \beta \ne 0$.
Then, $\Gamma$ and $\Delta$ are both irreducible and reduced.
Their proper transforms $\tilde{\Gamma}$ and $\tilde{\Delta}$ on $Y$ intersect $E$ transversally at $1$ smooth point of $Y$, hence $(\tilde{\Gamma} \cdot E) = (\tilde{\Delta} \cdot E) = 1$.
We have $(A \cdot \Gamma) = (A \cdot \Delta) = 1/3$ and hence 
\[
(B \cdot \tilde{\Gamma}) = (B \cdot \tilde{\Delta}) = \frac{1}{3} - \frac{1}{2} = - \frac{1}{6}.
\]
Thus, $\tilde{\Gamma}$ and $\tilde{\Delta}$ are numerically proportional to each other since $B$ and $E$ generate $\Pic (Y)_{\mbQ}$.
We have 
\[
(\tilde{T}^2 \cdot \tilde{S}) = 2 (A^3) - \frac{1}{2^2} (E^3) = - \frac{1}{3} < 0.
\]

(2)
Suppose that exactly one of $\alpha$ and $\beta$ is zero.
By symmetry, it suffices to consider the case where $\alpha = 0$ and $\beta \ne 0$.
In this case, 
\[
\Delta = \Delta_1 \cup \Delta_2,
\]
where
\[
\Delta_1 \coloneq (x = y = t = w = 0), \quad
\Delta_2 \coloneq (x = y = t = w + z^2 = 0).
\]
Note that $S$ is a normal surface by Lemma~\ref{lem:normqhyp}.
The curves $\Gamma$, $\Delta_1$ and $\Delta_2$ are quasismooth.
We have $\Gamma \cap \Delta_1 = \{\msp\}$, $\Gamma \cap \Delta_2 = \{\msp'\}$ and $\Delta_1 \cap \Delta_2 = \{\msp_s\}$, where $\msp' \coloneq (0\!:\!0\!:\!1\!:\!0\!:\!0\!:\!-1) \in X$.
It is easy to check that $T$ is quasismooth at $\msp, \msp'$ and $\msp_s$.
It follows that $T$ is quasismooth along $S \cap T = \Gamma \cup \Delta_1 \cup \Delta_2$.

We compute the intersection matrix of $\tilde{S}|_{\tilde{T}} = \tilde{\Gamma} + \tilde{\Delta}_1 + \tilde{\Delta}_2$.
We have
\[
\Gamma \cap \Delta_1 = \{\msp\}, \quad
\Gamma \cap \Delta_2 = \{\msp'\}, \quad
\Delta_1 \cap \Delta_1 = \{\msp_s\}.
\]
The curve $\tilde{\Delta}_1$ on $Y$ intersects $E$ transversally at $1$ smooth point of $E$ that is different from the intersection point $E \cap \tilde{\Gamma}$.
Hence $(E \cdot \tilde{\Delta}_1) = 1$ and $\tilde{\Gamma} \cap \tilde{\Delta}_1 = \emptyset$.
The curve $\Delta_2$ does not pass through $\msp$, hence $(E \cdot \tilde{\Delta}_2) = 0$.
The curve $\tilde{\Delta}_i$ is a smooth rational curve, $\Sing_{\tilde{\Delta}_1} (\tilde{T}) = \{\msp_s\} = \{1 \times \frac{1}{3} (1, 2)\}$ and $\Sing_{\tilde{\Delta}_2} (\tilde{T}) = \{\msp_s, \msp'\} = \{1 \times \frac{1}{2} (1, 1), 1 \times \frac{1}{3} (1, 2)\}$.
It follows that
\[
(\tilde{\Delta}^2_1) = -2 + \frac{2}{3} = - \frac{4}{3}, \quad
(\tilde{\Delta}^2_2) = -2 + \frac{1}{2} + \frac{2}{3} = - \frac{5}{6}.
\]
We have $(\tilde{S}|_{\tilde{T}} \cdot \tilde{\Gamma}) = (\tilde{S} \cdot \tilde{\Gamma}) = -1/3$ as in the previous case and we compute 
\[
\begin{split}
(\tilde{S}|_{\tilde{T}} \cdot \tilde{\Delta}_1) &= 2 (B \cdot \tilde{\Delta}_1) = 2 \left(\frac{1}{6} - \frac{1}{2} \right) = - \frac{2}{3}, \\
(\tilde{S}|_{\tilde{T}} \cdot \tilde{\Delta}_2) &= 2 (B \cdot \tilde{\Delta}_2) = \frac{1}{3}.
\end{split}
\] 
Moreover, we have $(\tilde{\Gamma} \cdot \tilde{\Delta}_1) = 0$ since $\tilde{\Gamma} \cap \tilde{\Delta}_1 = \emptyset$.
By taking intersection numbers of $\tilde{S}|_{\tilde{T}} = \tilde{\Gamma} + \tilde{\Delta}_1 + \tilde{\Delta}_2$ and $\tilde{\Delta}_1, \tilde{\Delta}_2$ and then $\tilde{\Gamma}$, we obtain
\[
(\tilde{\Delta}_1 \cdot \tilde{\Delta}_2) = \frac{2}{3}, \quad
(\tilde{\Gamma} \cdot \tilde{\Delta}_2) = \frac{1}{2}, \quad
(\tilde{\Gamma}^2) = - \frac{5}{6}.
\]
Thus, the intersection matrix of $\tilde{S}|_{\tilde{T}}$ is
\[
\begin{pmatrix}
(\tilde{\Gamma}^2) & (\tilde{\Gamma} \cdot \tilde{\Delta}_1) & (\tilde{\Gamma} \cdot \tilde{\Delta}_2) \\
(\tilde{\Gamma} \cdot \tilde{\Delta}_1) & (\tilde{\Delta}_1^2) & (\tilde{\Delta}_1 \cdot \tilde{\Delta}_2) \\
(\tilde{\Gamma} \cdot \tilde{\Delta}_2) & (\tilde{\Delta}_1 \cdot \tilde{\Delta}_2) & (\tilde{\Delta}_2^2) 
\end{pmatrix} =
\begin{pmatrix}
- \frac{5}{6} & 0 & \frac{1}{2} \\[0.5mm]
0 & - \frac{4}{3} & \frac{2}{3} \\[0.5mm]
\frac{1}{2} & \frac{2}{3} & - \frac{5}{6}
\end{pmatrix},
\]
which is negative definite.
By Lemma~\ref{lem:exclnegdef}, $\msp$ is not a maximal center.

(3)
Finally, suppose that $\alpha = \beta = 0$.
In this case, we have $s^2 y, t^2 y \in \msF_2$ by the quasismoothness of $X$ at $\msp_s, \msp_t$.
Hence, if we set $q (s, t) \coloneq g_8 (0, y, s, t, 0)/y$, then $s^2, t^2 \in q (s, t)$.
We have
\[
\Gamma = \Gamma_1 \cup \Gamma_2, \quad
\Delta = \Delta_1 \cup \Delta_2,
\]
where
\[
\begin{split}
\Gamma_1 &= (x = y = s = w = 0), \quad \Gamma_2 = (x = y = s = w + z^2 = 0), \\
\Delta_1 &= (x = y = t = w = 0), \quad \Delta_2 = (x = y = t = w + z^2 = 0).
\end{split}
\]
For $\lambda \in \mbC$, we set $D_{\lambda} \coloneq (w - \lambda y^2 = 0)_X$.

We claim that
\[
D_{\lambda}|_{T} = \Xi_{\lambda} + \Gamma_1 + \Delta_1,
\] 
where $\Xi_{\lambda}$ is an effective divisor on $T$ that contains neither $\Gamma_1$ nor $\Delta_1$ in its support.
It is sufficient to observe that the scheme-theoretic intersection $T \cap D_{\lambda} = (x = w - \lambda y^2 = 0)_X$ contains $\Gamma_1$ and $\Delta_1$ with multiplicity $1$.
By eliminating $x$ and $w = \lambda y^2$, $T \cap D_{\lambda}$ is isomorphic to the complete intersection in $\mbP (2_y, 2_z, 3_s, 3_t)$ defined by the polynomials
\[
\begin{split}
\overline{\msF}_1 \coloneq \msF_1 (0, y, z, s, t, \lambda y^2) &= z^2 y + \gamma z y^2 + \zeta (\lambda) y^3 + s t, \\
\overline{\msF}_2 \coloneq \msF_2 (0, y, z, s, t, \lambda y^2) &= \lambda z^2 y^2 + \delta z y^3 + \xi (\lambda) y^4 + y q (s, t),
\end{split}
\]
where $\gamma, \delta, \zeta (\lambda), \xi (\lambda) \in \mbC$.
We see that $\zeta (\lambda)$ and $\xi (\lambda)$ depend on $\lambda$.
Moreover we have $\xi (\lambda) \ne 0$ for a general $\lambda$ since $\xi (\lambda) = \lambda^2 + \cdots$ can be seen as a quadric in variable $\lambda$.
We can write $\overline{\msF}_2 = y \msG$, where
\[
\msG \coloneq \lambda z^2 y + \delta z y^2 + \xi (\lambda) y^3 + q (s, t).
\]
Hence $S \cap D_{\lambda} = \Gamma_1 \cup \Delta_1 \cup Z$, where $Z \cong (\overline{\msF}_1 = \msG = 0) \subset \mbP (2, 2, 3, 3)$.
We see that 
\[
(y = 0) \cap Z \cong (y = v t = q (s, t) = 0) = \emptyset
\] 
since $s^2, t^2 \in q (s, t)$.
This shows that $Z$ contains neither $\Gamma_1$ nor $\Delta_1$.
Thus, the claim is proved.

We see that $\tilde{D}_{\lambda} \cap \tilde{T} \cap E$ does not contain a curve, and hence $\tilde{D}_{\lambda}|_{\tilde{T}} = \tilde{\Xi}_{\lambda} + \tilde{\Gamma}_1 + \tilde{\Delta}_1$.
As in the previous case, we have $(B \cdot \tilde{\Gamma}_1) = (B \cdot \tilde{\Delta}_1) = -1/3$.
We have $\tilde{D}_{\lambda} \sim 4 \varphi^*A - \frac{4}{2} E \sim 4 B$ since $\ord_E (w) = 4/2$, and
\[
(B \cdot \tilde{D}_{\lambda} \cdot \tilde{T}) = 4 (A^3) - \frac{4}{2^2} (E^3) = - \frac{2}{3}.
\]
It follows that
\[
(B \cdot \tilde{\Xi}_{\lambda}) = (B \cdot \tilde{D}_{\lambda} \cdot \tilde{T}) - (B \cdot \tilde{\Gamma}_1) - (B \cdot \tilde{\Delta}_1) = 0,
\]
and there is an irreducible component of $\tilde{\Xi}_{\lambda}$ that intersects $B$ nonpositively.
For distinct $\lambda, \lambda' \in \mbC$, we have $(D_{\lambda} \cap T) \cap (D_{\lambda'} \cap T) = (x = y = w = 0)_X = \Gamma_1 \cup \Delta_1$ set-theoretically, and hence there are infinitely many curves that intersect $B$ nonpositively.
By Lemma~\ref{lem:excl:infinitecurve}, $\msp$ is not a maximal center.
\end{proof}

\begin{Rem}
The following cases are not treated in \cite{AZ16}.
\begin{enumerate}
\item The case where $z w \in \msF_1$ and exactly one of $z s^2$ and $z t^2$ appears in $\msF_2$ with nonzero coefficients.
This case is covered in \S \ref{sec:No20excl2C1}.
\item The case where $z w \notin \msF_1$ and $z s^2, z t^2, z s t \notin \msF_2$.
This is exactly the case where $z w \notin \msF_1$ and $\alpha = \beta = 0$ in \S \ref{sec:No20excl2C2}.
\end{enumerate}
\end{Rem}

\subsection{Family \textnumero~31 and the $\frac{1}{3} (1, 1, 2)$ point}
\label{sec:No31MC3}

Let $X = X_{8, 10} \subset \mbP (1, 2, 3, 4, 4, 5)$ be a member of family \textnumero~31 and let $\msp = \msp_z$ be the $\frac{1}{3} (1, 1, 2)$ point of $X$.
The generality assumption in \cite{OkadaI} imposed for $\msp \in X$ is the nonexistence of a WCI curve of type $(1, 2, 4, 5)$ on $X$ passing through $\msp$.
We assume that there is a WCI curve of type $(1, 2, 4, 5)$ on $X$ passing through $\msp$.

\subsubsection{Case: $z w \in \msF_1$ and either $z^2 s \in \msF_2$ or $z^2 t \in \msF_2$}
\label{sec:No31excl3C1}

We consider the case where $z w \in \msF_1$ and either $z^2 s \in \msF_2$ or $z^2 t \in \msF_2$.
Without loss of generality we may assume that $z^2 s \in \msF_2$.

\begin{Lem}
\label{lem:No31excl3C1eq}
The defining polynomials of $X$ can be written as
\[
\begin{split}
\msF_1 &= z w + s t + t f_4 + f_8, \\
\msF_2 &= z^2 s + z (t g_3 + g_7) + t^2 y + t g_6 + g_{10},
\end{split}
\]
where $f_4, f_8 \in \mbC [x, y, w]$ and $g_3, g_7, g_6, g_{10} \in \mbC [x, y, s, w]$.
\end{Lem}

\begin{proof}
We can write
\[
\begin{split}
\msF_1 &= z w + f'_8, \\
\msF_2 &= z^2 s + z g'_7 + g'_{10},
\end{split}
\]
where $f'_8, g'_7, g'_{10} \in \mbC [x, y, s, t, w]$.
Let $\Gamma$ be a WCI curve of type $(1, 2, 4, 5)$ on $X$.
Then $\Gamma$ is contained in $(x = y = w = 0)_X$ and we have
\[
(x = y = w = 0)_X = (x = y = w = f'_8 (0, 0, s, t, 0) = z^2 s = 0).
\]
This shows that $s \mid f'_8 (0, 0, s, t, 0)$ and $\Gamma = (x = y = s = w = 0)$.
In particular $\msp_t \in X$ since $t^2 \notin f'_8$.
We have $s t \in f'_8$ and $t^2 y \in g'_{10}$ by the quasismoothness of $X$ at $\msp_t$.
By replacing $t \mapsto \alpha t + h_4 (x, y, s)$ for a suitable $\alpha \in \mbC$ and a homogeneous polynomial $h_4 (x, y, s)$ of degree $4$, we may assume that $f'_8 = s t + t f_4 + f_8$, where $f_4, f_8 \in \mbC [x, y, w]$.
We can write $g'_7 = v g_3 + g_7$ and $g'_{10} = t^2 y + t g_6 + g_{10}$, where $g_i \in \mbC [x, y, s, w]$ by replacing $y$ suitably.
Thus, we obtain the desired defining polynomials.
\end{proof}

We recall the definition of rank $2$ toric varieties.

\begin{Def} \label{def:toricT}
Let $m < n$ be positive integers and $a_1, \dots, a_n, b_1 \dots b_n$ be integers.
Let $R = \mbC [x_1, \dots, x_n]$ be the polynomial ring endowed with a $\mbZ^2$-grading defined by the $2 \times n$ matrix
\[
\begin{pmatrix}
a_1 & \cdots & a_n \\
b_1 & \cdots & b_n 
\end{pmatrix}.
\]
This means that the bi-degree of the variable $x_i$ is $(a_i, b_i) \in \mbZ^2$.
We denote by 
\begin{equation} \label{eq:deftoricT}
\mbT \begin{pNiceArray}{ccc|ccc}[first-row]
x_1 & \dots & x_m & x_{m+1} & \dots & x_n \\
a_1 & \dots & a_m & a_{m+1} & \dots & a_n \\
b_1 & \dots & b_m & b_{m+1} & \dots & b_n
\end{pNiceArray}
\end{equation}
the toric variety whose Cox ring is $R$ and irrelevant ideal is 
\[
I = (x_1, \dots, x_m) \cap (x_{m+1}, \dots, x_n).
\]
In other words, this toric variety is the geometric quotient
\[
(\mbA^n_{x_1, \dots, x_n} \setminus V (I))/(\mbC^*)^2,
\]
where the $(\mbC^*)^2$-action is defined by
\[
(\lambda, \mu) \cdot (x_1, \dots, x_n) = (\lambda^{a_1} \mu^{b_1} x_1, \dots, \lambda^{a_n} \mu^{b_n} x_n).
\]
For a point $\msp$ of the toric variety \eqref{eq:deftoricT}, which we denote by $\mbT$, we pick a point $\msq = (\alpha_1, \dots, \alpha_n) \in \mbA^n$ in the preimage of the natural projection $\mbA^n \setminus V (I) \to \mbT$ over $\msp$.
In this case we express $\msp$ as $\msp = (\alpha_1\!:\!\cdots\!:\!\alpha_m;\alpha_{m+1}\!:\!\cdots\!:\!\alpha_n) \in \mbT$.

\end{Def}

We choose homogeneous coordinates as in Lemma~\ref{lem:No31excl3C1eq}.
We can choose $x, y, t$ as local orbifold coordinates of $X$ at $\msp$ and $\ord_E (x, y, t) = \frac{1}{3} (1, 2, 1)$.
By the equations $\msF_1 = \msF_2 = 0$, we have $\ord_E (s, w) = \frac{1}{3} (4, 5)$.
It follows that $\varphi \colon Y \to X$ can be embedded into the weighted blowup $\Phi \colon \mbT \to \mbP (1, 2, 3, 4, 4, 5)$ at $\msp$ with weight $\wt (x, y, s, t, w) = \frac{1}{3} (1, 2, 4, 1, 5)$.
The $2$-ray game for $\mbT$ is described as follows:

\[
\xymatrix{
\text{$\mbT := \mbT \begin{pNiceArray}{cc|ccccc}[first-row]
u & z & t & x & y & s & w \\
0 & 3 & 4 & 1 & 2 & 4 & 5 \\
-3 & 0 & 1 & 1 & 2 & 4 & 5
\end{pNiceArray}$} \ar@{-->}[r]^{\Theta} \ar[d]_{\Phi} &
\text{$\mbT \begin{pNiceArray}{ccc|cccc}[first-row]
u & z & t & x & y & s & w \\
0 & 3 & 4 & 1 & 2 & 4 & 5 \\
1 & 1 & 1 & 0 & 0 & 0 & 0
\end{pNiceArray} =: \mbT'$} \ar[d]^{\Phi'} \\
\mbP (3_z, 4_t, 1_x, 2_y, 4_s, 5_w) & \mbP (1_x, 2_y, 4_s, 5_w)}
\]

The game ends with a $\mbP^2$-fibration $\Phi' \colon \mbT' \to \mbP (1, 2, 4, 5)$.
Let $\theta \coloneq \Theta|_Y \colon Y \ratmap Y' \coloneq \Theta_*Y$ be the induced birational map.
We see that $Y$ and $Y'$ are defined in $\mbT$ and $\mbT'$ respectively by the (same) equations
\[
\begin{split}
\mbF_1 (u, z, t, x, y, s, w) &= z w + s t + t f_4 + u f_8 = 0, \\
\mbF_2 (u, z, t, x, y, s, w) &= z^2 s + z (t g_3 + u g_7) + t^2 y + u t g_6 + u^2 g_{10} = 0.
\end{split}
\]

\begin{Lem}
\label{lem:No31excl3C1flip}
The map $\theta \colon Y \ratmap Y'$ is an inverse flip.
\end{Lem}

\begin{proof}
The $\theta$-flipping curve is $\Gamma \coloneq (x = y = s = w = 0) \subset Y \subset \mbT$ and the $\theta$-flipped curve is
\[
\Gamma' \coloneq (u = z = 0) \cap Y' = (u = z = s + f_4 = y = 0) \subset \mbT'.
\]
We have an isomorphism $Y \setminus \Gamma \cong Y' \setminus \Gamma'$, and hence $Y'$ has terminal singularities outside $\Gamma'$.
It remains to consider singularities of $Y'$ along $\Gamma'$.
The Jacobian matrix of the cone of $Y'$ along $\Gamma'$ is of the form
\[
\begin{split}
J_{C_{Y'}}|_{\Gamma'} &\coloneq
\begin{pmatrix}
\frac{\prt \mbF_1}{\prt u} & \frac{\prt \mbF_1}{\prt z} & \frac{\prt \mbF_1}{\prt t} & \frac{\prt \mbF_1}{\prt x} & \frac{\prt \mbF_1}{\prt y} & \frac{\prt \mbF_1}{\prt s} & \frac{\prt \mbF_1}{\prt w} \\[1mm] 
\frac{\prt \mbF_2}{\prt u} & \frac{\prt \mbF_2}{\prt z} & \frac{\prt \mbF_2}{\prt t} & \frac{\prt \mbF_2}{\prt x} & \frac{\prt \mbF_2}{\prt y} & \frac{\prt \mbF_2}{\prt s} & \frac{\prt \mbF_2}{\prt w} 
\end{pmatrix}|_{\Gamma'} \\
&=
\begin{pmatrix}
f_8 & w & 0 & t \frac{\prt f_4}{\prt x} & t \frac{\prt f_4}{\prt y} & t & 0 \\[1mm]
t g_6 & t g_3 & 0 & 0 & t^2 & 0 & 0
\end{pmatrix}
\end{split}
\]
We see that $J_{C_{Y'}}$ is of rank $2$ at any point of $\Gamma'$ since $t$ does not vanish on $\Gamma'$.
This shows that $Y'$ is quasismooth along $\Gamma'$.
The variety $\mbT'$ is smooth on the open set $(x \ne 0)$, hence $Y'$ is smooth along $\Gamma' \cap (x \ne 0)$ since $Y'$ is quasismooth along $\Gamma'$.
We have $\Gamma' \cap (x = 0) = \{\msp'\}$, where $\msp' \coloneq (0\!:\!0\!:\!1; 0\!:\!0\!:\!0\!:\!1) \in Y'$.
The open set $\mbU'_{t, w} \coloneq (t \ne 0) \cap (w \ne 0) \subset \mbT'$ is the quotient of the affine $5$-space
\[
\mbU'_{v, w} = \Spec \mbC [\tilde{u}, \tilde{z}, \tilde{x}, \tilde{y}, \tilde{s}] \cong \mbA^5/\bmu_5
\]
where
\[
\tilde{u} =  \frac{u}{t w^{-4/5}}, \ 
\tilde{z} = \frac{z}{t w^{-1/5}}, \ 
\tilde{x} = \frac{x}{w^{1/5}}, \ 
\tilde{y} = \frac{y}{w^{2/5}}, \ 
\tilde{s} = \frac{s}{w^{4/5}},
\]
and the $\bmu_5$-action is given by 
\[
\frac{1}{5} (1_{\tilde{u}}, 4_{\tilde{z}}, 1_{\tilde{x}}, 2_{\tilde{y}}, 4_{\tilde{s}}).
\]
$Y' \cap \mbU'_{v, w}$ is defined by the equations 
\[
\begin{split}
\mbF_1 (\tilde{u}, \tilde{z}, 1, \tilde{x}, \tilde{y}, \tilde{s}, 1) &= \tilde{z} + \tilde{s} + \tilde{f}_4 + \tilde{u} \tilde{f}_8 = 0, \\
\mbF_2 (\tilde{u}, \tilde{z}, 1, \tilde{x}, \tilde{y}, \tilde{s}, 1) &= \tilde{z}^2 \tilde{s} + \tilde{z} (\tilde{g}_3 + \tilde{u} \tilde{g}_7) + \tilde{y} + \tilde{u} \tilde{g}_6 + \tilde{u}^2 \tilde{g}_{10} = 0,
\end{split}
\]
where $\tilde{f}_i = f (\tilde{x}, \tilde{y}, \tilde{w})$ and $\tilde{g}_i = g (\tilde{x}, \tilde{y}, \tilde{s}, \tilde{w})$.
The point $\msp'$ corresponds to the image of the origin on $\mbU'_{v, w}$.
Thus, the singularity of $Y'$ at $\msp'$ is of type $\frac{1}{5} (1, 1, 4)$.
This shows that $Y'$ has only terminal singularities and $\theta^{-1} \colon Y' \ratmap Y$ is a flip in the terminal category.
\end{proof}

We set $\varphi' \coloneq \Phi'|_{Y'} \colon Y' \to \mbP (1, 2, 4, 5)$.

\begin{Prop}
\label{prop:No31excl3C1}
The following assertions hold.
\begin{enumerate}
\item If $\varphi' \colon Y' \to \mbP (1, 2, 4, 5)$ contracts a divisor, then $\msp$ is not a maximal center.
\item If $\varphi'$ does not contract a divisor, then $\varphi$ initiates a Sarkisov self-link.
\end{enumerate}
\end{Prop}

\begin{proof}
The morphism $\varphi' \colon Y' \to \mbP (1, 2, 4, 5)$ is defined by $\left|-mK_Y\right|$ for a sufficiently divisible $m \gg 0$.
The $\varphi'$-fiber over a general point of $\mbP (1, 2, 4, 5)$ is a complete intersection of a line and a conic in $\mbP^2 = \mbP (1_u, 1_z, 1_t)$.
This shows that $\varphi'$ is generically $2$ to $1$.
Let 
\[
Y' \overset{\psi'}{\to} Z \overset{\pi}{\to} \mbP (1, 2, 4, 5)
\] 
be the Stein factorization of $\varphi'$.

Suppose that $\varphi'$ contracts a divisor.
Then $\psi'$ contracts a divisor.
This implies that $-K_{Y'}$ is in the boundary of $\bMov (Y')$, and hence $-K_Y$ is in the boundary of $\bMov (Y)$.
This shows that $\msp$ is not a maximal center by Lemma~\ref{lem:exclmcmob}.

Suppose that $\varphi'$ does not contract a divisor.
Note that $\varphi'$ cannot be an isomorphism since it contracts fibers over the nonempty set $(w = s + f_4 = f_8 = 0) \subset \mbP (1_x, 2_y, 4_s, 5_w)$.
Hence $\psi' \colon Y' \to Z$ is a flopping contraction.
Let $\tau_Z \colon Z \to Z$ be the biregular involution that switches two points in $\pi$-fibers, and let $\tau \coloneq {\psi'}^{-1} \circ \tau_Z \circ \psi' \colon Y' \ratmap Y'$ be the induced birational map.
Then, $\tau$ gives the flop of $\psi'$ (see \cite[Lemma~3.2]{OkadaII}) and we have the diagram:
\[
\xymatrix{
& \ar[ld]_{\varphi} Y \ar@{-->}[r]^{\theta} & Y' \ar@{-->}[rrr]^{\tau} \ar[rd] _{\psi'} & & & \ar[ld]^{\psi'} Y' \ar@{-->}[r]^{\theta^{-1}} & Y \ar[rd]^{\varphi} & \\
X & & & Z \ar[r]_{\tau_Z}^{\cong} & Z & & & X}
\]
This gives a Sarkisov self-link of $X$.
\end{proof}

\subsubsection{Case: $z w \in \msF_1$ and $z^2 s, z^2 t \notin \msF_2$}
\label{sec:No31excl3C2}

We consider the case where $z w \in \msF_1$ and $z^2 s, z^2 t \notin \msF_2$.

\begin{Lem}
\label{lem:No31excl3C2eq}
The defining polynomials of $X$ can be written as
\[
\begin{split}
\msF_1 &= z w + s t + f_8, \\
\msF_2 &= z^3 x + z^2 g_4 + z g_7 + w^2 + g_{10},
\end{split}
\]
where $f_8, g_4 \in \mbC [x, y]$ and $g_7, g_{10} \in \mbC [x, y, s, t]$ with $s^2 y, t^2 y \in g_{10}$.
\end{Lem}

\begin{proof}
We have $z w \in \mcF_1$ and hence we may assume $\msF_1 = z w + s t + f_8 (x, y)$ after replacing $z, s, t$.
By the assumption that $z^2 s, z^2 t \notin \msF_2$ and by the quasismoothness of $X$ at $\msp$, we have $z^3 x \in \msF_2$.
By replacing $w$, we may write $\msF_2$ as in the statement.
\end{proof}

We choose homogeneous coordinates as in Lemma~\ref{lem:No31excl3C2eq}.
We can choose $y, s, t$ as local orbifold coordinates of $X$ at $\msp$ and $\ord_E (y, s, t) = \frac{1}{3} (2, 1, 1)$.
By the equations $\msF_1 = \msF_2 = 0$, we have $\ord_E (x, w) = \frac{1}{3} (4, 2)$.
We set $S \coloneq (x = 0)_X$ and let $T \in |2A|$ be a general member.
We have $\tilde{S} \sim \varphi^*A - \frac{4}{3} E = B - E$ and $\tilde{T} \sim 2 \varphi^*A - \frac{2}{3} E = 2 B$.

\begin{Prop}
\label{prop:No31excl3C2}
The support of the intersection $\tilde{S} \cap \tilde{T}$ is the union of two irreducible and reduced curves that are numerically equivalent to each other and $(\tilde{T}^2 \cdot \tilde{S}) < 0$.
In particular, $\msp$ is not a maximal center.
\end{Prop}

\begin{proof}
We have 
\[
(S \cap T)_{\mathrm{red}} = (x = y = z w + s t = w^2 = 0)_{\mathrm{red}} = \Gamma \cup \Delta,
\]
where
\[
\Gamma = (x = y = s = w = 0), \quad \Delta = (x = y = t = w = 0).
\]
We have $(A \cdot \Gamma) = (A \cdot \Delta) = 1/12$ and $(E \cdot \tilde{\Gamma}) = (E \cdot \tilde{\Delta}) = 1$.
It follows that $(B \cdot \tilde{\Gamma}) = (B \cdot \tilde{\Delta}) = - 1/4$, and hence $\tilde{\Gamma}$ is numerically equivalent to $\tilde{\Delta}$ since $\Pic (Y)_{\mbQ}$ is generated by $E$ and $B$.
We have
\[
(\tilde{T}^2 \cdot \tilde{S}) = 4 (A^3) - \frac{2^2 \cdot 4}{3^3} (E^3) = - 2.
\]
By Lemma~\ref{lem:exclnumprop}, $\msp$ is not a maximal center.
\end{proof}

\subsubsection{$z w \notin \msF_1$ and either $z^2 s \in \msF_2$ or $z^2 t \in \msF_2$}
\label{sec:No31excl3C3}

We consider the case where $z w \notin \msF_1$ and either $z^2 s \in \msF_2$ or $z^2 t \in \msF_2$.

\begin{Lem}
\label{lem:No31excl3C3eq}
The defining polynomials of $X$ can be written as
\[
\begin{split}
\msF_1 &= z^2 y + z (t x + f_5) + w f_3 + s t + t f_4 + f_8, \\
\msF_2 &= z^2 s + z (t g_3 + g_7) + w^2 + t^2 g_2 + t g_6 + g_{10},
\end{split}
\]
where $f_i \in \mbC [x, y]$ and $g_i \in \mbC [x, y, s]$.
\end{Lem}

\begin{proof}
We may assume $z^2 s \in \msF_2$.
By the quasismoothness of $X$, we have $z^2 y \in \msF_1$ and $w^2 \in \msF_2$.
By choosing $y, s$ and $w$ suitably, we can write
\[
\begin{split} 
\msF_1 &= z^2 y + z f'_5 + w f_3 + f'_8, \\
\msF_2 &= z^2 s + z g'_7 + w^2 + g'_{10},
\end{split}
\]
for some $f'_i, g'_i \in \mbC [x, y, s, t]$.
Thus, we may assume that $\msF_1$ and $\msF_2$ are as in the statement after replacing coordinates suitably.
Let $\Gamma$ be a WCI curve of type $(1, 2, 4, 5)$ on $X$.
Arguing as in the proof of Lemma~\ref{lem:No31excl3C1eq}, we have $\Gamma = (x = y = s = w = 0)$ and $t^2 \notin f'_8$.
By the quasismoothness of $X$ at $\msp_t \in X$, we have $s t \in f'_8$ and we may assume $\coeff_{f'_8} (st) = 1$ by rescaling $t$.
Filtering off terms divisible by $s$ in $\msF_1$ and replacing $t$, we may assume that $st$ is the only monomial in $\msF_1$ divisible by $s$.
Then we can write $f'_5 = \zeta t x + f_5$ and $f'_8 = s t + t f_4 + f_8$, where $\zeta \in \mbC$ and $f_i \in \mbC [x, y]$.
We can also write $g'_7 = t g_3 + g_7$ and $g'_{10} = t^2 g_2 + t g_6 + g_{10}$, where $g_i \in \mbC [x, y, s]$.
If $\zeta = 0$, then $X$ is not quasismooth at $(0\!:\!0\!:\!\sqrt[4]{\alpha}\!:\!0\!:\!1\!:\!0) \in X$, where $\alpha \coloneq \coeff_{g_2}(y)$, which is impossible.
Hence $\zeta \ne 0$ and we may assume $\zeta = 1$ by rescaling $x$.
Thus we obtain the desired form of defining polynomials.
\end{proof}

We choose homogeneous coordinates as in Lemma~\ref{lem:No31excl3C3eq}.
We can choose $x, t, w$ as local orbifold coordinates of $X$ at $\msp$ and $\ord_E (x, t, w) = \frac{1}{3} (1, 1, 2)$.
By the equations $\msF_1 = \msF_2 = 0$, we have $\ord_E (y, s) = \frac{1}{3} (2, 4)$.
By setting $\xi \coloneq z y + t x + f_5$, we have $z \xi + w f_3 + s t + t f_4 + f_8 = 0$, hence $\ord_E (\xi) = 5/3$.
We think of $\xi$ as a new variable of weight $5$ and we think of $X$ as a complete intersection in $\mbP (1, 2, 3, 4, 4, 5, 5)$ defined by the three equations
\[
\begin{split}
\msG_1 &\coloneq z \xi + w f_3 + s t + t f_4 + f_8 = 0, \\
\msG_2 &\coloneq \msF_2 = z^2 s + z (t g_3 + g_7) + w^2 + t^2 g_2 + t g_6 + g_{10} = 0, \\
\msG_3 &= \xi - (z y + t x + f_5) = 0.
 \end{split}
\]
Then, $\varphi \colon Y \to X$ can be embedded into the weighted blowup $\Phi \colon \mbT \to \mbP (1, 2, 3, 4, 4, 5, 5)$ at $\msp$ with weight $\wt (x, y, s, t, w, \xi) = \frac{1}{3} (1, 2, 4, 1, 2, 5)$.
The $2$-ray game for $\mbT$ is described as follows:
\[
\xymatrix{
\text{$\mbT := \mbT \begin{pNiceArray}{cc|cccccc}[first-row]
u & z & t & w & x & y & s & \xi \\
0 & 3 & 4 & 5 & 1 & 2 & 4 & 5 \\
-3 & 0 & 1 & 2 & 1 & 2 & 4 & 5
\end{pNiceArray}$} \ar@{-->}[r]^{\Theta} \ar[d]_{\Phi} &
\text{$\mbT \begin{pNiceArray}{cccc|cccc}[first-row]
u & z & t & w & x & y & s & \xi \\
0 & 3 & 4 & 5 & 1 & 2 & 4 & 5 \\
1 & 1 & 1 & 1 & 0 & 0 & 0 & 0
\end{pNiceArray} =: \mbT'$} \ar[d]^{\Phi'} \\
\mbP (3_z, 4_t, 5_w, 1_x, 2_y, 4_s, 5_{\xi}) & \mbP (1_x, 2_y, 4_s, 5_{\xi})}
\]

The game ends with a $\mbP^3$-fibration $\Phi' \colon \mbT' \to \mbP (1, 2, 4, 5)$.
Let $\theta \coloneq \Theta|_Y \colon Y \ratmap Y' \coloneq \Theta_*Y$ be the induced birational map.
We see that $Y$ and $Y'$ are defined in $\mbT$ and $\mbT'$ respectively by the (same) equations
\[
\begin{split}
\mbG_1 (u, z, t, w, x, y, s, \xi) &\coloneq z \xi + w f_3 + s t + t f_4 + u f_8 = 0, \\
\mbG_2 (u, z, t, w, x, y, s, \xi) &\coloneq z^2 s + z (t g_3 + u g_7) + w^2 + t^2 g_2 + u t g_6 + u^2 g_{10} = 0, \\
\mbG_3 (u, z, t, w, x, y, s, \xi) &\coloneq u \xi - (z y + t x + u f_5) = 0.
\end{split}
\]

\begin{Lem}
The map $\theta \colon Y \ratmap Y'$ is an inverse flip.
\end{Lem}

\begin{proof}
The $2$-ray game can be decomposed into $\mbT \overset{\Theta_0}{\ratmap} \mbT_0 \overset{\Theta'_0}{\ratmap} \mbT'$, where 
\[
\xymatrix{
\text{$\mbT = \mbT \begin{pNiceArray}{cc|cccccc}[first-row]
u & z & t & w & x & y & s & \xi \\
0 & 3 & 4 & 5 & 1 & 2 & 4 & 5 \\
-3 & 0 & 1 & 2 & 1 & 2 & 4 & 5
\end{pNiceArray}$} \ar@{-->}[r]^{\Theta_0} &
\text{$\mbT \begin{pNiceArray}{ccc|ccccc}[first-row]
u & z & t & w & x & y & s & \xi \\
0 & 3 & 4 & 5 & 1 & 2 & 4 & 5 \\
1 & 1 & 1 & 1 & 0 & 0 & 0 & 0
\end{pNiceArray} =: \mbT_0$.}}
\]
The image $Y_0 \coloneq (\Theta_0)_*Y$ is the complete intersection in $\mbT_0$ defined by $\mbG_1 = \mbG_2 = \mbG_3 = 0$ and we have $(x = y = s = \xi = 0) \cap Y_0 = (x = y = s = \xi = w = 0) = \emptyset$.
This shows that the map $Y_0 \ratmap Y'$ induced by $\Theta_1 \colon \mbT_0 \ratmap \mbT'$ is an isomorphism.
Hence, we need to show that $\theta_0 \coloneq \Theta_0 \colon Y \ratmap Y_0$ is an inverse flip.

The $\theta_0$-flipping curve is $\Sigma \coloneq (x = y = s = w = \xi = 0)  \subset Y \subset \mbT$ and the $\theta_0$-flipped curve is 
\[
\Sigma_0 \coloneq (u = z = 0) \cap Y_0 = (u = z = x = s + f_4 = w^2 + \alpha t^2 y = 0),
\]
where $\alpha \coloneq \coeff_{g_2} (y) \ne 0$.
The Jacobian matrix of the cone of $Y_0$ along $\Sigma_0$ is of the form
\[
\begin{split}
J_{C_{Y_0}}|_{\Sigma_0} &=
\begin{pmatrix}
\frac{\prt \mbG_1}{\prt u} & \frac{\prt \mbG_1}{\prt z} & \frac{\prt \mbG_1}{\prt t} & \frac{\prt \mbG_1}{\prt w} & \frac{\prt \mbG_1}{\prt x} & \frac{\prt \mbG_1}{\prt y} & \frac{\prt \mbG_1}{\prt s} & \frac{\prt \mbG_1}{\prt \xi} \\
\frac{\prt \mbG_2}{\prt u} & \frac{\prt \mbG_2}{\prt z} & \frac{\prt \mbG_2}{\prt t} & \frac{\prt \mbG_2}{\prt w} & \frac{\prt \mbG_2}{\prt x} & \frac{\prt \mbG_2}{\prt y} & \frac{\prt \mbG_2}{\prt s} & \frac{\prt \mbG_2}{\prt \xi} \\
\frac{\prt \mbG_3}{\prt u} & \frac{\prt \mbG_3}{\prt z} & \frac{\prt \mbG_3}{\prt t} & \frac{\prt \mbG_3}{\prt w} & \frac{\prt \mbG_3}{\prt x} & \frac{\prt \mbG_3}{\prt y} & \frac{\prt \mbG_3}{\prt s} & \frac{\prt \mbG_3}{\prt \xi}
\end{pmatrix}|_{\Sigma_0} \\
&=
\begin{pmatrix}
* & * & * & * & * & * & t & 0 \\
* & * & * & * & 0 & \alpha t^3 & 0 & 0 \\
* & * & * & * & - t & 0 & 0 & 0
\end{pmatrix},
\end{split}
\]
where we omit computing the terms $*$ simply because they are not necessary.
We see that $J_{C_{Y_0}}$ is of rank $3$ at any point of $\Sigma_0$ since $\alpha \ne 0$ and $t$ does not vanish at $\Sigma_0$.
This shows that $Y_0$ is quasismooth along $\Sigma_0$ and thus $Y_0$ has only cyclic quotient singularities induced from $\mbT_0$.

We have $\Sigma_0 \cap (y = 0) = \{\msp_0\}$, where $\msp_0 \coloneq (0\!:\!0\!:\!1 ; 0\!:\!0\!:\!0\!:\!0\!:\!1) \in Y_0$.
By an argument similar to that in the proof of Lemma~\ref{lem:No31excl3C1flip}, we see that the singularity is of type $\frac{1}{5} (1, 1, 4)$.
We set $\mbU_{t, y} \coloneq (t \ne 0) \cap (y \ne 0) \subset \mbT_0$.
Then we have $\Sing (\mbT_0) \cap \mbU_{t, y} = (z = w = x = \xi = 0) \cap \mbU_{t, y}$ and $\Sigma_0 \cap \Sing (\mbT_0) \cap \mbU_{t, y} = \emptyset$.
This shows that $\msp_0$ is the only singular point of $Y_0$ along $\Sigma_0$.
Therefore $\theta^{-1} \colon Y' \cong Y_0 \ratmap Y$ is a flip in the terminal category.
\end{proof}

We set $\varphi' \coloneq \Phi'|_{Y'} \colon Y' \to \mbP (1, 2, 4, 5)$.

\begin{Prop}
\label{prop:No31excl3C3}
The following assertions hold.
\begin{enumerate}
\item If $\varphi' \colon Y' \to \mbP (1, 2, 4, 5)$ contracts a divisor, then $\msp$ is not a maximal center.
\item If $\varphi'$ does not contract a divisor, then $\varphi$ initiates a Sarkisov self-link.
\end{enumerate}
\end{Prop}

\begin{proof}
The morphism $\varphi'$ contracts a curve since the $\varphi'$-fiber over the point $(0\!:\!0\!:\!1\!:\!0) \in \mbP (1, 2, 4, 5)$ is a curve.
Thus the assertion follows by a similar argument to that in the proof of Proposition~\ref{prop:No31excl3C1}.
\end{proof}

\subsubsection{$z w \notin \msF_1$ and $z^2 s, z^2 t \notin \msF_2$}
\label{sec:No31excl3C4}

We consider the case where $z w \notin \msF_1$ and $z^2 s, z^2 t \notin \msF_2$.

\begin{Lem}
\label{lem:No31excl3C4eq}
The defining polynomials of $X$ can be written as
\[
\begin{split}
\msF_1 &= z^2 y + z f_5 + s t + w f_3 + f_8, \\
\msF_2 &= z^3 x + z^2 g_4 + z g_7 + w^2 + g_{10},
\end{split}
\]
where $f_i, g_i \in (x, y) \subset \mbC [x, y, s, t, w]$.
\end{Lem}

\begin{proof}
By the quasismoothness of $X$ at $\msp$, we have $z^2 y \in \msF_1$ and $z^3 x \in \msF_2$.
We can choose $s, t$ suitably so that $\msF_1 (0, 0, 0, s, t, 0) = s t$.
Moreover, we have $w^2 \in \msF_2$ and we may assume that $\coeff_{\msF_2} (w^2) = 1$.
It is then easy to see that $\msF_1$ and $\msF_2$ are as in the statement with $f_i, g_i \in (x, y) \subset \mbC [x, y, s, t, w]$.
\end{proof}

We choose homogeneous coordinates as in Lemma~\ref{lem:No31excl3C4eq}.
We can choose $s, t, w$ as local orbifold coordinates of $X$ at $\msp$ and $\ord_E (s, t, w) = \frac{1}{3} (1, 1, 2)$.
By the equations $\msF_1 = \msF_2 = 0$, we have $\ord_E (x, y) = \frac{1}{3} (4, 2)$.
We set $S \coloneq (x = 0)_X$ and $T = (y = 0)_X$.
We have $\tilde{S} \sim \varphi^*A - \frac{4}{3} E = B - E$ and $\tilde{T} \sim 2 \varphi^*A - \frac{2}{3} E = 2 B$.

\begin{Prop}
\label{prop:No31excl3C4}
The support of $\tilde{S} \cap \tilde{T}$ is the union of two irreducible and reduced curves that are numerically equivalent to each other, and $(\tilde{T}^2 \cdot \tilde{S}) < 0$.
In particular, $\msp$ is not a maximal center.
\end{Prop}

\begin{proof}
We have 
\[
(S \cap T)_{\red} = (x = y = s t = w^2 =0)_{\red} = \Gamma \cup \Delta,
\]
where
\[
\begin{split}
\Gamma &\coloneq (x = y = s = w = 0), \\
\Delta &\coloneq (x = y = t = w = 0).
\end{split}
\]
We have $(A \cdot \Gamma) = (A \cdot \Delta) = 1/12$ and $(E \cdot \tilde{\Gamma}) = (E \cdot \tilde{\Delta}) = 1$.
This shows that $\tilde{\Gamma}$ and $\tilde{\Delta}$ are numerically equivalent since $B$ and $E$ generate $\Pic (Y)_{\mbQ}$.
Finally, we have 
\[
(\tilde{T}^2 \cdot \tilde{S}) = 2^2 (A^3) - \frac{2^2 \cdot 4}{3^3} (E^3) = -2.
\]
By Lemma~\ref{lem:exclnumprop}, $\msp$ is not a maximal center.
\end{proof}

\begin{Rem}
The case of $z w \in \msF_1$, which is covered by \S \ref{sec:No31excl3C1} and \S \ref{sec:No31excl3C2}, is not treated in \cite{AZ16}.
In particular, the construction of the Sarkisov self-link (given in \S \ref{sec:No31excl3C1}) is not discussed in \cite{AZ16}.
\end{Rem}

\subsection{Family \textnumero~37 and $\frac{1}{2} (1, 1, 1)$ points}

Let $X = X_{8, 12} \subset \mbP (1, 2, 3, 4, 5, 6)$ be a member of family \textnumero~37 and let $\msp$ be a $\frac{1}{2} (1, 1, 1)$ point.
The generality assumption in \cite{OkadaI} imposed for $\msp \in X$ is the condition $t^2 \in \msF_1$.
We assume that $t^2 \notin \msF_1$.

\subsubsection{Case: $\msp \notin (y = 0)_X$}

We consider the case where $\msp \notin (y = 0)_X$.

\begin{Prop}
\label{prop:No37excl2C1}
We have $B^2 \notin \Int \bNE (Y)$.
In particular, $\msp$ is not a maximal center.
\end{Prop}

\begin{proof}
We may assume that $\msp = \msp_y$.
By the quasismoothness of $X$ at $\msp_t$, we have $t^2 y \in \msF_2$.
Then, we have
\[
(x = z = s = w = 0)_X = \{ \msp, \msp_t\},
\]
that is, the set $\{x, z, s, w\}$ isolates $\msp$.
We have $\ord_E (x) = \ord_E (z) = \frac{1}{2}$ and $\ord_E (s), \ord_E (w) \ge 1$ since $(s = 0)_X$ and $(w = 0)_X$ are Cartier divisors around $\msp$.
By \cite[Lemma~6.6]{OkadaII}, the divisor $3 B + E$ is nef and we compute
\[
(3 B + E \cdot B^2) = 3 (A^3) - \frac{1}{2^3} (E^3) = - \frac{1}{10} < 0.
\]
This shows that $B^2 \notin \bNE (Y)$ and $\msp$ is not a maximal center by Lemma \ref{lem:exclNE}.
\end{proof}

\subsubsection{Case: $\msp \in (y = 0)_X$}

We consider the case where $\msp \in (y = 0)_X$.
We set $S \coloneq (x = 0)_X$ and let $T \in |2A|$ be a general member.

\begin{Prop}
\label{prop:No37excl2C2}
The surface $\tilde{T}$ is normal and the $1$-cycle $\tilde{S}|_{\tilde{T}}$ consists of the sum of irreducible and reduced curves whose intersection matrix is negative-definite.
In particular, $\msp$ is not a maximal center.
\end{Prop}

\begin{proof}
The surface $T$ is normal by Lemma~\ref{lem:normqhyp}.
We may assume that $\msp = (0\!:\!0\!:\!0\!:\!1\!:\!0\!:\!1) \in X$.
By the quasismoothness of $X$, we have $t z \in \msF_1$ and $w^2, s^3 \in \msF_2$.
We may assume that $\coeff_{\msF_1} (t z) = \coeff_{\msF_2} (w^2) = 1$ and $\coeff_{\msF_2} (s^3) = -1$.
Then, we have
\[
\begin{split}
\overline{\msF}_1 &\coloneq \msF_1 (0, 0, z, s, t, w)= t z, \\
\overline{\msF}_2 &\coloneq \msF_2 (0, 0, z, s, t, w) = w^2 + \alpha z^2 w + \beta z^4 - s^3 + \gamma z s t,
\end{split}
\]
where $\alpha, \beta, \gamma \in \mbC$. 
Note that $\alpha^2 - 4 \beta \ne 0$ by the quasismoothness of $X$.
It follows that
\[
S \cap T = (x = y = \overline{\msF}_1 = \overline{\msF}_2 = 0) = \Gamma \cup \Delta,
\]
where 
\[
\begin{split}
\Gamma &\coloneq (x = y = z = w^2 - s^3 = 0), \\ 
\Delta &\coloneq (x = y = t = w^2 + \alpha z^2 w + \beta z^4 - s^3 = 0).
\end{split}
\]
We see that $\Gamma$ and $\Delta$ are irreducible and reduced curves.
We also see that both $\tilde{\Gamma}$ and $\tilde{\Delta}$ intersect $E$ at one smooth point, and $\tilde{\Gamma}$ is disjoint from $\tilde{\Delta}$.
We compute $(A \cdot \Gamma) = 1/10$ and $(A \cdot \Delta) = 1/6$.
Hence $(B \cdot \tilde{\Gamma}) = -2/5$ and $(B \cdot \tilde{\Delta}) = -1/3$.
By taking intersection numbers of $B|_{\tilde{T}} = \tilde{S}|_{\tilde{T}} = \tilde{\Gamma} + \tilde{\Delta}$, we have $(\tilde{\Gamma}^2) = -2/5$ and $(\tilde{\Delta}^2) = -1/3$.
Therefore, the intersection matrix of $\tilde{S}|_{\tilde{T}} = \tilde{\Gamma} + \tilde{\Delta}$ is negative-definite.
By Lemma~\ref{lem:exclnegdef}, $\msp$ is not a maximal center.
\end{proof}

\begin{Rem}
The case where $\msp \in (y = 0)_X$ is not treated in \cite{AZ16}.
\end{Rem}

\subsection{Family \textnumero~37 and $\frac{1}{3} (1, 1, 2)$ points}

Let $X = X_{8, 12} \subset \mbP (1, 2, 3, 4, 5, 6)$ be a member of family \textnumero~37 and let $\msp$ be a $\frac{1}{3} (1, 1, 2)$ point.
The generality assumption in \cite{OkadaI} imposed for $\msp \in X$ is that $s^2 \in \msF_1$ and there is no WCI curve of type $(1, 2, 6, 8)$ on $X$ passing through $\msp$.
We assume that either $s^2 \notin \msF_1$ or there is a WCI curve of type $(1, 2, 6, 8)$ on $X$ passing through $\msp$.

\begin{Lem}
\label{lem:No37exclzt}
Under the above assumption, we have $z t \in \msF_1$.
\end{Lem}

\begin{proof}
Suppose that $z t \notin \msF_1$.
Then $\msF_1 = z^2 y + z f_5 + f_8$ for some $f_5, f_8 \in \mbC [x, y, s, t, w]$ with $t \notin f_5$.
We see that $f_5 \in (x, y)$ and we can write $f_8 = \alpha s^2 + f'_8$, where $\alpha \in \mbC$ and $f'_8 \in (x, y)$.

Suppose that $\alpha \ne 0$.
Then, there is a WCI curve $\Gamma \subset X$ of type $(1, 2, 6, 8)$ passing through $\msp$.
Such a curve $\Gamma$ must be contained in $(x = y = w = 0)_X$ and
\[
(x = y = w = 0)_X = (x = y = w = s^2 = \msF_2 = 0).
\]
This is impossible.

Thus $\alpha = 0$.
In this case, $\msF_1 \in (x, y)$ and we can write $\msF_1 = x \msG_1 + y \msG_2$ for some $\msG_1, \msG_2 \in \mbC [x, y, z, s, t, w]$.
 Then $X$ is not quasismooth along the nonempty set
 \[
 (x = y = \msG_1 = \msG_2 = \msF_2 = 0) \subset \mbP (1, 2, 3, 4, 5, 6).
 \]
 This is a contradiction and the assertion is proved.
\end{proof}

\subsubsection{Case: $s^2 \in \msF_1$ and there is a WCI curve of type $(1, 2, 6, 8)$ on $X$ passing through $\msp$}

We consider the case where $s^2 \in \msF_1$ and there is a WCI curve of type $(1, 2, 6, 8)$ on $X$ passing through $\msp$.

\begin{Lem}
\label{lem:No37excl3C1eq}
The defining polynomials of $X$ can be written as
\[
\begin{split}
\msF_1 &= z t + s^2 + f_8, \\
\msF_2 &= z^2 w + z g_9 + w^2 + t^2 y + g_{12},
\end{split}
\]
where $f_8, g_9, g_{12} \in \mbC [x, y, s, t, w]$ are such that $f_8 \in (x, y)$, $g_9 \in (x^3, x y, y^2)$ and $g_{12} \in (x^3, x y, y^2, w x, w y)$.
\end{Lem}

\begin{proof}
By Lemma~\ref{lem:No37exclzt}, by the assumption and by the quasismoothness of $X$, we have $z t, s^2 \in \msF_1$ and $z^2 w, w^2 , t^2 y \in \msF_2$.
We may assume that their coefficients in $\msF_1$ and $\msF_2$ are all $1$ and we can write
\[
\begin{split}
\msF_1 &= z t + s^2 + f_8, \\
\msF_2 &= z^2 w + z (\alpha s t + \beta s^2 x + g_9) + w^2 + \gamma s^3 + t^2 y + g_{12},
\end{split}
\]
where $\alpha, \beta, \gamma \in \mbC$ and $f_8, g_9, g_{12} \in \mbC [x, y, s, t, w]$ with $f_8, g_9, g_{12} \in (x, y) \subset \mbC [x, y, s, t, w]$ with $t^2 x \notin g_9$ and $t^2 y \notin g_{12}$.

Let $\Gamma$ be a WCI curve of type $(1, 2, 6, 8)$ passing through $\msp$.
We have $\Gamma \subset (x = y = w = 0)_X$ and 
\[
(x = y = w = 0)_X = (x = y = w = z t + s^2 = \alpha z s t + \gamma s^3 = 0).
\]
It follows that $\alpha = \gamma$ and $\Gamma = (x = y = w = z t + s^2 = 0)$.
Replacing $\msF_2$ by $\msF_2 - \alpha s \msF_1$, we may assume that $\alpha = \gamma = 0$.
Finally, replacing $\msF_2$ by $\msF_2 - \beta x z \msF_1$ and then replacing $w \mapsto w + \beta r x$, we may assume that $\beta = 0$.
Finally, replacing $y \mapsto y - \coeff_{g_{12}} (t^2 x) x^2$, we may assume that $t^2 x^2 \notin g_{12}$.
Then, it remains to simply observe that $f_8 \in (x^2, y)$, $g_9 \in (x^3, x y, y^2)$ and $g_{12} \in (x^3, x y, y^2, w x, w y)$, which are left to readers.
\end{proof}

We choose homogeneous coordinates as in Lemma~\ref{lem:No37excl3C1eq}.
We can choose $x, y, s$ as local orbifold coordinates of $X$ at $\msp$ and $\ord_E (x, y, s) = \frac{1}{3} (1, 2, 1)$.
By the equations $\msF_1 = \msF_2 = 0$, we have $\ord_E (t, w) = \frac{1}{3} (2, 6)$.

\begin{Prop}
\label{prop:No37excl3C1}
There are infinitely many irreducible and reduced curves on $Y$ that intersect $-K_Y$ non-positively.
In  particular, $\msp$ is not a maximal center.
\end{Prop}

\begin{proof}
For $\lambda, \mu \in \mbC$, we set $S_{\lambda} \coloneq (y - \lambda x^2 = 0)_X$ and $T_{\mu} \coloneq (w - \mu x^6 = 0)_X$.
Then,
\[
S_{\lambda} \cap T_{\mu} = (y - \lambda x^2 = w - \mu x^6 = \overline{\msF}_1 = \overline{\msF}_2 = 0) \subset \mbP (1, 2, 3, 4, 5, 6),
\]
where $\overline{\msF}_i \coloneq \msF_i (x, \lambda x^2, z, s, t, \mu x^6) \in \mbC [x, z, s, t]$.
By the properties of $f_8, g_9, g_{12}$ given in Lemma~\ref{lem:No37excl3C1eq}, we see that $x^2 \mid \overline{\msF}_2$ and $\msG_{\lambda, \mu} = \msG_{\lambda, \mu} (x, z, s, t) \coloneq \overline{\msF}_2/x^2 \in \mbC [x, z, s, t]$ is not divisible by $x$.
We also see that $\lambda t^2$ is the only monomial in $\msG_{\lambda, \mu}$ that is not divisible by $x$ and that $x^2 \mid f_8 (x, \lambda y^2, s, t, \mu x^6)$ since $f_8 \in (x^2, y)$.
Thus, as a $1$-cycle on $X$, we can write
\[
S_{\lambda} \cdot T_{\mu} = 2 \Gamma + \Delta_{\lambda, \mu},
\]
where $\Gamma \coloneq (x = y = w = z t + s^2 = 0)$ is a WCI curve of type $(1, 2, 6, 8)$ and $\Delta_{\lambda, \mu}$ is an effective $1$-cycle on $X$ which does not contain $\Gamma$ in its support.
We have $(A \cdot \Gamma) = 2/15$ and $(E \cdot \tilde{\Gamma}) = 1$, hence $(B \cdot \tilde{\Gamma}) = -1/5$.
We have $\tilde{S}_{\lambda} \sim 2\varphi^*A - \frac{2}{3} E$ and $\tilde{T}_{\lambda} \sim 6 \varphi^*A - \frac{6}{3} E$.
It follows that
\[
(B \cdot \tilde{\Delta}_{\lambda, \mu}) = (B \cdot \tilde{S}_{\lambda} \cdot \tilde{T}_{\mu}) - 2 (B \cdot \tilde{\Gamma}) 
= - \frac{2}{5} + \frac{2}{5} = 0.
\]
For each pair $(\lambda, \mu) \in \mbC^2$, there is at least one component of $\Delta_{\lambda, \mu}$ that intersects $-K_Y$ non-positively.
For distinct $(\lambda, \mu), (\lambda', \mu') \in \mbC^2$, we have $(S_{\lambda} \cap T_{\mu}) \cap (S_{\lambda'} \cap T_{\mu'}) = (x = y = w = 0)_X = \Gamma$ set-theoretically, and hence there are infinitely many curves on $Y$ that intersect $-K_Y$ non-positively.
By Lemma~\ref{lem:excl:infinitecurve}, $\msp$ is not a maximal center.
\end{proof}

\subsubsection{Case: $s^2 \notin \msF_1$}

We consider the case where $s^2 \notin \msF_1$.

\begin{Lem}
\label{lem:No37excl3C2eq}
The defining polynomials of $X$ can be written as
\[
\begin{split}
\msF_1 &= z t + w y + f_8, \\
\msF_2 &= z^2 w + z g_9 + w^2 + s^3 + g_{12},
\end{split}
\]
where $f_8 \in \mbC [x, y, s, t]$ and $g_9, g_{12} \in \mbC [x, y, s, t, w]$ are such that $f_8 \in (x^3, x y, y^2)$, $g_9, g_{12} \in (x, y)$.
\end{Lem}

\begin{proof}
By Lemma~\ref{lem:No37exclzt}, by the assumption and by the quasismoothness of X, we have $z t \in \msF_1$ and $z^2 w, w^2, s^3 \in \msF_2$.
We may assume that their coefficients are all $1$ by rescaling coordinates.
Then we can write
\[
\begin{split}
\msF_1 &= z t + w f_2 + f_8, \\
\msF_2 &= z^2 w + z (\alpha s t + g_9) + w^2 + s^3 + g_{12},
\end{split}
\]
where $\alpha \in \mbC$, $f_2 \in \mbC [x, y]$, $f_8 \in (x, y) \subset \mbC [x, y, s, t]$ and $g_i \in (x, y) \subset \mbC [x, y, s, t, w]$.
Replacing $\msF_2$ by $\msF_2 - \alpha s \msF_1$, we may assume that $\alpha = 0$.
It is straightforward to see that $f_8 \in (x^2, x y, y^2)$.
If $y \notin f_2$, then $\msF_1 \in (x, y, z, t)^2$ and $X$ cannot be quasismooth at any point of the nonempty set
\[
(x = y = z = t = \msF_2 = 0) \subset X.
\]
Hence we may assume that $f_2 = y$ after replacing $y$. 
We then obtain the desired defining polynomials.
\end{proof}

We choose homogeneous coordinates as in Lemma~\ref{lem:No37excl3C2eq}.
We can choose $x, y, s$ as local orbifold coordinates of $X$ at $\msp$ and $\ord_E (x, y, s) = \frac{1}{3} (1, 2, 1)$.
By the equations $\msF_1 = \msF_2 = 0$, we have $\ord_E (t, w) = \frac{1}{3} (5, 3)$.

\begin{Prop}
There are infinitely many irreducible and reduced curves that intersect $-K_Y$ trivially.
In particular, $\msp$ is not a maximal center.
\end{Prop}

\begin{proof}
For $\lambda, \mu \in \mbC$, we set $S_{\lambda} \coloneq (y - \lambda x^2 = 0)_X$ and $T_{\mu} \coloneq (t - \mu x^5 = 0)_X$.
We have
\[
S_{\lambda} \cap T_{\mu} = (y - \lambda x^2 = t - \mu x^5 = \overline{\msF}_1 = \overline{\msF}_2 = 0) \subset \mbP (1, 2, 3, 4, 5, 6),
\]
where $\overline{\msF}_i \coloneq \msF_i (x, \lambda x^2, z, s, \mu x^5, w)$.
We have $x^2 \mid \overline{\msF}_1$ and $\msG_{\lambda, \mu} \coloneq \overline{\msF}_1/x^2 = \lambda w + \cdots$, where any monomial in the omitted term is divisible by $x$.
We define 
\[
\Gamma \coloneq (x = y = t = z^2 w + w^2 + s^3 = 0) \subset X.
\]
Then, $\Gamma$ appears in the $1$-cycle $S_{\lambda} \cdot T_{\mu}$ with multiplicity $2$ and we can write
\[
S_{\lambda} \cdot T_{\mu} = 2 \Gamma + \Delta_{\lambda, \mu},
\]
where $\Delta_{\lambda, \mu}$ is an effective $1$-cycle that does not contain $\Gamma$ in its support.
We have $(A \cdot \Gamma) = 1/6$ and $(E \cdot \tilde{\Gamma}) = 1$, which implies $(B \cdot \tilde{\Gamma}) = -1/6$.
We also have $\tilde{S}_{\lambda} \sim 2 \varphi^*A - \frac{2}{3} E$ and $\tilde{T}_{\mu} = 5 \varphi^*A - \frac{5}{3} E$.
It follows that
\[
(B \cdot \tilde{\Delta}_{\lambda}) = (B \cdot \tilde{S}_{\lambda} \cdot \tilde{T}_{\mu}) - 2 (B \cdot \tilde{\Gamma}) = - \frac{1}{3} + \frac{1}{3} = 0.
\]
For each pair $(\lambda, \mu) \in \mbC^2$, there is at least one component of $\Delta_{\lambda, \mu}$ that intersects $-K_Y$ non-positively.
By Lemma~\ref{lem:excl:infinitecurve}, $\msp$ is not a maximal center.
\end{proof}

\subsection{Family \textnumero~47 and $\frac{1}{3} (1, 1, 2)$ points}

Let $X = X_{10, 12} \subset \mbP (1, 3, 4, 4, 5, 6)$ be a member of family \textnumero~47 and let $\msp  = \msp_y \in X$ be the $\frac{1}{3} (1, 1, 2)$ point.
The generality assumption in \cite{OkadaI} imposed for $\msp \in X$ is the nonexistence of a WCI curve of type $(1, 4, 5, 6)$ on $X$ passing through $\msp$.
We assume that there is a WCI curve of type $(1, 4, 5, 6)$ on $X$ passing through $\msp$.

\subsubsection{Case: either $y^2 z \in \msF_1$ or $y^2 s \in \msF_1$}
\label{sec:No47excl3C1}

We consider the case where either $y^2 z \in \msF_1$ or $y^2 s \in \msF_1$.

\begin{Lem}
\label{lem:No47excl3C1eq}
The defining polynomials of $X$ can be written as
\[
\begin{split}
\msF_1 &= y^2 z + y f_7 + s w + t^2 + f_{10}, \\
\msF_2 &= y^2 w + y (\theta s t + g_9) + w^2 + g_{12},
\end{split}
\]
where $\theta \in \mbC$ and $f_7, g_9, g_{12} \in (x, z) \subset \mbC [x, z, s, t, w]$.
\end{Lem}

\begin{proof}
We may assume $y^2 z \in \msF_1$.
By the quasismoothness of $X$ at $\msp$, we have $y^2 w \in \msF_2$.
We also have $t^2 \in \msF_1$ and $w^2 \in \msF_2$.
Hence we may write 
\[
\begin{split}
\msF_1 &= y^2 z + y f_7 + t^2 + f'_{10}, \\
\msF_2 &= y^2 w + y g'_9 + w^2 + g_{12}
\end{split}
\] 
for some $f_7, f'_{10}, g'_9, g'_{12} \in \mbC [x, z, s, t, w]$ such that $t^2 \notin f'_{10}$ and $w^2 \notin g_{12}$.
We see that $f_7$ is divisible by $x$, and hence $f_7 \in (x, z)$.
We set $\alpha \coloneq \coeff_{g'_9} (s t)$ and write $g'_9 = g_9 + \alpha s t$.
Then, $g_9 \in (x, z)$.
Let $\Gamma \subset X$ be a WCI curve of type $(1, 4, 5, 6)$ passing through $\msp$.
Then $\Gamma$ is contained in $(x = t = w = 0)_X$ and
\[
(x = t = w = 0)_X = (x = t = w = y^2 z = g_{12} (0, z, s, 0, 0) = 0).
\]
This shows that $\Gamma = (x = z = t = w = 0)$ and $s^3 \notin g_{12}$, which implies that $g_{12} \in (x, z)$.
By the quasismoothness of $X$ at $\msp_s \in X$, we have $s w \in f'_{10}$.
Rescaling coordinates, we may write $f'_{10} = s w + f_{10}$, where $s w \notin f_{10}$.
It follows that $f_{10} \in (x, z)$ and the proof is complete.
\end{proof}

We choose homogeneous coordinates as in Lemma~\ref{lem:No47excl3C1eq}.
We can choose $x, s, t$ as local orbifold coordinates of $X$ at $\msp$ and $\ord_E (x, s, t) = \frac{1}{3} (1, 1, 2)$.
By the equations $\msF_1 = \msF_2 = 0$, we have $\ord_E (z) = \frac{4}{3}$ and
\[
\ord_E (w) =
\begin{dcases}
\frac{3}{3}, & \text{if either $\theta \ne 0$ or $s^2 x \in g_9$}, \\
\frac{6}{3}, & \text{otherwise}.
\end{dcases}.
\]
We set $S \coloneq (x = 0)_X$ and let $T$ be a general member of the pencil generated by $z$ and $x^4$.
We have $\tilde{S} \sim \varphi^*A - \frac{1}{3} E = B$ and $\tilde{T} \sim 4 \varphi^*A - \frac{4}{3} E = 4 B$.
Note that $T$ is a normal surface by Lemma~\ref{lem:normqhyp}.

\begin{Prop}
\label{prop:No47excl3C1}
The support of $\tilde{S}|_{\tilde{T}}$ consists of two curves whose intersection matrix is negative-definite.
In particular, $\msp$ is not a maximal center.
\end{Prop}

\begin{proof}
We set $\Gamma \coloneq (x = z = t = w = 0) \subset X$, which is the WCI curve of type $(1, 4, 5, 6)$ passing through $\msp$.
We have
\[
S \cap T = (x = z = w s + t^2 = y^2 w + \theta y s t + w^2 = 0)
\]

Suppose that $\theta \ne 0$.
On the open subset $U_s$, the scheme $S \cap T$ is isomorphic to the $\bmu_4$-quotient of the affine curve
\[
C \coloneq (w + t^2 = y^2 w + \theta y t + w^2 = 0) \subset \mbA^3_{y, t, w}.
\]
Eliminating the variable $w = - t^2$, we have
\[
C \cong (t (-y^2 t + \theta y + t^3) = 0) \subset \mbA^2_{y, t}.
\]
The curve $(t = 0)$ corresponds to $\Gamma$ on $U_s$ and the remaining component of $C$ defined by $- y^2 t + \theta y + t^3 = 0$ is irreducible and reduced since $\theta \ne 0$.
It follows that $S|_T = \Gamma + \Delta$, where $\Delta \ne \Gamma$ is an irreducible and reduced curve since $(s = 0) \cap S \cap T$ is a finite set of points.
We have $(A \cdot \Gamma) = 1/12$ and $(E \cdot \tilde{\Gamma}) = 1$.
Hence $(B \cdot \tilde{\Gamma}) = - 1/4$.
We have $\tilde{S}|_{\tilde{T}} = \tilde{\Gamma} + \tilde{\Delta}$, and hence
\[
(B \cdot \tilde{\Delta}) = (B \cdot \tilde{S} \cdot \tilde{T}) - (B \cdot \tilde{\Gamma}) = - \frac{1}{3} + \frac{1}{4} = - \frac{1}{12}.
\]
We set $m \coloneq (\tilde{\Gamma} \cdot \tilde{\Delta}) \ge 0$.
By taking intersection numbers of $B|_{\tilde{T}} = \tilde{S}|_{\tilde{T}} = \tilde{\Gamma} + \tilde{\Delta}$ and $\tilde{\Gamma}, \tilde{\Delta}$, we have
\[
\begin{pmatrix}
(\tilde{\Gamma}^2) & (\tilde{\Gamma} \cdot \tilde{\Delta}) \\
(\tilde{\Gamma} \cdot \tilde{\Delta}) & (\tilde{\Delta}^2)
\end{pmatrix} =
\begin{pmatrix}
-\frac{1}{4} - m & m \\
m & - \frac{1}{12} - m
\end{pmatrix}
\]
which is clearly negative-definite.
By Lemma~\ref{lem:exclnegdef}, $\msp$ is not a maximal center.

Suppose that $\theta = 0$.
Then $S|_T = 2 \Gamma + \Delta$, where 
\[
\Delta \coloneq (x = z = w + y^2 = w s + t^2 = 0)
\] 
is an irreducible and reduced curve. 
We have $(A \cdot \Gamma) = 1/12$ and $(E \cdot \tilde{\Gamma}) = 1$, hence $(B \cdot \tilde{\Gamma}) = -1/4$.
Then, $(B \cdot \tilde{\Delta}) = (B \cdot \tilde{S} \cdot \tilde{T}) - (B \cdot \tilde{\Gamma}) = -1/12$.
We set $m \coloneq (\tilde{\Gamma} \cdot \tilde{\Delta}) \ge 0$.
By taking intersection numbers of $B|_{\tilde{T}} \sim \tilde{S}|_{\tilde{T}} = 2 \tilde{\Gamma} + \tilde{\Delta}$ and $\tilde{\Gamma}$, $\tilde{\Delta}$, we have
\[
\begin{pmatrix}
(\tilde{\Gamma}^2) & (\tilde{\Gamma} \cdot \tilde{\Delta}) \\
(\tilde{\Gamma} \cdot \tilde{\Delta}) & (\tilde{\Delta}^2)
\end{pmatrix} =
\begin{pmatrix}
- \frac{m}{2} - \frac{1}{8} & m \\
m & -2m - \frac{1}{12}
\end{pmatrix},
\]
which is negative-definite.
By Lemma~\ref{lem:exclnegdef}, $\msp$ is not a maximal center.
\end{proof}

\subsubsection{Case: $y^2 z, y^2 s \notin \msF_1$}

We consider the case where $y^2 z, y^2 s \notin \msF_1$.

\begin{Lem}
\label{lem:No47excl3C2eq}
The defining polynomials of $X$ can be written as
\[
\begin{split}
\msF_1 &= y^3 x + y^2 f_4 + y f_7 + t^2 + f_{10}, \\
\msF_2 &= y^2 w + y (t \ell + g_9) + \ell_1 \ell_2 \ell_3 + g_{12},
\end{split}
\]
where $\ell = \ell (z, s), \ell_i = \ell_i (z, s)$ are linear forms in $z, s$, and $f_i, g_i \in \mbC [x, z, s, t, w]$ are such that $f_i, g_i \in (x, w)$ and no two of $\ell_1, \ell_2, \ell_3$ are mutually proportional.
\end{Lem}

\begin{proof}
By the assumption and by the quasismoothness of $X$, we have $y^3 x, t^2 \in \msF_1$ and $y^2 w \in \msF_2$.
We may assume that their coefficients are all $1$.
We can write
\[
\begin{split}
\msF_1 &= y^3 x + y^2 f_4 + y f_7 + t^2 + f_{10}, \\
\msF_2 &= y^2 w + y (t \ell + g_9) + g'_{12},
\end{split}
\]
where $f_i, g_i \in \mbC [x, y, s, t, w]$, $g'_{12} \in \mbC [x, y, s, t, w]$ with $f_i, g_i \in (x, w)$ and $\ell = \ell (z, s)$ is a linear form in variables $z, s$.
By the quasismoothness of $X$, we can write $g'_{12} (0, z, s, 0, 0) = \ell_1 \ell_2 \ell_3$ for some linear forms $\ell_i = \ell_i (z, s)$ in $z, s$ and no two of them are mutually proportional.
Thus we obtain the desired defining polynomials. 
\end{proof}

We choose homogeneous coordinates as in Lemma~\ref{lem:No47excl3C2eq}.
We can choose $z, s, t$ as local orbifold coordinates of $X$ at $\msp$ and $\ord_E (z, s, t) = \frac{1}{3} (1, 1, 2)$.
By the equation $\msF_1 = \msF_2 = 0$, we have $\ord_E (x, w) = \frac{1}{3} (4, 3)$.
We set $S \coloneq (x = 0)_X$ and let $T$ be a general member of the pencil generated by $w$ and $x^6$.
We have $\tilde{S} \sim \varphi^*A - \frac{4}{3} E = B - E$ and $\tilde{T} \sim 6 \varphi^*A - \frac{3}{3} E = 6 B + E$.

\begin{Prop}
\label{prop:No47excl3C2}
The support of the intersection $\tilde{S} \cap \tilde{T}$ is the union of three irreducible and reduced curves that are numerically equivalent to each other, and $(\tilde{T}^2 \cdot \tilde{S}) < 0$.
In particular, $\msp$ is not a maximal center.
\end{Prop}

\begin{proof}
We have
\[
(S \cap T)_{\red} = (x = w = t = \ell_1 \ell_2 \ell_3 = 0) = \Gamma_1 \cup \Gamma_2 \cup \Gamma_3,
\]
where
\[
\Gamma_i = (x = t = w = \ell_i = 0)
\]
for $i = 1, 2, 3$.
We have $(A \cdot \Gamma_i) = 1/12$ and $(E \cdot \tilde{\Gamma}_i) = 1$ for $i = 1, 2, 3$.
This shows that $(B \cdot \tilde{\Gamma}_i) = - 1/4$, and hence any two of $\tilde{\Gamma}_1, \tilde{\Gamma}_2, \tilde{\Gamma}_3$ are numerically equivalent to each other since $B$ and $E$ generate $\Pic (Y)_{\mbQ}$.
We have
\[
(\tilde{T}^2 \cdot \tilde{S}) = 6^2 (A^3) - \frac{3^2 \cdot 4}{3^3} (E^3) = 3 - 6 = -3 < 0.
\]
By Lemma~\ref{lem:exclnumprop}, $\msp$ is not a maximal center.
\end{proof}

\begin{Rem}
The case where either $y^2 z \in \msF_1$ or $y^2 s \in \msF_1$, which is treated in \S \ref{sec:No47excl3C1}, is not treated in \cite{AZ16}.
\end{Rem}

\subsection{Family \textnumero~51 and the $\frac{1}{4} (1, 1, 2)$ point}

Let $X = X_{10, 14} \subset \mbP (1, 2, 4, 5, 6, 7)$ be a member of family \textnumero~51 and $\msp = \msp_z \in X$ the $\frac{1}{4} (1, 1, 3)$ point.
The generality assumption in \cite{OkadaI} imposed for $\msp \in X$ is the irreducibility of $(x = y = 0)_X$.
We do not consider this condition and exclude $\msp$ as a maximal center without any generality assumption.

\begin{Lem}
\label{lem:No51excl4eq}
\begin{enumerate}
\item Suppose that $z t \in \msF_1$.
Then the defining polynomials of $X$ can be written as
\[
\begin{split}
\msF_1 &= z t + s^2 + f_{10}, \\
\msF_2 &= z^3 y + z^2 g_6 + z (\alpha s^2 + g_{10}) + w^2 + g_{14},
\end{split}
\]
where $\alpha \in \mbC$ and $f_i, g_i \in (x, y) \subset \mbC [x, y, s, t, w]$.
\item Suppose that $z t \notin \msF_1$. 
Then the defining polynomials of $X$ can be written as
\[
\begin{split}
\msF_1 &= z^2 y + z f_6 + s^2 + f_{10}, \\
\msF_2 &= z^2 t + z (\alpha s^2 + g_{10}) + w^2 + g_{14},
\end{split}
\]
where $\alpha \in \mbC$ and $f_i, g_i \in (x, y) \subset \mbC [x, y, s, t, w]$.
\end{enumerate}
\end{Lem}

\begin{proof}
Suppose that $z t \in \msF_1$.
By the quasismoothness of $X$, we have $s^2 \in \msF_1$ and $z^3 y, w^2 \in \msF_2$.
By replacing $t$ and $s$, we may write $\msF_1 = z t + s^2 + f_{10} (x, y, t, w)$.  
It is easy to see that $f_{10} \in (x, y)$.
By replacing $y$ and $w$, we can write $\msF_2$ as in the statement for some $g_i \in \mbC [x, y, s, t]$ with $s^2 \notin g_{10}$.
By replacing $\msF_2$ by $\msF_2 - \lambda z \msF_1$ for a suitable $\lambda \in \mbC$, we may assume that $t \notin g_6$.
Then, it is easy to see that $g_6, g_{10} ,g_{14} \in (x, y)$.

Suppose that $z t \notin \msF_1$.
Then $z^2 y \in \msF_1$ and $z^2 t \in \msF_2$ by the quasismoothness of $X$ at $\msp$.
We also have $s^2 \in \msF_1$ and $w^2 \in \msF_2$.
By a similar argument to that in the previous case, we see that $\msF_1$ and $\msF_2$ are written as in the statement for some $f_i \in \mbC [x, y, t, w]$ and $g_i \in \mbC [x, y, s, t]$.
Note that $t \notin f_6$.
Then it is easy to see that $f_i, g_i \in (x, y)$ and the proof is complete.
\end{proof}

We choose homogeneous coordinates as in Lemma~\ref{lem:No51excl4eq}.
We can choose $x, s, w$ as local orbifold coordinates of $X$ at $\msp$ and $\ord_E (z, s, w) = \frac{1}{4} (1, 1, 3)$.
By the equations $\msF_1 = \msF_2 = 0$, we have
\[
\begin{split}
\ord_E (y) &= 
\begin{dcases}
\frac{6}{4}, & \text{if either $z t \notin \msF_1$ or $z t \in \msF_1$, $\alpha = 0$ and $s x \notin g_6$}, \\
\frac{2}{4}, & \text{otherwise},
\end{dcases} \\
\ord_E (t) &=
\begin{dcases}
\frac{2}{4}, & \text{if either $z t \in \msF_1$ or $z t \notin \msF_1$ and $\alpha \ne 0$}, \\
\frac{6}{4}, & \text{otherwise}.
\end{dcases}
\end{split}
\]
We set $S \coloneq (y = 0)_X$ and $T \coloneq (x = 0)_X$.

\begin{Prop}
\label{prop:No51excl4}
The support of $\tilde{S} \cap \tilde{T}$ is an irreducible and reduced curve and $(\tilde{T}^2 \cdot \tilde{S}) \le 0$.
In particular, $\msp$ is not a maximal center.
\end{Prop}

\begin{proof}
By Lemma~\ref{lem:No51excl4eq}, we have
\[
(S \cap T)_{\mathrm{red}} = 
\begin{cases}
(x = y = z t + s^2 = \alpha z s^2 + w^2 = 0), & \text{if $z t \in \msF_1$ and $\alpha \ne 0$}, \\
(x = y = z t + s^2 = w = 0), & \text{if $z t \in \msF_1$ and $\alpha = 0$}, \\
(x = y = s = z^2 t + w^2 = 0), & \text{if $z t \notin \msF_1$}.
\end{cases}
\]
It is easy to see that $(S \cap T)_{\mathrm{red}}$ is an irreducible curve in the latter two cases.
Suppose that $z t \in \msF_1$ and $\alpha \ne 0$.
Then, on the open set $U_s$, $(S \cap T)_{\red}$ is isomorphic to the $\bmu_5$-quotient of the affine curve
\[
C \coloneq (z t + 1 = \alpha z + w^2 = 0) \subset \mbA^3_{z, s, w}.
\]
Eliminating the variable $z = - \alpha^{-1} w^2$, this curve $C$ is isomorphic to $(- \alpha^{-1} w^2 t + 1 = 0) \subset \mbA^2_{t, w}$, which is clearly irreducible and reduced.
This shows that $(S \cap T)_{\red}$ is irreducible since $S \cap T \cap (s = 0)$ is a finite set of points.
We have $\tilde{T} \sim \varphi^*A - \frac{1}{4} E$ and $\tilde{S} \sim 2 \varphi^*A - \ord_E (y) E$.
It follows that
\[
(\tilde{T}^2 \cdot \tilde{S}) = 
\begin{dcases}
- \frac{1}{3}, & \text{if either $z t \notin \msF_1$ or $z t \in \msF_1$, $\alpha = 0$ and $s x \notin g_6$}, \\
0, & \text{otherwise}.
\end{dcases}
\]
By Lemma~\ref{lem:exclnumprop}, $\msp$ is not a maximal center.
\end{proof}

\begin{Rem}
In \cite[\S 4.3.2]{AZ16}, it is erroneously stated that the Kawamata blowup $\varphi \colon Y \to X$ at the $\frac{1}{4} (1, 1, 3)$ point $\msp \in X$ initiates a Sarkisov self-link in the case where $z t \in \msF_1$, $\alpha = 0$ and $s x \notin g_6$.
\end{Rem}

\subsection{Family \textnumero~59 and $\frac{1}{4} (1, 1, 3)$ points}

Let $X = X_{12, 14} \subset \mbP (1, 4, 4, 5, 6, 7)$ be a member of family \textnumero~59 and let $\msp \in X$ be a $\frac{1}{4} (1, 1, 3)$ point.
The generality assumption in \cite{OkadaI} imposed for $\msp \in X$ is the nonexistence of a WCI curve of type $(1, 4, 6, 7)$ on $X$ passing through $\msp$.
We assume that there is a WCI curve of type $(1, 4, 6, 7)$ passing through $\msp$.
We may assume that $\msp = \msp_z$.

\begin{Lem}
The defining polynomials of $X$ can be written as
\[
\begin{split}
\msF_1 &= z^2 y + z f_8 + w s + t^2 + f_{12}, \\
\msF_2 &= z^2 t + z g_{10} + w^2 + g_{14},
\end{split}
\]
where $f_i, g_i \in (x, y) \subset \mbC [x, y, s, t, w]$.
\end{Lem}

\begin{proof}
By the quasismoothness of $X$ at $\msp$, we have $z^2 y \in \msF_1$ and $z^2 t \in \msF_2$.
Moreover, we have $w s, t^2 \in \msF_1$ and $w^2 \in \msF_2$.
By replacing coordinates, we may assume that $\msF_1$ and $\msF_2$ are as in the statement for some $f_i, g_i \in \mbC [x, y, z, t, w]$ with $w s, t^2 \notin f_{12}$ and $w^2 \notin g_{14}$.
It is easy to see that $f_{12}, g_{14} \in (x, y)$ and it remains to show that $g_{10} \in (x, y)$.
By assumption, there is a WCI curve $\Gamma$ of type $(1, 4, 6, 7)$ passing through $\msp$.
We see that $\Gamma = (x = y = t = w = 0)$, and $X$ contains $\Gamma$ if and only if $s^2 \notin g_{10}$.
Hence $s^2 \notin g_{10}$ and we have $g_{10} \in (x, y)$.
\end{proof}

We can choose $x, s, w$ as local orbifold coordinates of $X$ at $\msp$ and $\ord_E (x, s, w) = \frac{1}{4} (1, 1, 3)$.
By the equations $\msF_1 = \msF_2 = 0$, we have $\ord_E (y, t) = \frac{1}{4} (4, 6)$.
We set $S \coloneq (x = 0)_X$ and let $T$ be a general member of $|\mfm_{\msp} (4A)|$ so that $\tilde{S} \sim \varphi^*A - \frac{1}{4} E = B$ and $\tilde{T} \sim 4 \varphi^*A - \frac{4}{4} E = 4 B$.
Note that $T$ is a normal surface by Lemma~\ref{lem:normqhyp}.

\begin{Prop}
\label{prop:No59excl4}
The support of $\tilde{S}|_{\tilde{T}}$ consists of two curves whose intersection matrix is negative-definite.
In particular, $\msp$ is not a maximal center.
\end{Prop}

\begin{proof}
We have
\[
S \cap T = (x = y = w s + t^2 = z^2 t + w^2 = 0).
\]
On the open set $U_z = (z \ne 0)$, $S \cap T$ is isomorphic to the $\bmu_4$-quotient of the affine curve
\[
(w s + t^2 = t + w^2 = 0) \subset \mbA^3_{s, t, w},
\]
which is then isomorphic to $(w (s + w^3) = 0)) \subset \mbA^2$.
The curve $(w = 0)$ corresponds to $\Gamma = (x = y = t = w = 0) \subset X$ and the remaining curve $s + w^3 = 0$ is irreducible and reduced.
It follows that $S|_T = \Gamma + \Delta$, where $\Delta$ is an irreducible and reduced curve since $S \cap T \cap (z = 0) = \{\msp_s\}$.
We see that $S \cap T$ is not quasismooth at $\msp$ and $\msp_s$.
Then, $\Gamma$ intersects $\Delta$ at $\msp_s$ and the local intersection multiplicity $(\Gamma \cdot \Delta)_{\msp_s}$ at $\msp_s$ is at least $1/5$.
This shows that $m \coloneq (\tilde{\Gamma} \cdot \tilde{\Delta}) \ge 1/5$.
We have $(A \cdot \Gamma) = 1/20$ and $(E \cdot \tilde{\Gamma}) = 1$ since $\tilde{\Gamma}$ intersects $E$ transversally at one smooth point.
It follows that $(B \cdot \tilde{\Gamma}) = -1/5$.
By taking intersection numbers of $B$ and $\tilde{S}|_{\tilde{T}} = \tilde{\Gamma} + \tilde{\Delta}$, we have $(B \cdot \tilde{\Delta}) = 1/15$.
By taking intersection numbers of $B|_{\tilde{T}} = \tilde{S}|_{\tilde{T}} = \tilde{\Gamma} + \tilde{\Delta}$ and $\tilde{\Gamma}$, $\tilde{\Delta}$, we have
\[
\begin{pmatrix}
(\tilde{\Gamma}^2) & (\tilde{\Gamma} \cdot \tilde{\Delta}) \\
(\tilde{\Gamma} \cdot \tilde{\Delta}) & (\tilde{\Delta}^2)
\end{pmatrix} =
\begin{pmatrix}
- m - \frac{1}{5} & m \\
m & - m + \frac{1}{15}
\end{pmatrix},
\]
which is negative-definite since $m \ge 1/5$.
By Lemma~\ref{lem:exclnegdef}, $\msp$ is not a maximal center.
\end{proof}

\subsection{The family \textnumero~71 and $\frac{1}{4} (1, 1, 3)$ points}

Let $X = X_{14, 16} \subset \mbP (1, 4, 5, 6, 7, 8)$ be a member of family \textnumero~71 and let $\msp$ be a $\frac{1}{4} (1, 1, 3)$ point.
The generality assumption in \cite{OkadaI} imposed for $\msp \in X$ is the nonexistence of a WCI curve of type $(1, 7, 8, 10)$ on $X$ passing through $\msp$.
We assume that there is a WCI curve of type $(1, 7, 8, 10)$ passing through $\msp$.
Replacing coordinates, we may assume that $\msp = \msp_y$.

\begin{Lem}
\label{lem:No71excleq}
The defining polynomials of $X$ can be written as
\[
\begin{split}
\msF_1 &= y^2 s + y (z^2 + f_{10}) + w s + t^2 +f_{14}, \\
\msF_2 &= y^2 w + y (s^2 + \theta z t + g_{12}) + z^2 s + w^2 + g_{16},
\end{split}
\]
where $\theta \in \mbC$ and $f_{10}, f_{14}, g_{12}, g_{16} \in \mbC [x, z, s, t, w]$ such that $f_{10}, f_{14}, g_{12}, g_{16}$ are all divisible by $x$.
\end{Lem}

\begin{proof}
By the quasismoothness of $X$ at $\msp$, we have $y^2 s \in \msF_1$ and $y^2 w \in \msF_2$.
We can write
\[
\begin{split}
\msF_1 &= y^2 s + y f'_{10} + f'_{14}, \\
\msF_2 &= y^2 w + y g'_{12} + g'_{16},
\end{split}
\]
where $f'_i, g'_i \in \mbC [x, z, s, t, w]$.
The monomial(s) $z^2$ (resp.\ $w s, t^2$, resp.\ $z t, s^2$, resp.\ $w^2, z^2 s$) are all the monomial(s) of degree $10$ (resp.\ $14$, resp.\ $12$, resp.\ $16$) in variables $z, s, t, w$.
Hence we can write $f'_{10} = \alpha z^2 + f_{10}$, $f'_{14} = \lambda w s + \gamma t^2$, $g'_{12} = \theta z t + \delta s^2 + g_{12}$ and $g'_{16} = \mu w^2 + \varepsilon z^2 s + g_{16}$, where $\alpha, \dots, \varepsilon, \lambda, \mu, \theta \in \mbC$ and $f_{10}, f_{14}, g_{12}, g_{16}$ are all divisible by $x$, that is,
\[
\begin{split}
\msF_1 &= y^2 s + y (\alpha z^2 + f_{10}) + \beta w s + \gamma t^2 + f_{14}, \\
\msF_2 &= y^2 w + y (\theta z t + \delta s^2 + g_{12}) + \varepsilon w^2 + \eta z^2 s + g_{16}.
\end{split}
\]
By the quasismoothness of $X$ at $\msp_z$ (resp.\ $\msp_s$), we have $\alpha \ne 0$ (resp.\ $\beta \ne 0$ and $\delta \ne 0$).
Moreover, we have $\gamma \ne 0$ and $\varepsilon \ne 0$ also by the quasismoothness of $X$.
Rescaling coordinates, we may assume that $\alpha = \beta = \gamma = \delta = \varepsilon = 1$.
By assumption, there is a WCI curve $\Gamma$ of type $(1, 7, 8, 10)$ passing through $\msp$.
Such a curve $\Gamma$ is contained in $(x = t = w = 0)$ and we have
\[
(x = t = w = 0)_X = (x = t = w = y (y s + z^2) = s (y s + \eta z^2) = 0).
\]
Thus, we have $\Gamma = (x = t = w = y s + z^2 = 0)$ and $\eta = 1$.
Therefore we obtain the desired equations.
\end{proof}

We choose homogeneous coordinates as in Lemma~\ref{lem:No71excleq}.
We can choose $x, z, t$ as local orbifold coordinates of $X$ at $\msp$ and we have $\ord_E (x, z, t) = \frac{1}{4} (1, 1, 3)$.
By the equations $\msF_1 = \msF_2 = 0$, we have $\ord_E (s, w) = \frac{1}{4} (2, 4)$.
Note that by the equation $\msF_1 = 0$, we have
\[
y (y s + z^2) = y f_{10} + w s + t^2 + f_{14}
\]
and the right hand side has order at least $\frac{6}{4}$ along $E$.
This shows that $\ord_E (y s + z^2) = \frac{6}{4}$.
We define $S \coloneq (x = 0)_X$ and $T \coloneq (y s + z^2 = 0)_X$.
We have $\tilde{S} \sim \varphi^*A - \frac{1}{4} E = B$ and $\tilde{T} \sim 10 \varphi^*A - \frac{6}{4} E = 6 B + E$.

\begin{Prop}
\label{prop:No71excl4}
The support of $\tilde{S}|_{\tilde{T}}$ consists of irreducible and reduced curves whose intersection matrix is negative-definite.
In particular, $\msp$ is not a maximal center.
\end{Prop}

\begin{proof}
We set 
\[
\Gamma \coloneq (x = t = w = y s + z^2 = 0) \subset X,
\]
which is a WCI curve of type $(1, 7, 8, 10)$ passing through $\msp$.
We have
\[
S \cap T = (x = y s + z^2 = w s + t^2 = y^2 w + \theta y z t + w^2 = 0).
\] 

Suppose that $\theta \ne 0$.
On the open subset $U_s$, $S \cap T$ is isomorphic to the $\bmu_6$-quotient of the complete intersection $C$ in $\mbA_{y, z, t, w}^4$ defined by the equations
\[
y + z^2 = w + t^2 = y^2 w + \theta y z t + w^2 = 0.
\]
Eliminating the variables $y$ and $w$, $C$ is isomorphic to the plane curve in $\mbA^2_{z, t}$ defined by the equation
\[
t (t^3 - z^4 t - \theta z^3) = 0.
\]
The curve $\Gamma$ corresponds to the plane curve $(t = 0)$ and the polynomial $t^3 - z^4 t - \theta z^3$ is irreducible for any $\theta \ne 0$.
This shows that $S|_T = \Gamma + \Delta$, where $\Delta$ is an irreducible and reduced curve, since $S \cap T \cap (s = 0)$ is a finite set of points.
We have $(A \cdot \Gamma) = 1/12$ and $(E \cdot \tilde{\Gamma}) = 1$ since $\tilde{\Gamma}$ and $E$ intersect transversally at $1$ smooth point.
It follows that $(B \cdot \tilde{\Gamma}) = -1/6$.
By taking the intersection number of $B|_{\tilde{T}} \sim \tilde{S}|_{\tilde{T}} = \tilde{\Gamma} + \tilde{\Delta}$ and $B|_{\tilde{T}}$, we have $(B \cdot \tilde{\Delta}) = 0$.
We set $m \coloneq (\tilde{\Gamma} \cdot \tilde{\Delta})$.
We see that $S \cap T$ is not quasismooth at $\msp_s \in \Gamma$ and $\Gamma$ is quasismooth at $\msp_s$.
This shows that $\msp_s \in \Gamma \cap \Delta$, and hence $m > 0$.
By taking intersection numbers of $\tilde{S}|_{\tilde{T}} = \tilde{\Gamma} + \tilde{\Delta}$ and $\tilde{\Gamma}$, $\tilde{\Delta}$, we have
\[
\begin{pmatrix}
(\tilde{\Gamma}^2) & (\tilde{\Gamma} \cdot \tilde{\Delta}) \\
(\tilde{\Gamma} \cdot \tilde{\Delta}) & (\tilde{\Delta}^2)
\end{pmatrix} =
\begin{pmatrix}
- m - \frac{1}{6} & m \\
m & -m
\end{pmatrix}
\]
which is negative-definite since $m > 0$.
By Lemma~\ref{lem:exclnegdef}, $\msp$ is not a maximal center.

Suppose that $\theta = 0$.
Then, we have $S|_T = 2 \Gamma + \Delta$, where
\[
\Delta \coloneq (x = y s + z^2 = w s + t^2 = y^2 + w = 0).
\]
On the open set $U_z \coloneq (z \ne 0)_X$, the curve $\Delta$ is the $\bmu_5$-quotient of the complete intersection $C$ in $\mbA^4_{y, s, t, w}$ defined by the equations
\[
y s + 1 = w s + t^2 = y^2 + w = 0.
\]
Eliminating the variable $w$, the curve $C$ is isomorphic to $(y s + 1 = - y + t^2 = 0) \subset \mbA^3_{y, s, t}$, and then to $(s t^2 + 1 = 0) \subset \mbA^2_{s, t}$.
It follows that $\Delta$ is irreducible and reduced since $\Delta \cap (z = 0)$ is a finite set of points.
We set $m \coloneq (\tilde{\Gamma} \cdot \tilde{\Delta})$.
We see that $\msp_s \in \Gamma \cap \Delta$ and $\omult_{\msp_s} (\Delta) = 2$, which shows that the local intersection multiplicity of $\Gamma$ and $\Delta$ at $\msp_s$ is at least $\omult_{\msp} (\Gamma) \omult_{\msp} (\Delta)/6 \ge 1/3$.
In particular, we have $m \ge 1/3$.
We have $(B \cdot \tilde{\Gamma}) = -1/6$ as above and $(B \cdot \tilde{\Delta}) = (B|_{\tilde{T}} \cdot \tilde{S}|_{\tilde{T}}) - 2 (B \cdot \tilde{\Gamma}) = 1/6$.
By taking intersection numbers of $\tilde{S}|_{\tilde{T}} = 2 \tilde{\Gamma} + \tilde{\Delta}$ and $\tilde{\Gamma}$, $\tilde{\Delta}$, we have
\[
\begin{pmatrix}
(\tilde{\Gamma}^2) & (\tilde{\Gamma} \cdot \tilde{\Delta}) \\
(\tilde{\Gamma} \cdot \tilde{\Delta}) & (\tilde{\Delta}^2)
\end{pmatrix} =
\begin{pmatrix}
- \frac{1}{2} m - \frac{1}{12} & m \\
m & - 2 m + \frac{1}{6}
\end{pmatrix}
\]
which is negative-definite since $m \ge 1/3$.
By Lemma~\ref{lem:exclnegdef}, $\msp$ is not a maximal center.
\end{proof}

\begin{Rem}
We are unable to follow the arguments given in \cite[Pages~437-438]{AZ16} for a $\frac{1}{4} (1, 1, 3)$ point $\msp$ of a member $X$ of family \textnumero~71.

Firstly, in \cite{AZ16}, it is proved that the $2$-ray game initiated by $\varphi \colon Y \to X$ remains the same regardless of the generality assumption that $X$ does not contain a WCI curve of type $(1, 7, 8, 10)$.
However, this is impossible, that is, the $2$-ray game changes as we drop the generality assumption, since the cone $\bNE (Y)$ (and also the nef cone of $Y$) changes as we drop the assumption.
 
Secondly, it is stated in \cite{AZ16} that the first four steps of the $2$-ray game for $T''$ restrict to an isomorphism of $Y$, and the fifth step restricts to a (log) flip of $Y$.
If this is indeed correct, then the divisor $N \coloneq 10 \varphi^*A - \frac{6}{4} E$ is nef and not ample.
However, $N$ cannot be ample as explained below.
If $X$ contains a WCI curve $\Gamma$ of type $(1, 7, 8, 10)$ passing through $\msp$, we have 
\[
(N \cdot \tilde{\Gamma}) = \frac{10}{12} - \frac{6}{4} < 0
\]
by the computation given in the proof of Proposition~\ref{prop:No71excl4}.
If $X$ does not contain such a WCI curve, then $\{x, t, v, w\}$ isolates $\msp$ and $6 \varphi^*A - \frac{2}{4} E$ is a nef divisor, which shows that $N$ does not belong to the nef cone of $Y$.
\end{Rem}

\section{Local alpha invariants at smooth points}
\label{sec:alphasmpt}

The aim of this section is to prove the following.

\begin{Thm}
\label{thm:alphasmpt}
Let $X$ be a member of family $\msi$, where $\msi \in \msI_{\br}^* \setminus \{60\}$, and let $\msp$ be a smooth point of $X$.
Then, $\msp$ is not a maximal center and we have $\alpha_{\msp} (X) \ge 1/2$.
\end{Thm}

The first assertion follows from \cite{OkadaI} (see also the explanations given at the beginning of \S \ref{sec:pfbirrig}), hence we need to show the latter assertion concerning local alpha invariants, which follows from Propositions~\ref{prop:mostsmptalpha} and \ref{prop:No8smptalpha} below.

\subsection{Families other than \textnumero~8 and 60}

Let $X = X_{d_1, d_2} \subset \mbP (a_0, \dots, a_5)$ be a member of family \textnumero~$\msi$, where $\msi \in \msI^*_{\mathrm{br}} \setminus \{8, 60\}$.
We assume that $a_0 \le \cdots \le a_5$.
Note that $a_0 = 1$.

\begin{Lem}
\label{lem:smptHxsupp}
For $\msi \in \msI^*_{\mathrm{br}} \setminus \{8, 60\}$, let $e_{\msi}$ be the number defined as follows:
\[
e_{\msi} \coloneq
\begin{cases}
2, & \text{if $\msi = 14, 20, 24$}, \\
4, & \text{if $\msi = 31, 37, 45, 47, 51, 59$}, \\
6, & \text{if $\msi = 64, 75$}, \\
8, & \text{if $\msi = 78$}, \\
10, & \text{if $\msi = 71, 76$}, \\
15, & \text{if $\msi = 84$}, \\
18, & \text{if $\msi = 85$}.
\end{cases}
\]
Then, for any smooth point $\msp$ of $X$ that is contained in $H_x$, there is an effective Weil divisor $S_{\msp} \in |e_{\msi} A|$ such that $S_{\msp}$ passes through $\msp$ and it does not contain $H_x$ in its support.
\end{Lem} 

\begin{proof}
There are two linearly independent homogeneous polynomials of degree $e_{\msi}$ in $\mbC [y, z, t, w]$.
For example, if $\msi = 14$ (resp.\ $\msi = 71$, resp.\ $\msi = 85$), then we can take $y$ and $z$ (resp.\ $z^2$ and $y t$, resp.\ $y t$ and $z^2$).
Let $h_1, h_2 \in \mbC [y, z, t, w]$ be such polynomials and let $\msp$ be a smooth point of $X$ contained in $H_x$.
Then, we can choose $\lambda_1, \lambda_2 \in \mbC$ such that at least one of $\lambda_i$ is nonzero and $\lambda_1 h_1 + \lambda_2 h_2$ vanishes at $\msp$.
The assertion follows immediately by setting $S_{\msp} \coloneq (\lambda_1 h_1 + \lambda_2 h_2 = 0)_X \sim e_{\msi} A$.
\end{proof}

\begin{Lem}
\label{lem:smptHxmult}
We have
\[
\mult_{\msp} (H_x) \le 2
\]
for any smooth point $\msp$ of $X$.
\end{Lem}

\begin{proof}
Let $\msp \in X$ be a smooth point of $X$.
We may assume that $\msp \in H_x$.
Let $e_{\msi}$ and $S \coloneq S_{\msp} \in |e_{\msi} A|$ be as in Lemma~\ref{lem:smptHxsupp}.
Then $H_x \cdot S$ is an effective $1$-cycle.

Suppose that $\msi \ne 14$ (resp.\ $\msi = 14$).
Then there is a $\msp$-isolating set of maximum degree $a_4$ (resp.\ $2$).
By Lemma~\ref{lem:isolomult}, we have
\[
\mult_{\msp} (H_x) \le (H_x \cdot S \cdot T) =
\begin{cases}
e_{\msi} a_4 (A^3), & \text{if $\msi \ne 14$}, \\
2 e_{14} (A^3) = 2, & \text{if $\msi = 14$}.
\end{cases}
\]
Moreover, we have $e_{\msi} a_4 (A^3) < 3$ for $\msi \ne 14$.
This shows that $\mult_{\msp} (H_x) \le 2$.
\end{proof}

\begin{Prop}
\label{prop:mostsmptalpha}
Let $X$ be a member of family \textnumero~$\msi$, where $\msi \in \msI^*_{\mathrm{br}} \setminus \{8, 60\}$.
Then, for any smooth point $\msp$ of $X$, we have $\alpha_{\msp} (X) \ge 1/2$.
\end{Prop}

\begin{proof}
Let $D \in |A|_{\mbQ}$ be an irreducible $\mbQ$-divisor passing through $\msp$.

Suppose first that $\msp \in U_x$.
We may assume $\msp = \msp_x$ by replacing coordinates.
Then, either $\frac{1}{a_1} H_y$ or $\frac{1}{a_2} H_z$ does not coincide with $D$.
In particular, there is an effective divisor $S \sim b A$ with $b \le a_2$ such that $D$ and $S$ do not share a common component and $\msp \in \Supp (S)$.
If $\msi \ne 14$ (resp.\ $\msi = 14$), then there is a $\msp$-isolating class of maximum degree $a_4$ (resp.\ $2$).
By Lemma~\ref{lem:isolomult}, we have
\[
\mult_{\msp} (D) \le (D \cdot S \cdot T) \le 
\begin{cases}
a_2 a_4 (A^3) \le 2, & \text{if $\msi \ne 14$}, \\
4 (A^3) = 2, & \text{if $\msi = 14$}.
\end{cases}
\]
This implies $\lct_{\msp} (X; D) \ge 1/2$.

Suppose next that $\msp \in H_x$.
By Lemma~\ref{lem:smptHxmult}, we have $\lct_{\msp} (X; H_x) \ge 1/2$.
We may assume that $D \ne H_x$.
As above, if $\msi \ne 14$ (resp.\ $\msi = 14$), there is a $\msp$-isolating set of maximum degree $a_4$ (resp.\ $2$).
By Lemma~\ref{lem:isolomult}, we have
\[
\mult_{\msp} (D) \le (D \cdot H_x \cdot T) \le 
\begin{cases}
a_4 (A^3) \le 1, & \text{if $\msi \ne 14$}, \\
2 (A^3) = 1, & \text{if $\msi = 14$}.
\end{cases}
\]
This implies $\lct_{\msp} (X; D) \ge 1$ and the proof is complete.
\end{proof}

\subsection{Family \textnumero~8}

Let $X = X_{4, 6} \subset \mbP (1, 1, 2, 2, 2, 3)$ be a member of family \textnumero~$8$.

\begin{Lem}
\label{lem:No8isolsm}
Let $\msp \in X$ be a smooth point.
Then, there is a $\msp$-isolating set of maximum degree $2$.
\end{Lem}

\begin{proof}
Suppose that $\msp \notin (x = y = 0)$.
Then, after replacing homogeneous coordinates, we may assume that $\msp = \msp_x$.
We have $w^2 \in \msF_2$ by the quasismoothness of $X$, and hence $\{y, z, s, t\}$ isolates $\msp$.

Suppose that $\msp \in (x = y = 0)_X$.
Then, after replacing $z, s, t$ and $w$, we may assume $\msp = (0\!:\!0\!:\!0\!:\!0\!:\!1\!:\!1)$.
Note that $\msF_2 (0, 0, 0, 0, t, w) = \alpha (w^2 - t^3)$ for some nonzero $\alpha \in \mbC$ because otherwise $X$ is not quasismooth at $\msp$.
In this case, $\{x, y, z, s\}$ isolates $\msp$.
\end{proof}

\begin{Lem}
\label{lem:No8Lxy}
The $1$-dimensional scheme $(x = y = 0)_X$ is either an irreducible smooth curve or the union of two irreducible smooth curves intersecting transversally at $1$ smooth point of $X$.
\end{Lem}

\begin{proof}
We can write
\[
\begin{split}
\msF_1 &= q (z, s, t) + \msG_1, \\
\msF_2 &= c (z, s, t) + w^2 + \msG_2,
\end{split}
\]
where $q, c$ are quadratic, cubic forms in variables $z, s, t$, respectively, and $\msG_1, \msG_2 \in (x, y) \subset \mbC [x, y, z, s, t, w]$.
The scheme $(x = y = 0)_X$ is isomorphic to the scheme
\[
\begin{split}
& (q (z, s, t) = c (z, s, t) + \tilde{w} = 0) \subset \mbP (1, 1, 1, 3) \\
\cong & (q (z, s, t) = 0) \subset \mbP^2,
\end{split}
\]
The assertion follows immediately since the rank of the quadratic form $q$ is either $2$ or $3$ by the quasismoothness of $X$.
\end{proof}

\begin{Lem}
\label{lem:No8Hsmpt}
Let $H$ be a member of the linear system $|A|$.
Then, $\lct_{\msp} (X; H) \ge 1/2$ for any smooth point $\msp$ of $X$.
Moreover, $H$ is quasismooth at any $\frac{1}{2} (1, 1, 1)$ point of $X$.
\end{Lem}

\begin{proof}
Let $\msp$ be a smooth point of $X$.
We may assume that $\msp \in H$ because otherwise there is nothing to prove.

We claim that if there is a member $S \in |\mfm_{\msp}^2 (2A)|$ other than $2 H$, then $\mult_{\msp} (H) \le 2$ and hence $\lct_{\msp} (X; H) \ge 1/2$.
Indeed, if there exists a member $S \in |\mfm_{\msp}^2 (2 A)| \setminus \{2H\}$, then we have $\mult_{\msp} (H) \le 2$ by Lemma~\ref{lem:No8isolsm} and \ref{lem:isolomult}.

Suppose that $\msp \in (x = y = 0)_X$.
Let $H' \in |A|$ be a general member so that $H \cap H' = (x = y = 0)_X$.
By Lemma~\ref{lem:No8Lxy}, we have $\mult_{\msp} (H) \le 2$, which implies $\lct_{\msp} (X; H) \ge 1/2$.

Suppose that $\msp \notin (x = y = 0)_X$.
Replacing coordinates, we may assume that $\msp = \msp_x$ and $H = H_y$.
We can write
\[
\begin{split}
\msF_1 &= \alpha_1 x^3 y + \beta_1 x^2 z + \gamma_1 x^2 s + \delta_1 x^2 t + \varepsilon_1 x w + \msG_1, \\
\msF_2 &= \alpha_2 x^5 y + \beta_2 x^4 z + \gamma_2 x^4 s + \delta_2 x^4 t + \varepsilon_2 x^3 w + \msG_2, \\
\end{split}
\]
where $\alpha_i, \dots, \varepsilon_i \in \mbC$ for $i = 1, 2$ and $\msG_1, \msG_2 \in (y, z, s, t, w)^2 \subset \mbC [x, y, z, s, t, w]$.
The matrix 
\[
\begin{pmatrix}
\alpha_1 & \beta_1 & \gamma_1 & \delta_1 & \varepsilon_1 \\
\alpha_2 & \beta_2 & \gamma_2 & \delta_2 & \varepsilon_2 \\
\end{pmatrix}
\]
is of rank $2$ by the quasismoothness of $X$ at $\msp$.

Suppose that $\varepsilon_1 = 0$.
If in addition at least one of $\beta_1, \gamma_1$ and $\delta_1$ is nonzero, then the divisor $(\alpha_1 x y +\beta_1 z + \gamma_1 s + \delta_1 t = 0)_X$ is a member of $|\mfm_{\msp}^2 (2A)|$, and we have $\mult_{\msp} (H) \ge 2$ by the claim.
We proceed with the proof assuming that $\beta_1 = \gamma_1 = \delta_1 = 0$.
We have $\alpha_1 \ne 0$.
If $\varepsilon_2 = 0$, then at least one of $\beta_2, \gamma_2$ and $\delta_2$ is nonzero by the quasismoothness of $X$ at $\msp$.
Then the divisor $(\alpha_2 x y +\beta_2 z + \gamma_2 s + \delta_2 t = 0)_X$ is a member of $|\mfm_{\msp}^2(2A)|$, and we have $\mult_{\msp} (H) \ge 2$ by the claim.
We assume that $\varepsilon_2 \ne 0$.
In this case, we can take $z, s, t$ as a system of local coordinates of $X$ at $\msp$.
We have $\mult_{\msp} (H) = 2$ since $\msF_1 (0, 0, z, s, t, w) = \msG_1 (0, 0, z, s, t, w)$ is a nonzero quadratic form in variables $z, s, t$.
Thus $\mult_{\msp} (H) \ge 2$, and hence $\lct_{\msp} (X; H) \ge 1/2$ in this case.

Suppose that $\varepsilon_1 \ne 0$.
Replacing $w$, we may assume that $\varepsilon_1 =1$ and $\alpha_1 = \beta_1 = \gamma_1 = \delta_1 = 0$.
We may also assume that $\varepsilon_2 = 0$ by replacing $\msF_2$ by $\msF_2 - \varepsilon_2 x^2 \msF_1$.
If at least one of $\beta_2, \gamma_2$ and $\delta_2$ is nonzero, then the divisor $(\alpha_2 x y +\beta_2 z + \gamma_2 s + \delta_2 t = 0)_X$ is a member of $|\mfm_{\msp}^2(2A)|$, and we have $\mult_{\msp} (H) \ge 2$ by the claim.
We proceed with the proof assuming $\beta_2 = \gamma_2 = \delta_2 = 0$.
In this case we have $\alpha_2 \ne 0$ and we may assume $\alpha_2 = 1$.
Replacing $x \mapsto x - \theta y$ for a suitable $\theta \in \mbC$, we may assume that $\msG_1$ does not involve the variable $w$, that is, $\msG_1 = \msG_1 (x, y, z, s, t)$.
We set $g = g (1, y, z, s, t) \coloneq - \msG_1 (1, y, z, s, t)$.
Then, in a neighborhood of $\msp$, $X$ is isomorphic to the hypersurface in $\mbA^4_{y, z, s, t}$ defined by the equation $f = 0$, where we define $f = f (y, z, s, t) \coloneq \msF_1 (1, y, z, s, t, g)$.
Filtering off terms divisible by $y$, the equation $f = 0$ can be written as
\[
(1 + \cdots) y + f (0, z, s, t) = 0,
\]
where the omitted terms are nonconstant monomials in variables $y, z, s, t$.
We can write $f (0, z, s, t) = \sum_i f_i$, where $f_i = f_i (z, s, t)$ is a homogeneous polynomial of degree $i$ (with respect to $\wt (z, s, t) = (1, 1, 1)$). 
By construction, we have $f_0 = f_1 = 0$.
If $f_2 \ne 0$, then $\mult_{\msp} (H) = 2$, and hence $\lct_{\msp} (X; H) \ge 1/2$.
Suppose that $f_2 = 0$.
Note that we have $f_3 = \msF_2 (0, 0, z, s, t, 0)$ and, in particular, $f_3$ cannot be the cube of a linear form in variables $z, s, t$.
Let $\varphi \colon Y \to X$ be the blowup of $X$ at $\msp$ with exceptional divisor $E \cong \mbP^2$ and let $\tilde{H}$ be the strict transform of $H$ on $Y$.
Then $\tilde{H}|_E$ is isomorphic to the plane cubic curve defined by $f_3 = 0$ in $\mbP^2$ and it is not a triple line.
By \cite[Lemmas~3.4 and 3.25]{KOWalpha}, we have $\lct_{\msp} (X; H) \ge 1/2$ and the proof of the first assertion is complete.

It is easy to see that $H$ is quasismooth at any $\frac{1}{2} (1, 1, 1)$ point.
\end{proof}

\begin{Prop}
\label{prop:No8smptalpha}
Let $\msp$ be a smooth point of $X$.
Then, $\alpha_{\msp} (X) \ge 1/2$.
\end{Prop}

\begin{proof}
Let $D \in |A|_{\mbQ}$ be an irreducible $\mbQ$-divisor passing through $\msp$.
We can take a member $H \in |A|$ that passes through $\msp$.
We may assume that $D \ne H$ because otherwise $\lct_{\msp} (X; D) = \lct_{\msp} (X; H) \ge 1/2$ by Lemma~\ref{lem:No8Hsmpt}.
By Lemmas~\ref{lem:No8isolsm} and \ref{lem:isolomult}, we have $\mult_{\msp} (D) \le 2$.
This shows that $\lct_{\msp} (X; D) \ge 1/2$ and we obtain the inequality $\alpha_{\msp} (X) \ge 1/2$.
\end{proof}

\section{Local alpha invariants at singular points}
\label{sec:alphasingpt}

This section is devoted to the proof of the following.

\begin{Thm}
\label{thm:alphasingpt}
Let $X$ be a member of family \textnumero~$\msi$, where $\msi \in \msI_{\br}^* \setminus \{60\}$ and let $\msp \in X$ be a singular point with the mark $\alpha$ as a subscript in the fourth column of \emph{Table~\ref{table:BRcodim2}}.
Then, $\msp$ is not a maximal center and $\alpha_{\msp} (X) \ge 1/2$.
\end{Thm}

This follows from Propositions~\ref{prop:alphaSingPt1} and \ref{prop:alphaSingPt2} below.

\subsection{Local alpha invariants at singular points, Part I}
\label{sec:alphasingptI}

Let $X$ be a member of family \textnumero~$\msi$ and $\msp \in X$ a singular point listed in \emph{Table~\ref{table:singptNE}}.
We denote by $\varphi \colon Y \to X$ the Kawamata blowup of $X$ at $\msp$ with exceptional divisor $E$.

\begin{Prop}
\label{prop:alphaSingPt1}
Let $X$ be a member of family \textnumero~$\msi$ and $\msp \in X$ a singular point listed in \emph{Table~\ref{table:singptNE}} except for $\msi = 47$ and $\msp$ is of type $\frac{1}{3} (1, 1, 2)$.
Then the following assertions hold.
\begin{enumerate}
\item $\msp$ is not a maximal center.
\item We have  $B^2 \notin \bNE (Y)$ and there is a prime divisor $S$ on $X$ such that $\tilde{S} \sim m B$ for some positive integer $m$.
\end{enumerate}
In particular, $\alpha_{\msp} (X) \ge 1$.
\end{Prop}

\begin{proof}
Let $\msp \in X$ be as in the statement.
If $\msi = 76$ and $\msp \in X$ is of type $\frac{1}{5} (1, 1, 4)$, then the assertions (1) and $B^2 \notin \bNE (Y)$ are proved in \cite[Theorem~8.9]{OkadaI}.
Otherwise the assertion (1) and the existence of a test class $M$ on $Y$ such that $(M \cdot B^2) \le 0$ are proved in \cite[Proposition~8.4]{OkadaI}.
A test class on $Y$ is nothing but a nonzero nef divisor on $Y$ (see \cite[Definition~3.6]{OkadaI}), hence this implies $B^2 \notin \bNE (Y)$.

Except for the cases where $\msi = 31$ and $\msp$ is of type $\frac{1}{3} (1, 1, 2)$, and $\msi = 47$ and $\msp$ is of type $\frac{1}{3} (1, 1, 2)$, we can choose $x$ as part of local orbifold coordinates of $X$ at $\msp$, which implies that $\tilde{H}_x \sim B$.

Suppose that $\msi = 47$ and $\msp$ is of type $\frac{1}{3} (1, 1, 2)$.
We may assume that $x$ is not a part of local orbifold coordinates of $X$ at $\msp$.
We can write the defining polynomials of $X$ at $\msp$ as
\[
\begin{split}
\msF_1 &= y^3 x + y^2 f_4 + y f_7 + f_{10}, \\
\msF_2 &= y^2 w + y g_9 + g_{12},
\end{split}
\]
where $f_i, g_i \in \mbC [x, z, s, t, w]$ are homogeneous polynomials of degree $i$ such that $z, s \notin f_4$.
We have
\[
(x = t = w = 0)_X = (x = t = w = g_{12} (0, z, s, 0, 0) = 0)
\]
and this shows that $X$ contains three WCI curves of type $(1, 4, 5, 6)$ corresponding to the solutions of the cubic polynomial $g_{12} (0, z, s, 0, 0) = 0$.
This contradicts the assumption that $X$ does not contain any WCI curve of type $(1, 4, 5, 6)$ passing through $\msp$.
Therefore, the existence of a prime $S$ on $X$ such that $\tilde{S} \sim m B$ for some $m > 0$ is proved in all the cases.

Finally, by \cite[Lemma~2.8]{KOWalpha}, we have $\alpha_{\msp} (X) \ge 1$.
\end{proof}

\begingroup
\renewcommand{\arraystretch}{1.2}
\begin{table}[htb]
\caption{Singular points for which $(-K_Y)^2 \notin \bNE (Y)$}
\label{table:singptNE}
\centering
\begin{tabular}{clc}
\toprule
\textnumero & $X_{d_1,d_2} \subset \mbP (a_0,\dots, a_5)$ & $\msp$ \\ 
\midrule
14 & $X_{6,6} \subset \mbP (1,2,2,2,3,3)$ & $\frac{1}{2} (1, 1, 1)$ \\
24 & $X_{6,10} \subset \mbP (1,2,2,3,4,5)$ & $\frac{1}{2} (1, 1, 1)$ \\
31 & $X_{8,10} \subset \mbP (1,2,3,4,4,5)$ & $\frac{1}{2} (1, 1, 1)$ \\
45 & $X_{10,12} \subset \mbP (1,2,4,5,5,6)$ & $\frac{1}{2} (1, 1, 1)$ \\
47 & $X_{10,12} \subset \mbP (1,3,4,4,5,6)$ & $\frac{1}{2} (1, 1, 1)$, $\frac{1}{4} (1, 1, 3)$ \\
51 & $X_{10,14} \subset \mbP (1,2,4,5,6,7)$ & $\frac{1}{2} (1, 1, 1)$ \\
59 & $X_{12,14} \subset \mbP (1,4,4,5,6,7)$ & $\frac{1}{2} (1, 1, 1)$, $\frac{1}{5} (1, 1, 4)$ \\
64 & $X_{12,16} \subset \mbP (1,2,5,6,7,8)$ & $\frac{1}{2} (1, 1, 1)$ \\
71 & $X_{14,16} \subset \mbP (1,4,5,6,7,8)$ & $\frac{1}{2} (1, 1, 1)$, $\frac{1}{6} (1, 1, 5)$ \\
75 & $X_{14,18} \subset \mbP (1,2,6,7,8,9)$ & $\frac{1}{2} (1, 1, 1)$, $\frac{1}{3} (1, 1, 2)$ \\
76 & $X_{12,20} \subset \mbP (1,4,5,6,7,10)$ & $\frac{1}{2} (1, 1, 1)$, $\frac{1}{5} (1, 1, 4)$ \\
78 & $X_{16,18} \subset \mbP (1,4,6,7,8,9)$ & any singular point \\
84 & $X_{18,30} \subset \mbP (1,6,8,9,10,15)$ & any singular point \\
85 & $X_{24,30} \subset \mbP (1,8,9,10,12,15)$ & any singular point \\
\bottomrule
\end{tabular}
\end{table}
\endgroup

\subsection{Local alpha invariants at singular points, Part II}
\label{sec:alphasingptII}

Let $X$ be a member of family \textnumero~$\msi$ and $\msp \in X$ a singular point listed in Table~\ref{table:singpttc}.

\begin{Lem}
\label{lem:lctHxSingPt2}
We have $\lct_{\msp} (X; H_x) \ge 1/2$.
\end{Lem}

\begin{proof}
The assertion is easily verified if $d_i \not\equiv 1 \pmod{r}$ for $i = 1, 2$, where $r$ is the index of the singularity $\msp \in X$.
Indeed, if this is the case, then we can choose $x$ as part of local orbifold coordinates of $X$ at $\msp$ and this immediately implies that $H_x$ is quasismooth at $\msp$ and $\lct_{\msp} (X; H_x) = 1$.
This condition is satisfied except when $\msi = 47$ and $\msp \in X$ is of type $\frac{1}{3} (1, 1, 2)$, and $\msi = 71$ and $\msp \in X$ is of type $\frac{1}{5} (1, 2, 3)$.

Suppose that $\msi = 47$ and $\msp \in X$ is of type $\frac{1}{3} (1, 1, 2)$.
We may assume $\msp = \msp_y$.
Then, $H_x$ is not quasismooth at $\msp$ if and only if $y^2 z, y^2 s \notin \msF_1$.
In this case, we can write the defining polynomials of $X$ as
\[
\begin{split}
\msF_1 &= z^3 x + z^2 f_4 + z f_7 + f_{10}, \\
\msF_2 &= z^2 w + z g_9 + g_{12},
\end{split}
\]
where $f_i, g_i \in \mbC [x, z, s, t, w]$ are homogeneous polynomials such that $z, s \notin f_4$.
We can choose $z, s, t$ as a local orbifold coordinates of $X$ at $\msp$.
By the quasismoothness of $X$, we have $t^3 \in f_{10}$ and this implies that $\omult_{\msp} (H_x) = 2$. 
Hence, we have $\lct_{\msp} (X; H_x) \ge 1/2$.

Suppose that $\msi = 71$ and $\msp = \msp_z \in X$ is of type $\frac{1}{5} (1, 2, 3)$.
Then, $H_x$ is not quasismooth at $\msp$ if and only if $z^2 s \notin \msF_2$.
In this case, we can take $s, t, w$ as local orbifold coordinates of $X$ at $\msp$ and we have $\omult_{\msp} (H_x) = 2$ by the existence of $w^2 \in \msF_2$.
Thus, we have $\lct_{\msp} (X; H_x) \ge 1/2$ and the proof is complete.
\end{proof}

\begin{Prop}
\label{prop:alphaSingPt2}
The point $\msp$ is not a maximal center and we have $\alpha_{\msp} (X) \ge 1/2$.
\end{Prop}

\begin{proof}
The first assertion is proved in \cite[Propositions~8.4 and 8.11]{OkadaI} under suitable generality assumptions (that are made explicit in \cite[Section 9]{OkadaI}), and the complete proof is given in \S~\ref{sec:pfbirrig}.

We prove the latter assertion.
Let $D \in |A|_{\mbQ}$ be an irreducible $\mbQ$-divisor other than $H_x$.
Then, $D \cdot H_x$ is an effective $1$-cycle.

Let $\Lambda$ be the set of coordinates given in the $4$th column of Table~\ref{table:singpttc}.
It is straightforward to see that the common zero locus of the coordinates in $\Lambda$ is a finite set of points including $\msp$.
This shows that $\Lambda$ is a $\msp$-isolating set of maximum degree $e \coloneq \max \{\, \deg v \mid v \in \Lambda \,\}$, which is given in the $5$th column of Table~\ref{table:singpttc}.
By Lemma~\ref{lem:isolomult}, we have 
\[
\omult_{\msp} (D) \le r (D \cdot H_x \cdot e A) = r e (A^3) \le 2,
\]
where $r$ is the index of the singularity $\msp \in X$ and the last inequality follows from direct computations.
This shows that $\lct_{\msp} (X; D) \ge 1/2$ and, together with Lemma~\ref{lem:lctHxSingPt2}, we conclude $\alpha_{\msp} (X) \ge 1/2$.
\end{proof}

\begingroup
\renewcommand{\arraystretch}{1.2}
\begin{table}[htb]
\caption{Isolating sets for some singular points}
\label{table:singpttc}
\centering
\begin{tabular}{clcclc}
\toprule
\textnumero & $X_{d_1,d_2} \subset \mbP (a_0,\dots, a_5)$ & $(A^3)$ & $\msp$ & Isolating set & $e$ \\ 
\midrule
20 & $X_{6,8} \subset \mbP (1,2,2,3,3,4)$ & $\frac{1}{3}$ & $\frac{1}{2} (1, 1, 1)$ & $\{x, s, t\}$ & $3$ \\
37 & $X_{8,12} \subset \mbP (1,2,3,4,5,6)$ & $\frac{2}{15}$ & $\frac{1}{2} (1, 1, 1)$ & $\{x, z, t\}$ & $5$ \\
& & & $\frac{1}{3} (1, 1, 2)$ & $\{x, y, s, t\}$ & $5$ \\
47 & $X_{10,12} \subset \mbP (1,3,4,4,5,6)$ & $\frac{1}{12}$ & $\frac{1}{3} (1, 1, 2)$ & $\{x, z, s\}$ & $4$ \\
51 & $X_{10,14} \subset \mbP (1,2,4,5,6,7)$ & $\frac{1}{12}$ & $\frac{1}{4} (1, 1, 3)$ & $\{x, y, s, t\}$ & $6$ \\
59 & $X_{12,14} \subset \mbP (1,4,4,5,6,7)$ & $\frac{1}{20}$ & $\frac{1}{4} (1, 1, 3)$ & $\{x, s, t\}$ & $6$ \\
71 & $X_{14,16} \subset \mbP (1,4,5,6,7,8)$ & $\frac{1}{30}$ & $\frac{1}{4} (1, 1, 3)$ & $\{x, z, s\}$ & $6$ \\
& & & $\frac{1}{5} (1, 2, 3)$ & $\{x, y, s\}$ & $6$ \\
\bottomrule
\end{tabular}
\end{table}
\endgroup

By the same argument, we can prove the following, where we note that $\msp \in X$ below is a maximal center (unless $X$ is special), and we will come back to this result in \S \ref{sec:No76delta7}.

\begin{Prop}
\label{prop:No76delta7}
Let $X = X_{12, 20} \subset \mbP (1, 4, 5, 6, 7, 10)$ be a member of family \textnumero~$76$ and let $\msp$ be the singular point of type $\frac{1}{7} (1, 3, 4)$.
Then, we have $\alpha_{\msp} (X) \ge 1$.
In particular, we have $\delta_{\msp} (X) > 1$.
\end{Prop}

\begin{proof}
We have $\msp = \msp_y$.
By the same argument as in the proof of Lemma~\ref{lem:lctHxSingPt2}, we have $\lct_{\msp} (X; H_x) = 1$ since $d_1 = 14 \not\equiv 1 \pmod{7}$ and $d_2 = 20 \not\equiv 1 \pmod{7}$.
We have $t z, s^2 \in \msF_1$ and $t^2 s, w^2 \in \msF_2$ by the quasismoothness of $X$, and we may assume that their coefficients in $\msF_1$ and $\msF_2$ are all $1$ by rescaling coordinates.
We can write defining equations  of $X$ as
\[
\begin{split}
\msF_1 &= t z + s^2 + f_{12} = 0, \\
\msF_2 &= t^2 s + t g_{13} + w^2 + g_{20},
\end{split}
\]
where $f_{12}, g_{13}, g_{20} \in \mbC [x, y, z, s, w]$ such that $s^2 \notin f_{12}$ and $w^2 \notin g_{20}$.
We see that $f_{12}, g_{13}, g_{20} \in (x, y, z) \subset \mbC [x, y, z, s, w]$.
It is then easy to verify that $(x = y = z =0)_X = \{\msp\}$, which shows that $\{x, y, z\}$ is a $\msp$-isolating set of maximum degree $5$.
Let $D \in |A|_{\mbQ}$ be an irreducible $\mbQ$-divisor other than $H_x$.
By Lemma~\ref{lem:isolomult}, we have $\omult_{\msp} (D) \le 7 (D \cdot H_x \cdot 5 A) = 1$.
This shows $\lct_{\msp} (X; D) \ge 1$ and $\alpha_{\msp} (X) \ge 1$.
In particular, we have $\delta_{\msp} (X) \ge (4/3) \alpha_{\msp} (X) > 1$.
\end{proof}

\section{Local delta invariants}
\label{sec:delta}

This section is devoted to the proof of the following.

\begin{Thm}
\label{thm:locdelta}
Let $X$ be a member of family \textnumero~$\msi$, where $\msi \in \msI^*_{\br}$, and let $\msp$ be a singular point with the mark $\delta$ as a subscript in the $4$th column of \emph{Table~\ref{table:BRcodim2}}.
Then either $\msp \in X$ is not a maximal center and $\alpha_{\msp} (X) \ge 1/2$ or $\delta_{\msp} (X) > 1$.
\end{Thm}

\subsection{Families \textnumero~24, 37, 51, 64, 75 and the largest index singular point}
\label{sec:Manyfamdelta}

Let $X = X_{d_1, d_2} \subset \mbP (a_0, \dots, a_5)$ be a member of family \textnumero~$\msi$, where 
\[
\msi \in \{24, 37, 51, 64, 75\},
\] 
and let $\msp = \msp_t$ be the largest index singular point of $X$.
We observe that $a_0 = 1$, $a_1 = 1$, $a_{i+2} = a_2 + i$ for $i = 1, 2, 3$, $d_1 = 2 a_3$ and $d_2 = 2 a_5$. 
By setting $a \coloneq a_2 \ge 2$, we have
\[
X = X_{2a + 2 , 2a + 6} \subset \mbP (1, 2, a, a+1, a+2, a+3).
\]

\begin{Lem}
\label{lem:Manyfamdefeq}
The following hold.
\begin{enumerate}
\item If either $\msi \ne 37$ or $\msi = 37$ and $s^2 \in f_{2a+2}$, then the defining polynomials of $X$ can be written as
\[
\begin{split}
\msF_1 &= t z - w f_{a-1} + s^2 + f_{2a+2}, \\
\msF_2 &= t^2 y + t g_{a+4} + w^2 + g_{2a+6},
\end{split}
\]
where $f_i \in \mbC [x, y]$ and $g_i \in \mbC [x, y, z, s]$.
Moreover, we may assume that there is no monomial divisible by $s^2$ in $g_{2a+6}$.
\item If $\msi = 37$ and $s^2 \notin f_{2a+2} = f_8$, then the defining polynomials of $X$ can be written as
\[
\begin{split}
\msF_1 &= t z - w f_2 + s f_4 + f_8, \\
\msF_2 &= t^2 y + t g_7 + w^2 + g_{12},
\end{split}
\]
where $f_i \in \mbC [x, y]$ and $g_i \in \mbC [x, y, z, s]$ with $\coeff_{f_2} (y) = \coeff_{g_{12}} (s^3) = 1$.
\end{enumerate}
\end{Lem} 

\begin{proof}
By the quasismoothness of $X$ at $\msp$, we have $t z \in \msF_1$ and $t^2 y \in \msF_2$.
Note that, if $\msi = 24$, then we may choose $y, z$ so that $t z \in \msF_1$ and $t^2 y \in \msF_2$.
We have $w^2 \in \msF_2$ by the quasismoothness of $X$.
Note also that $d_1 = 2 a_4 < 2 a_5$.
Then the defining polynomials can be written as
\[
\begin{split}
\msF_1 &= t z - w f_{a-1} + f_{2a+2}, \\
\msF_2 &= t^2 y + t (w g_1 + g_{a+4}) + w^2 + w g_{a+3} + g_{2a+6},
\end{split}
\]
where $f_{a-1} \in \mbC [x, y]$ and $f_{2a+2}, g_i \in \mbC [x, y, z, s]$.
Note that if $s^2 \in \msF_1$, then we can remove the terms divisible by $s^2$ in $\msF_2$ by replacing $\msF_2 - h \msF_1$ for some homogeneous polynomial $h$ of degree $d_2 - d_1$.
Replacing $w \mapsto w - (t g_1 + g_{a+3})/2$ and then $y \mapsto y - g_1^2/4$, we may assume that $g_1 = g_{a+3} = 0$.
Filtering off terms divisible by $z$ in $\msF_1$, we can write $\msF_1 = (t + \phi_{a+2}) z - w f_{a-1} + f_{2a+2} (x, y, 0, s)$, $\phi_{a+2} = \phi_{a+2} (x, y, z, s)$ is a homogeneous polynomial of degree $a+2$.
Replacing the coordinate $t \mapsto t - \phi_{a+2}$, we may assume that $f_{2a+2}$ does not involve the coordinate $z$.

If $\msi \ne 37$, then $d_1 = 2 a_4$ and $a_4 \nmid d_2$, hence $s^2 \in f_{2a+2}$ by the quasismoothness of $X$.
Hence, if either $\msi \ne 37$ or $\msi = 37$ and $s^2 \in \msF_1$, we may assume that the coefficient of $s^2$ in $\msF_1$ is $1$.
Replacing $s$ suitably, we may assume that there is no monomial divisible by $s$ in $\msF_1$ other than $s^2$ and we obtain the desired polynomials in this case.

Suppose that $\msi = 37$, that is, $a = 3$, and $s^2 \notin f_{2a+2} = f_8$.
It is clear that $s^3 \in g_{2a+6} = g_{12}$ by the quasismoothness of $X$.
We claim that $f_{a-1} (0, y) = f_2 (0, y) \ne 0$.
Assume to the contrary that $f_2 (0, y) = 0$.
Then, $\msF_1 = t z + f_8 (x, y, z) \in (x, y, z, t)^2$ and we see that the set $\Lambda \coloneq (x = y = z = t = 0)_X$ is nonempty (in fact, $\Lambda$ consists of one point).
This shows that $X$ is not quasismooth at any point of $\Lambda$.
This is a contradiction and the claim is proved.
\end{proof}

\subsubsection{Case: $a$ is odd}
\label{sec:MfamdeltaC1}
We assume that $a$ is odd.
Note that either $a = 3$ or $a = 5$.
In this case $X$ is a member of families \textnumero~$37$ and \textnumero~$64$, and the point $\msp \in X$ is of type $\frac{1}{5} (1, 1, 4)$ and $\frac{1}{7} (1, 1, 6)$, respectively.
Let $S \in |(a+1) A|$ be a general member and we set $C \coloneq (x = 0)_S$.

\begin{Lem}
\label{sec:MfamdeltaC1qsmS}
The surface $S$ is quasismooth and $\msp \in S$ is a singular point of type $\frac{1}{a+2} (1, 1)$.
\end{Lem}

\begin{proof}
It is easy to see that $\msp \in S$ is of type $\frac{1}{a+2} (1, 1)$.

We have 
\[
\Bs |(a+1)A| = (x = y = s = 0)_X =
\begin{cases}
\{\msp_1, \msp_2, \msp_t\}, & \text{if $a = 3$}, \\
\{\msp_z, \msp_t\}, & \text{if $a = 5$},
\end{cases}
\]
where $\msp_1, \msp_2 \in X$ are the singular points of type $\frac{1}{3} (1, 1, 2)$ when $a = 3$.
It is straightforward to check that a general $S \in |(a+1)A|$ is quasismooth at any point in $\Bs |(a+1)A|$.
This shows that $S$ is quasismooth.
\end{proof}

Let 
\[
s = \lambda y^{(a+1)/2} + x h_a
\] 
be the equation which defines $S$ in $X$, where $\lambda \in \mbC$ and $h_a = h_a (x, y, z) \in \mbC [x, y, z]$ are general.
We set $\alpha \coloneq \coeff_{f_{a-1}} \left(y^{(a-1)/2}\right)$.
By eliminating $x, y$ in terms of the equations $x = 0$ and $s = \lambda y^{(a+1)/2}$, the scheme $C$ is isomorphic to the complete intersection in $\mbP (2, a, a+2, a+3)$ defined by the equations $\overline{\msF}_1 = \overline{\msF}_2 = 0$, where
\[
\overline{\msF}_i = \overline{\msF}_i (y, z, t, w) \coloneq \msF_1 \left(0, y, z, \lambda y^{(a+1)/2}, t, w\right)
\]
for $i = 1, 2$.
We can write
\[
\begin{split}
\overline{\msF}_1 &= t z - \alpha w y^{(a-1)/2} + \beta y^{a+1} = 0, \\
\overline{\msF}_2 &= t^2 y + \gamma t y^{(a+1)/2} z + w^2 + \bar{g}_{2a+6} = 0,
\end{split}
\]
where $\beta, \gamma \in \mbC$ and $\bar{g}_{2a+6} \coloneq g_{2a+6} (0, y, z, \lambda y^{(a+1)/2}) \in \mbC [y, z]$.
We can write
\[
\bar{g}_{2a+6} = \theta + \delta y^3 + \varepsilon y^{a+3},
\]
where $\theta, \delta, \varepsilon \in \mbC$.
We have $\theta = 0$ if $a = 5$ and $\theta = \coeff_{\msF_2} (z^4)$ if $a = 3$.
Note that we can view $\beta, \gamma, \delta$ and $\varepsilon$ as polynomials in $\lambda$.
We denote by $\deg_{\lambda} \beta$ etc.\ the degree of $\beta$ as a polynomial in $\lambda$.

\begin{Lem}
\label{sec:MfamdeltaC1coeff}
The following assertions hold.
\begin{enumerate}
\item If either $a = 5$ or $a = 3$ and $s^2 \in \msF_1$, then $\deg_{\lambda} \beta = 2$.
If $a = 3$ and $s^2 \notin \msF_1$, then $\deg_{\lambda} \beta \le 1$.
In any case, we have $\beta \ne 0$.
\item We have $\deg_{\lambda} \gamma \le 1$ and $\deg_{\lambda} \delta \le 1$.
\item If $a = 3$ and $s^2 \in \msF_1$, then $\deg_{\lambda} \varepsilon \le 1$.
If $a = 3$ and $s^2 \notin \msF_1$, then $\deg_{\lambda} \varepsilon = 3$.
\item We have $\beta^2 + \alpha^2 \varepsilon \ne 0$.
\end{enumerate}
\end{Lem}

\begin{proof}
Note that $\lambda^m$ appears in $\beta$ if $y^k z^l s^m \in \msF_1$ for some $k, l \ge 0$, and similarly for $\gamma, \delta, \varepsilon$.
It is then easy to see that the assertions (1), (2) and (3) hold except for the assertion that $\beta \ne 0$.

If either $a = 5$ or $a = 3$ and $s^2 \in \msF_1$, then $\beta \ne 0$ since $\deg_{\lambda} \beta = 2$ and $\lambda$ is general.
Suppose that $a = 3$ and $s^2 \notin \msF_1$.
Then, the assertion $\beta = 0$ is equivalent to $s y^2, y^4 \notin \msF_1$, and in this case $X$ is singular along the curve $(x = z = t = 0)_X$.
This is impossible, and hence $\beta \ne 0$.

It remains to prove $\beta^2 + \alpha^2 \varepsilon \ne 0$.
Note that $\alpha$ does not depend on $\lambda$, that is, $\deg_{\lambda} \lambda = 0$.
By (1), (2) and (3), we have 
\[
\deg_{\lambda} (\beta^2 + \alpha^2 \varepsilon) = 
\begin{cases}
4, & \text{if either $a = 5$ or $a = 3$ and $s^2 \in \msF_1$}, \\
3, & \text{if $a = 3$ and $s^2 \notin \msF_1$}.
\end{cases}
\]
This shows that $\beta^2 + \alpha^2 \varepsilon \ne 0$ since $\lambda$ is general.
\end{proof}

\begin{Lem}
\label{sec:MfamdeltaC1Cz0}
The scheme $(z = 0)_C$ does not contain a curve.
\end{Lem}

\begin{proof}
Suppose that $\alpha = 0$.
Then it is easy to see that $(z = 0)_C = \{\msp\}$ set-theoretically since $\beta \ne 0$.

Suppose that $\alpha \ne 0$.
In this case, we have $(z = 0)_C = \{\msp\} \cup \Xi$ set-theoretically, where
\[
\Xi \coloneq \left(y = z = w - \alpha^{-1} y^{(a+3)/2} = \overline{\msF}_2 \left(y, 0, t, \alpha^{-1} \beta y^{(a+3)/2}\right) = 0\right).
\]
We have
\[
\overline{\msF}_2 \left(y, 0, t, \alpha^{-1} \beta y^{(a+3)/2}\right) = y (t^2 + (\alpha^{-2} \beta^2 + \varepsilon) y^{a+2}).
\]
This shows that $(z = 0)_C$ does not contain a curve.
\end{proof}

\begin{Lem}
\label{sec:MfamdeltaC1qsmC}
The curve $C$ is irreducible and reduced, and it is quasismooth at $\msp$.
\end{Lem}

\begin{proof}
It is straightforward to check that $C$ is quasismooth at $\msp$.
Hence it remains to show that $C$ is irreducible and reduced.

By Lemma~\ref{sec:MfamdeltaC1Cz0}, it is enough to show that $C \cap U_z$ is irreducible and reduced.
Plugging $z = 1$ in the equations $\overline{\msF}_1 = \overline{\msF}_2 = 0$ and eliminating the variable $t$ in terms of the equation $\overline{\msF}_1 (y, 1, t, w) = 0$, $C \cap U_z$ is the $\bmu_a$-quotient of the subscheme in $\mbA^2_{y, w}$ defined by the equation
\[
\widetilde{\msF}_2 \coloneq \overline{\msF}_2 \left(y, 1, \alpha w y^{(a-1)/2} - \beta y^{a+1}, w\right) = 0.
\]
We have
\[
\begin{split}
\widetilde{\msF}_2 &= (\alpha^2 y^a + 1) w^2 + \alpha y^a \left(\gamma - 2 \beta y^{(a+3)/2}\right)w \\
& \hspace{2.5cm} + \beta^2 y^{2a+3} - \beta \gamma y^{3(a+1)/2} + \theta + \delta y^3 + \varepsilon y^{a+3}.
\end{split}
\]
It suffices to show that $\widetilde{\msF}_2$ is an irreducible polynomial.

Suppose that $\alpha = 0$.
Then
\[
\widetilde{\msF}_2 = w^2 +  \beta^2 y^{2a+3} - \beta \gamma y^{3(a+1)/2} + \theta + \delta y^3 + \varepsilon y^{a+3}
\]
and the polynomial $\beta^2 y^{2a+3} - \beta \gamma y^{3(a+1)/2} + \theta + \delta y^3 + \varepsilon y^{a+3}$ cannot be a square since its top degree term $\beta^2 y^{2a+3}$, where $\beta \ne 0$, is not a square.
This shows that $\widetilde{\msF}_2$ is irreducible.

Suppose that $\alpha \ne 0$.
We claim that three polynomials
\[\alpha^2 y^a + 1, \ y^a \left(\gamma - 2 \beta y^{(a+3)/2}\right), \  \beta^2 y^{2a+3} - \beta \gamma y^{3(a+1)/2} + \theta + \delta y^3 + \varepsilon y^{a+3}
\] 
do not share a common component.
If either $a = 5$ or $a = 3$ and $s^2 \in \msF_1$, then no component of $\alpha^2 y^a + 1$ divides $y^a \left(\gamma  - 2 \beta y^{(a+3)/2}\right)$ since $\lambda \in \mbC$ is general, $\beta = \lambda^2 + \cdots$ and $\deg_{\lambda} (\gamma) \le 1$ by Lemma~\ref{sec:MfamdeltaC1coeff}.
If $a = 3$ and $s^2 \notin \msF_1$, then no component of $\alpha^2 y^a + 1 = \alpha^2 y^3 + 1$ divide $\beta^2 y^9 + (\varepsilon - \beta \gamma) y^6 + \delta y^3 + \theta$ since $\deg_{\lambda} \beta, \deg_{\lambda} \gamma, \deg_{\lambda} \delta \le 1$ and $\deg_{\lambda} \varepsilon = 3$ by Lemma~\ref{sec:MfamdeltaC1coeff}.
This proves the claim.

Now we assume to the contrary that $\widetilde{\msF}_2$ is reducible.
By the claim, the equation $\widetilde{\msF}_2 = 0$ has a solution $w = p/q$, where $p, q \in \mbC [y]$ are relatively prime polynomials.
We have
\begin{equation}
\label{eq:Mfamdelta-1}
\begin{split}
0 &= q^2 \widetilde{\msF}_2 = (\alpha^2 y^a + 1) p^2 + \alpha y^a \left(\gamma - 2 \beta y^{(a+3)/2}\right) p q \\
& \hspace{3.5cm} + \left(\beta^2 y^{2a+3} - \beta \gamma y^{3(a+1)/2} + \theta + \delta y^3 + \varepsilon y^{a+3}\right) q^2 \\
&= y^a \left(\alpha p - \beta y^{(a+3)/2} q\right)^2 + p^2 + \alpha \gamma y^a pq \\
& \hspace{3.5cm} + \left(- \beta \gamma y^{3(a+1)/2} + \theta + \delta y^3 + \varepsilon y^{a+3}\right) q^2 
\end{split}
\end{equation}
We set $n \coloneq \deg_y q$.
We have $\deg_y p = n + (a+3)/2$ because otherwise $q^2 \widetilde{\msF}_2$ cannot be $0$.

Suppose that $a = 3$.
We put $\tilde{p} \coloneq \alpha p - \beta y^3 q$.
We have $\deg_y (y^a \tilde{p}^2) \le 2 n + 6$ by \eqref{eq:Mfamdelta-1}, which implies $\deg_y \tilde{p} \le n + 1$.
Plugging $p = \alpha^{-1} \left(\tilde{p} + \beta y^{(a+3)/2}q\right)$ in \eqref{eq:Mfamdelta-1}, we obtain
\[
q^2 \widetilde{\msF}_2 = y^3 \tilde{p}^2 + \alpha^{-2} (\tilde{p} + \beta y^3 q)^2 + \gamma y^3 \tilde{p} q + (1 + \delta y^3 + \varepsilon y^6)q^2.
\]
The coefficient of $y^{2n+6}$ in $q^2 \widetilde{\msF}_2$ is $\alpha^{-2} \beta^2 + \varepsilon$ and it is nonzero by Lemma~\ref{sec:MfamdeltaC1coeff}.
This is a contradiction since $q^2 \widetilde{\msF}_2 = 0$.

Suppose that $a = 5$.
Note that $\theta = 0$ in this case.
By \eqref{eq:Mfamdelta-1}, we see that $p$ is divisible by $y^2$ and we write $p = y^2 p_1$.
Then, by plugging $p = y^2 p_1$ and dividing \eqref{eq:Mfamdelta-1} by $y^3$, we get
\[
y^6 (\alpha p_1 - \beta y^2 q)^2 + y p_1^2 + \alpha \gamma y^4 p_1 q + (-\beta \gamma y^6 + \delta + \varepsilon y^5) q^2 = 0.
\]
If $\delta \ne 0$, then $q$ must be divisible by $y$, but this is impossible since $p$ is coprime to $q$.
Hence $\delta = 0$ and we see that $p_1$ is divisible by $y^2$.
We write $p_1 = y^2 p_2$ and, by dividing the above equation by $y^5$, we obtain
\begin{equation}
\label{eq:Mfamdelta-2}
y^5 (\alpha p_2 - \beta q)^2 + p_2^2 + \alpha \gamma y p_2 q + (-\beta \gamma y + \varepsilon) q^2 = 0.
\end{equation}
We have $\deg_y p_2 = \deg_y p - 4 = n$.
We put $\tilde{p} \coloneq \alpha p_2 - \beta q$.
By \eqref{eq:Mfamdelta-2}, we have $\deg (y^5 \tilde{p}^2) \le 2 n + 1$, which implies $\deg_y \tilde{p} \le n - 2$.
Plugging $p_2 = \alpha^{-1} (\tilde{p} + \beta q)$ in \eqref{eq:Mfamdelta-2}, we obtain
\begin{equation}
\label{eq:Mfamdelta-3}
\msG \coloneq y^5 \tilde{p}^2 + \alpha^{-2} (\tilde{p} + \beta q)^2 + \gamma y \tilde{p} q + \varepsilon q^2 = 0.
\end{equation}
We have $\coeff_{\msG} (y^{2n+1}) = \coeff_{\tilde{p}} (y^{n-2})^2$.
It follows that $\coeff_{\tilde{p}} (y^{n-2}) = 0$, that is, $\deg_y \tilde{p} \le n-3$, since $\msG = 0$.
Then, we have $\coeff_{\msG} (y^{2n}) = \alpha^{-2} \beta^2 + \varepsilon = 0$ since $\msG = 0$.
This is impossible by Lemma~\ref{sec:MfamdeltaC1coeff} and the proof is complete.
\end{proof}

By Lemmas~\ref{sec:MfamdeltaC1qsmS} and \ref{sec:MfamdeltaC1qsmC}, $\msp \in C \subset S \subset X$ is a flag of type $\mathrm{I}$.
By Proposition~\ref{prop:typeI}, we have
\[
\delta_{\msp} (X) \ge
\min \left\{4, \ 4 (a+1), \ \frac{4}{(a+2) \cdot (a+1) \cdot 1 \cdot (A^3)} \right\} = \frac{2a}{a+1} > 1
\]
since $a = 3, 5$.

\subsubsection{Case: $a$ is even} 
We assume that $a$ is even.
Note that $a = 2, 4, 6$. 
In this case $X$ is a member of family \textnumero~$24$, \textnumero~$51$ and \textnumero~$75$ and the point $\msp \in X$ is of type $\frac{1}{4} (1, 1, 3)$, $\frac{1}{6} (1, 1, 5)$ and $\frac{1}{8} (1, 1, 7)$, respectively.

We set $S \coloneq (x = 0)_X$ which is a normal surface by Lemma~\ref{lem:normqhyp}.
We choose homogeneous coordinates as in Lemma~\ref{lem:Manyfamdefeq}.
We have $f_{a-1} (0, y) = 0$ since $a$ is even and the weight of $y$ is $2$.
By eliminating the variable $x$, $S$ is isomorphic to the weighted complete intersection in $\mbP (2, a, a+1, a+2, a+3)$ defined by the equations
\[
\begin{split}
\overline{\msF}_1 &\coloneq t z + s^2 + \bar{f}_{2a+2} = 0, \\
\overline{\msF}_2 &\coloneq t^2 y + t \bar{g}_{a+4} + w^2 + \bar{g}_{2a+6} = 0.
\end{split}
\]
We can choose $s$ and $w$ as local orbifold coordinates of $S$ at $\msp$.
Let $\psi \colon T \to S$ be the weighted blowup of $S$ at $\msp$ with weight
\[
\wt (s, w) = \frac{1}{a+2} (a+1, 1).
\]
We denote by $F$ the exceptional divisor of $\psi$.
We set 
\[
\begin{split}
P (u ) &\coloneq -K_X - u S \sim (1 - u)A, \\
N (u) &\coloneq 0.
\end{split}
\]
For $0 \le u \le 1$, let $P (u, v)$ and $N (u, v)$ be the positive and negative parts, respectively, of the Zariski decomposition of the divisor
\[
\psi^*(P (u)|_S) - v F = (1-u) \psi^*A_S - v F,
\] 
where $A_S \coloneq A|_S$.
We have
\[
(\psi^*A_S)^2 = (A_S)^2 = A^3 = \frac{2}{a (a+2)} \quad \text{and} \quad
(F^2) = - \frac{a+2}{a+1}.
\]

By the equations $\overline{\msF}_1 = \overline{\msF}_2 = 0$, we have
\[
\ord_F (y, z) = \frac{1}{a+2} (2, 2a+2).
\]

We set $\alpha \coloneq - \coeff_{\bar{f}_{2a+2}} (y^{a+1})$.
We have
\[
(z = 0)_S = (z = s^2 - \alpha y^{a+1} = \overline{\msF}_2 = 0)
\]

\begin{Def}
We define a divisor class $D$ and a Weil divisor $G$ on $S$ as
\[
\begin{split}
D &\coloneq \psi^*A - \frac{1}{a+2} F, \\
G &\coloneq \psi_*^{-1} ((z = 0)_S) \sim a \psi^*A_S - \frac{2a+2}{a+2} F.
\end{split}
\]
\end{Def}

\begin{Lem}
\label{lem:MfamdeltaC2cone}
\begin{enumerate}
\item For a real number $r > 0$, $\psi^*A_S - r F$ is nef if and only if 
\[
r \le \frac{1}{a+2}.
\]
In particular, $D$ is a nef but not ample.
\item For a real number $r > 0$, $\psi^*A_S - r F$ is pseudoeffective if and only if 
\[
r \le \frac{2(a+1)}{a(a+2)}.
\]
\item Either $G|_F = \msq_1 + \msq_2$ for some distinct smooth points $\msq_1, \msq_2 \in F$ of $T$ or $G|_F = 2 \msq_3$ for some smooth point $\msq_3 \in F$ of $T$.
\end{enumerate}
\end{Lem} 

\begin{proof}
We set $r_0 \coloneq 1/(a+2)$.
We have
\begin{equation}
\label{eq:MfamdeltaC2cone1}
(D \cdot G) = \left(\psi^*A - \frac{1}{a+2} F\right) \cdot \left(a \psi^*A_S - \frac{2a+2}{a+2} F\right) = \frac{2}{a+2} - \frac{2}{a+2} = 0.
\end{equation}
This shows that $\psi^*A_S - r F$ is nef only if $r \le r_0$.
We show that $\psi^*A_S - r_0 F$ is nef.
For $v \in \{y, z, s, t, w\}$, we set $H_v \coloneq (v = 0)_S$ and $\tilde{H}_v \coloneq \psi_*^{-1} H_v$.
For a sufficiently divisible $m > 0$, there are integers $k_v > 0$ and $l_v \ge 0$ for $v \in \{x, y, s, t, w\}$ such that the linear system $|m (\psi^*A_S-r_0F)|$ contains $k_v \tilde{H}_v + l_v G$.
Note that $l_y = l_s = 0$.
This shows that the base locus of $|m (\psi^*A_S-r_0F)|$ is contained in the union of $\tilde{H}_x \cap \tilde{H}_y \cap \tilde{H}_s \cap \tilde{H}_t$ and $\tilde{H}_y \cap \tilde{H}_s \cap G$.
It is easy to see that both sets are zero dimensional and hence $\psi^*A_S - r_0F$ is nef.
This shows (1).
The assertion (2) follows from (1) and \eqref{eq:MfamdeltaC2cone1}.

The exceptional divisor $F$ is isomorphic to the complete intersection
\[
(z + s^2 + f_{2a+2} (0, y) = y + w^2 = 0) \subset \mbP (2, 2a+2, a+1, 1),
\]
where, by abuse of notation, we think of $y, z, s, w$ as homogeneous coordinates of $\mbP (2, 2a+2, a+1, 1)$ of weight $2, 2a+2, a+1, 1$, respectively.
The divisor $G|_F$ is defined by the equation $z = 0$.
The assertion (3) follows immediately from this description.
\end{proof}

\begin{Lem}
We set
\[
\tau_1 (u) \coloneq \frac{1}{a+2} (1-u) \text{ and }
\tau_2 (u) \coloneq \frac{2(a+1)}{a(a+2)} (1-u).
\]
Then, we have
\[
\begin{split}
P (u, v) &=
\begin{dcases}
(1-u) \psi^*A_S - v F, & \left(0 \le v \le \tau_1 (u) \right), \\
\frac{2(a+1)(1-u) - a (a+2) v}{a+2} D, & \left(\tau_1 (u) \le v \le \tau_2 (u)\right),
\end{dcases} \\
N (u, v) &=
\begin{dcases}
0, & \left(0 \le v \le \tau_1 (u) \right), \\
\left(v - \frac{1}{a+2} (1-u) \right) G, & \left(\tau_1 (u) \le v \le \tau_2 (u) \right).
\end{dcases}
\end{split}
\]
\end{Lem}

\begin{proof}
This follows from Lemma~\ref{lem:MfamdeltaC2cone}.
\end{proof}

We have $K_T = \psi^*K_S$ and $(K_T + F)|_F = K_F + \frac{a}{a+1} \msq_0$, where $\msq_0$ is the singular point of $T$ of type $\frac{1}{a+1} (a, 1)$ on $F$.
By \cite[Corollary~4.18]{Fujita23}, we have
\[
\delta_{\msp} (X) \ge \min \left\{ \frac{A_X (S)}{S_{-K_X} (S)}, \ \frac{A_{S} (\msq)}{S (V_{\bullet, \bullet}^S; F)}, \ \inf_{\msq \in F} \frac{A_{C, \frac{a}{a+1} \msq_0} (\msq)}{S (W_{\bullet, \bullet, \bullet}^{S, F}; \msq)} \right\},
\]
where the first value is $4$ and the remaining values are computed as follows.
\[
\begin{split}
S (V_{\bullet, \bullet}^S; F) &= \frac{3}{(A^3)} \int_0^1 \int_0^{\tau_2 (u)} P (u, v)^2 d v d u \\
&= \frac{3}{(A^3)} \int_0^1 \left( \int_0^{\tau_1 (u)} \left (\frac{2}{a (a+2)} (1-u)^2 - \frac{a+2}{a+1} v^2 \right) d v \right. \\
& \hspace{1cm} \left. + \int_{\tau_1 (u)}^{\tau_2 (u)} \frac{1}{a (a+1)} \left(\frac{2 (a+1)(1-u) - a (a+2) v}{a+2} \right)^2 d v \right) d u \\
&= \frac{3}{(A^3)} \int_0^1 \left(\frac{5 a + 6}{3 a (a+1) (a+2)^2} (1-u)^3 + \frac{1}{3 a^2 (a+1)} (1-u)^3 \right) d u \\
&= \frac{3 a (a+2)}{2} \int_0^1 \frac{3a+2}{3 a^2 (a+2)^2} (1-u)^3 \\
&= \frac{3 a + 2}{8 a (a+2)}.
\end{split}
\]
\[
\begin{split}
S (W_{\bullet, \bullet, \bullet}^{S, F}; \msq) &=
\frac{3}{(A^3)} \int_0^1 \int_0^{\tau_2 (u)} (P (u, v) \cdot F)^2 d v d u + F_{\msq} (W_{\bullet, \bullet, \bullet}^{S, F}) \\
&= \frac{3}{(A^3)} \int_0^1 \left( \frac{1}{3 (a+1)^2 (a+2)} + \frac{1}{3 a (a+1)^2} \right) (1-u)^3 + F_{\msq} (W_{\bullet, \bullet, \bullet}^{S, F}) \\
&= \frac{3 a (a+2)}{2} \int_0^1 \frac{2}{3 a (a+2)} (1-u)^3 + F_{\msq} (W_{\bullet, \bullet, \bullet}^{S, F}) \\
&= \frac{1}{4 (a+1)} + F_{\msq} (W_{\bullet, \bullet, \bullet}^{S, F}).
\end{split}
\]
\[
\begin{split}
& F_{\msq} (W_{\bullet, \bullet, \bullet}^{S, F}) \\
&= \frac{6}{(A^3)} \int_0^1 \int_0^{\tau_2 (u)} (P (u, v) \cdot F) \cdot \ord_{\msq} (N (u, v)|_F) d v d u \\
&= \frac{6}{(A^3)} \int_0^1 \int_{\tau_1 (u)}^{\tau_2 (u)} \frac{2 (a+1)(1-u) - a (a+2)v}{(a+1)(a+2)} \cdot \left(v - \frac{1-u}{a+2} \right) \ord_{\msq} (G|_F) d v d u \\
&= 3 a (a+2) \ord_{\msq} (G|_F) \int_0^1 \frac{1}{6 a^2 (a+1)} (1-u)^3 d u \\
&= \frac{a+2}{8 a (a+1)} \ord_{\msq} (G|_F).
\end{split}
\]
By Lemma~\ref{lem:MfamdeltaC2cone}, we have
\[
\delta_{\msp} (X) \ge \min \left\{ 4, \ \frac{8 a (a+2)}{3 a + 2}, \ \min \{4, \ 2 a\} \right\} = 4
\]
since $a \ge 2$.

\subsection{Family \textnumero~8 and $\frac{1}{2} (1, 1, 1)$ points}
\label{sec:No8delta2}

Let $X = X_{4, 6} \subset \mbP (1, 1, 2, 2, 2, 3)$ be a member of family \textnumero~8 and let $\msp$ be a $\frac{1}{2} (1, 1, 1)$ point of $X$.
We may assume that $\msp = \msp_t$.

\begin{Lem}
\label{lem:No8del2eq}
The defining polynomials of $X$ can be written as
\[
\begin{split}
\msF_1 &= t z + w f_1 (x, y) + \theta s^2 + s f_2 (x, y) + f_4 (x, y), \\
\msF_2 &= t^2 s + t g_4 + w^2 + g_6,
\end{split}
\]
where $\theta \in \mbC$, $f_i \in \mbC [x, y]$ and $g_i \in \mbC [x, y, z, s]$.
Moreover, if $\theta = 0$, then $s^3 \in g_6$.
\end{Lem}

\begin{proof}
By the quasismoothness of $X$ at $\msp =\msp_t$, we may assume that $t z \in \msF_1$, $t^2 s \in \msF_2$.
Moreover, we have $w^2 \in \msF_2$ by the quasismoothness of $X$.
The defining polynomials of $X$ can be written as
\[
\begin{split}
\msF_1 &= t z + w f_1 (x, y) + \theta s^2 + s f_2 + f_4, \\
\msF_2 &= t^2 s + t (w g_1 + g_4) + w^2 + w g_3 + g_6,
\end{split}
\]
where $\theta \in \mbC$, $f_i \in \mbC [x, y, z]$ and $g_i \in \mbC [x, y, z, s]$.
Replacing $w \mapsto w - (t g_1 + g_3)/2$ and then replacing $z$, we may assume that $g_1 = g_3 = 0$.
Filtering off terms divisible by $z$ in $\msF_1$ and replacing $t$, we may assume that $f_i \in \mbC [x, y]$.
By the quasismoothness of $X$, either $s^2 \in \msF_1$ or $s^3 \in \msF_2$, and the proof is complete.
\end{proof}

\subsubsection{Case: $\theta \ne 0$}

We may assume that $\theta = 1$ by rescaling coordinates.
Let $S$ and $H$ be general members of $|A|$.
Note that $S$ is a normal surface by Lemma~\ref{lem:normqhyp} and we set $C \coloneq S \cap H$.

\begin{Lem}
\label{lem:No8delta2C2C}
The curve $C$ is irreducible and reduced and it is quasismooth at $\msp$.
\end{Lem}

\begin{proof}
We see that $C = (x = y = 0)_X$ is isomorphic to the complete intersection
\[
(t z + s^2 = t^2 s + t \bar{g}_4 + w^2 + \bar{g}_6 = 0) \subset \mbP (2, 2, 2, 3),
\]
where $\bar{g}_i = g_i (0, 0, z, s)$.
It is straightforward to see that $C$ is irreducible and reduced and quasismooth at $\msp$.
\end{proof}

By Lemma~\ref{lem:No8delta2C2C}, the flag $\msp \in C \coloneq S \cap H \subset S \subset X$ is of type $\mathrm{I}$ and, by Proposition~\ref{prop:typeI}, we obtain
\[
\delta_{\msp} (X) \ge \min \left\{ 4, \ 4, \ \frac{4}{2 \cdot 1 \cdot 1 \cdot 1} \right\} = 2.
\]

\subsubsection{Case: $\theta = 0$}

We assume $\theta = 0$.
By Lemma~\ref{lem:No8del2eq}, we may assume $\coeff_{g_6} (s^3) = -1$ by rescaling $s$.
Let $S, H \in |A|$ be general members.
Then, $H|_S = \Gamma + \Delta$, where
\[
\begin{split}
\Gamma &\coloneq (x = y = z = t^2 s + \alpha t s^2 + w^2 - s^3 = 0), \\
\Delta &\coloneq (x = y = t = w^2 + \bar{g}_6 = 0),
\end{split}
\]
where $\alpha \coloneq \coeff_{g_4} (s^2)$ and $\bar{g}_i \coloneq g_i (0, 0, z, s)$ for $i = 4, 6$.
Note that $\msp \in \Gamma$ and $\msp \notin \Delta$.

\begin{Lem}
\label{lem:No8delta2C1S}
The following hold.
\begin{enumerate}
\item The surface $S$ is quasismooth.
\item The curves $\Gamma$ and $\Delta$ are both irreducible smooth  rational curves.
\item We have
\[
(\Gamma^2) = (\Delta^2) = - \frac{1}{2} \quad \text{and} \quad (\Gamma \cdot \Delta) = 1.
\]
\end{enumerate}
\end{Lem}

\begin{proof}
Let $x - \lambda y = 0$ be the equation that defines $S$ in $X$, where $\lambda \in \mbC$ is general.
For the assertion (1), it is enough to show that $S$ is quasismooth along the base locus of $|A|$, where we have $\Bs |A| = (x = y = 0)_X$.
On the subset, we have
\[
J_S|_{(x = y = 0)} =
\begin{pmatrix}
\alpha w & \beta w & t & 0 & z & 0 \\
0 & 0 & \prt \msF_2/\prt z & \prt \msF_2/\prt s & \prt \msF_2/\prt t & 2 w \\
1 & \lambda & 0 & 0 & 0 & 0
\end{pmatrix}
\]
The matrix consisting of the first two rows is $J_X$ and it is of rank $2$ at any point of $X$.
This shows that $J_S$ is of rank $3$ at any point of $S$, and (1) is proved.

Eliminating $x, y, z$, the curve $\Gamma$ is isomorphic to the hypersurface in $\mbP (2, 2, 2, 3)$ defined by $t^2 s + \alpha t s^2 + w^2 - s^3 = 0$.
We have an isomorphism $\mbP (2, 2, 3) \cong \mbP (1, 1, 3)$ by the correspondence $\bar{s} = s$, $\bar{t} = t$ and $\bar{w} = w^2$, where we think of $\bar{s}, \bar{t}, \bar{w}$ as homogeneous coordinates of $\mbP (1, 1, 3)$ of weights $1, 1, 3$, respectively.
By this isomorphism, $\Gamma$ corresponds to the hypersurface in $\mbP (1, 1, 3)$ defined by $\bar{t}^2 \bar{s} + \alpha \bar{t} \bar{s}^2 + \bar{w} - \bar{s}^3 = 0$, which shows $\Gamma \cong \mbP^1$.
Similarly, we have $\Delta \cong \mbP^1$ and (2) is proved.

We see that $\Sing_{\Gamma} (S) = \{3 \times \frac{1}{2} (1, 1)\}$.
Hence we have
\[
(\Gamma^2) = -2 + \frac{1}{2} + \frac{1}{2} + \frac{1}{2} = - \frac{1}{2}.
\]
Similarly we have $(\Delta^2) = -1/2$ since $\Sing_{\Delta} (S) = \{3 \times \frac{1}{2} (1, 1)\}$. 
By taking intersection numbers of $H|_S = \Gamma + \Delta$ and $\Gamma$, we have $(\Gamma \cdot \Delta) = 1$ since $(H|_S \cdot \Gamma) = (A \cdot \Gamma) = 1/2$.
\end{proof}

It follows that $\msp \in \Gamma \subset S \subset X$ is a flag of type $\mathrm{IIa}$ and, by Proposition~\ref{prop:typeII}, we have
\[
\delta_{\msp} (X) \ge \min \left\{ 4, \ \frac{1}{S (V_{\bullet, \bullet}^S; \Gamma)}, \ \frac{1}{2 S (W_{\bullet, \bullet, \bullet}^{S, \Gamma}; \msp)} \right\},
\]
where
\[
\begin{split}
S (V_{\bullet, \bullet}^S; \Gamma) &= 3 \int_0^1 \left(\int_0^{\frac{1-u}{2}} \left((1-u)^2 - (1-u)v - \frac{1}{2} v^2 \right) d v \right. \\
&\hspace{3.5cm} \left. + \int_{(1-u)/2}^{1-u} \frac{3}{2} (1-u-v)^2 d v \right) d u \\
&= \frac{5}{16},
\end{split}
\]
and
\[
\begin{split}
S (W_{\bullet, \bullet, \bullet}^{S, \Gamma}; \msp) &= 3 \int_0^1 \left( \int_0^{\frac{1-u}{2}} \left(\frac{1}{2} (1-u) + \frac{1}{2} v \right)^2 d v  \right. \\
& \hspace{2cm} \left. + \int_{\frac{1-u}{2}}^{1-u} \frac{9}{4} (1-u-v)^2 d v \right) d u + F_{\msp} (W_{\bullet, \bullet, \bullet}^{S, \Gamma}) \\
&= \frac{7}{32}.
\end{split}
\]
Note that $F_{\msp} (W_{\bullet, \bullet, \bullet}^{S, \Gamma}) = 0$ since $\msp \notin \Gamma \cap \Delta$.
Thus, we obtain
\[
\delta_{\msp} (X) \ge \min \left\{ 4, \ \frac{16}{5}, \ \frac{16}{7} \right\} = \frac{16}{7}.
\]

\subsection{Family \textnumero~20 and $\frac{1}{3} (1, 1, 2)$ points}
\label{sec:No20delta3}

Let $X = X_{6, 8} \subset \mbP (1, 2, 2, 3, 3, 4)$ be a member of family \textnumero~20 and let $\msp$ be a $\frac{1}{3} (1, 1, 2)$ point of $X$.
We may assume that $\msp = \msp_t$.

\begin{Lem}
The defining polynomials of $X$ can be written as
\[
\begin{split}
\msF_1 &= t s + w f_2 + f_6, \\
\msF_2 &= t^2 z + t g_5 + w^2 + g_8,
\end{split}
\]
where $f_i \in \mbC [x, y, z]$ and $g_i \in \mbC [x, y, z, s]$.
Moreover, we have $f_6 (0, y, z) \ne 0$ and either $s^2 y \in g_8$ or $s^2 z \in g_8$.
\end{Lem}

\begin{proof}
The proof of the first part is similar to that of Lemma~\ref{lem:No8del2eq} and we omit it.
Suppose that $f_6 (0, y, z) = 0$.
Then it is straightforward to see that $X$ cannot be quasismooth at $\frac{1}{2} (1, 1, 1)$ points.
We also see that either $s^2 y \in g_8$ or $s^2 z \in g_8$ by the quasismoothness of $X$ at $\msp_s \in X$.
\end{proof}

Let $S \in |2A|$ be a general member, $H \coloneq (x = 0)_X$ and set $C \coloneq S \cap H$.
Note that $S$ is a normal surface by Lemma~\ref{lem:normqhyp}.

\begin{Lem}
\label{lem:No20delta3C}
The scheme $C$ is irreducible and reduced.
Moreover, $S$ and $C$ are both quasismooth at $\msp$.
\end{Lem}

\begin{proof}
It is straightforward to see that both $S$ and $S \cap H$ are quasismooth at $\msp$.

It remains to show that $S \cap H$ is an irreducible and reduced curve.
Let $y = \lambda x^2 + \mu z$ be the equation which defines $S$ in $X$, where $\lambda, \mu \in \mbC$ are general.
Eliminating $x$ and $y = \lambda x^2 + \mu z$, $S \cap H$ is isomorphic to the complete intersection in $\mbP (2_y, 3_s, 3_z, 4_w)$ defined by the equations
\[
t s + \alpha w z + \beta z^3 = t^2 z + t \bar{g}_5 + w^2 + \bar{g}_8 = 0,
\]
where $\alpha, \beta \in \mbC$ and $\bar{g}_i \coloneq g_i (0, \mu z, z, s)$.
Note that $\beta \ne 0$ since $g_6 (0, y, z) \ne 0$ and $\mu \in \mbC$ is general.
We see that $S \cap H \cap (s = 0)$ is a finite set of points since $\beta \ne 0$.
Hence, by setting $s = 1$, it is enough to show that the complete intersection in $\mbA^3_{z, t, w}$ defined by the equations
\[
t + \alpha w z + \beta z^3 = t^2 z + t \bar{g}_5 (z, 1) + w^2 + \bar{g}_8 (z, 1) = 0
\]
is irreducible and reduced.
Eliminating $t = - (\alpha w z + \beta z^3)$, it is then enough to show that the polynomial
\begin{equation}
\label{eq:No20del3-1}
\begin{split}
& (\alpha w z + \beta z^3)^2 z - (\alpha w z + \beta z^3) \bar{g}_5 (z, 1) + w^2 + \bar{g}_8 (z, 1) \\
&= (\alpha z^3 + 1) w^2 + (2 \alpha \beta z^5 + \alpha z \bar{g}_5 (z, 1)) w + \beta^2 z^7 + \beta z^3 \bar{g}_5 (z, 1) + \bar{g}_8 (z, 1)
\end{split}
\end{equation}
is irreducible.
We have $s^2 z \in \bar{g}_8 (z, s)$ since either $s^2 y \in g_8$ or $s^2 z \in g_8$.
Hence $z \in \bar{g}_8 (z, 1)$.
By Lemma~\ref{lem:irredpoly1} below, the polynomial in \eqref{eq:No20del3-1} is irreducible and the proof is complete.
\end{proof}

\begin{Lem}
\label{lem:irredpoly1}
Let
\[
f (x, y) =  a (x) y^2 + b (x) y + c (x),
\]
be a polynomial in variables $x, y$, where $a (x), b (x), c (x) \in \mbC [x]$.
Suppose that $x \nmid a (x)$, $x \mid b (x)$, $x \mid c (x)$, $x^2 \nmid c (x)$ and $a (x)$, $b (x)$, $c (x)$ do not share a common component.
Then $f (x, y)$ is irreducible.
\end{Lem}

\begin{proof}
Suppose that $f (x, y)$ is not irreducible.
Then, since $a (x)$, $b (x)$, $c (x)$ do not share a common component, we can decompose $f (x, y)$ as
\[
f (x, y) = (f (x) y + g (x))(f' (x) y + g' (x)),
\]
where $f (x), g (x), f' (x), g'(x) \in \mbC [x]$.
We see that $g (x) g' (x) = c (x)$ is divisible by $x$.
We may assume that $g (x)$ is divisible by $x$ possibly after interchanging $g$ and $g'$.
We can write $g (x) = x g_1 (x)$ for some $g_1 (x) \in \mbC [x]$.
Then, by comparing the coefficients of $y$, we have $f (x) g' (x) + x g_1 (x) f' (x) = b (x)$, which implies that $f (x) g' (x)$ is divisible by $x$.
But this is impossible since $x \nmid a (x)$ and $x^2 \nmid c (x)$.
Therefore $f (x, y)$ is irreducible.
\end{proof}

By Lemma~\ref{lem:No20delta3C}, $\msp \in C \subset S \subset X$ is a flag of type $\mathrm{I}$.
By Proposition~\ref{prop:typeI} we obtain
\[
\delta_{\msp} (X) \ge \min \left\{8, 4, \frac{4}{3 \cdot 2 \cdot 1 \cdot \frac{1}{3}} \right\} = 2.
\]

\subsection{Family \textnumero~31 and the $\frac{1}{3} (1, 1, 2)$ point}
\label{sec:No31delta3}

Let $X = X_{8, 10} \subset \mbP (1, 2, 3, 4, 4, 5)$ be a member of family \textnumero 31 and let $\msp$ be the $\frac{1}{3} (1, 1, 2)$ point.

\subsubsection{Case: $\msp$ is not a maximal center}

Suppose that $\msp$ is not a maximal center.
We set $S \coloneq (x = 0)_X$.

\begin{Lem}
\label{lem:No31delta3C1S}
We have $\lct_{\msp} (X; S) \ge 1/2$.
\end{Lem}

\begin{proof}
We may assume that $\omult_{\msp} (S) > 2$ because otherwise there is nothing to prove.
Then, $x$ must be a tangent coordinate and we can write $\msF_2 = z^3 x + z^2 g_4 + z g_7 + g_{10}$ for some $g_i \in \mbC [x, y, s, t, w]$ such that $s, t \notin g_4$.
We have $w^2 \in g_{10}$ by the quasismoothness of $X$ and we may assume that $y^2 \notin g_4$ by replacing $w \mapsto w - \theta z y$ for a suitable $\theta \in \mbC$.
If $z^2 y \in \msF_1$, then we can choose $s, t, w$ as orbifold coordinates of $X$ at $\msp$ and then we have $\omult_{\msp} (S) = 2$.
This is a contradiction and we have $z^2 y \notin \msF_1$.
By choosing coordinates suitably, we can write 
\[
\begin{split}
\msF_1 &= z w + s t + a_8, \\
\msF_2 &= z^3 x + z^2 b_4 + z (g_7 + w c_2) + w^2 + w g_5 + g_{10},
\end{split}
\]
where $a_8, b_4, c_2 \in \mbC [x, y]$ and $g_i \in \mbC [x, y, s, t]$.
Plugging $z = 1$ and eliminating $w = - (st+a_8)$, $\msp \in X$ is isomorphic to the $\bmu_2$-quotient of the hypersurface in $\mbA^4_{x, y, s, t}$ defined by the equation
\[
x + b_4 + g_7 - (st+a_8)c_2 + (st+a_8)^2 - (st + a_8)g_5 + g_{10} = 0.
\]
Filtering off terms divisible by $x$, w have
\begin{equation}
\label{eq:lem:No31delta3C1S1}
(-1 + \cdots) x = \bar{b}_4 + \bar{g}_7 - (st+\bar{a}_8)\bar{c}_2 + (st + \bar{a}_8)^2 - (st+\bar{a}_8)\bar{g}_5 + \bar{g}_{10},
\end{equation}
where $\bar{a}_2 = a_2 (0, y)$, $\bar{g}_i = b_i (0, y, s, t)$ and so on.
The least degree term in the right-hand side coincides with $\msF_2 (0, y, 0, s, t, 0)$ and it can be written as $y (\alpha s^2 + \beta s t + \gamma t^2)$.
Note that $\alpha \ne 0$ and $\beta \ne 0$ by the quasismoothness of $X$ at $\msp_s$ and $\msp_t$.
In particular, the least degree term of the right-hand side of \eqref{eq:lem:No31delta3C1S1} cannot be a cube of a linear form.
By (the proof of) \cite[Lemma~3.18]{KOWsuperrigid}, we have $\lct_{\msp} (X; S) \ge 1/2$ and the proof is complete.
\end{proof}

Let $D \in |A|_{\mbQ}$ be an irreducible $\mbQ$-divisor other than $S$.
It is easy to see that the set $\{x, y, s, t\}$ isolates $\msp$ and its maximum degree is $4$ since $w^2 \in \msF_2$.
By Lemma~\ref{lem:isolomult}, we have $\omult_{\msp} (D) \le 2$, and hence $\lct_{\msp} (X; D) \ge 1/2$.
Combining this with Lemma~\ref{lem:No31delta3C1S}, we obtain  $\alpha_{\msp} (X) \ge 1/2$.

\subsubsection{Case: $\msp$ is a maximal center}

Suppose that $\msp$ is a maximal center.
By the results of \S~\ref{sec:No31MC3}, we have $z w \in \msF_1$ and either $z^2 s \in \msF_2$ or $z^2 t \in \msF_2$.
Moreover, we may assume that the defining polynomials $\msF_1$ and $\msF_2$ of $X$ are as in Lemma~\ref{lem:No31excl3C1eq}.
We set $S \coloneq (x = 0)_X \sim A$ and let $H \in |2A|$ be a general member.
Note that $S$ is a normal surface by Lemma~\ref{lem:normqhyp} and it is quasismooth at $\msp$.

\begin{Lem}
We have $H|_S = \Gamma + \Delta$, where $\Gamma$ and $\Delta$ are irreducible and reduced curves satisfying the following properties.
\begin{enumerate}
\item The curve $\Gamma$ is quasismooth, $\msp \in \Gamma$ and $\msp \notin \Delta$.
\item $(\Gamma^2) = - 7/12$, $(\Delta^2) = -1/4$ and $(\Gamma \cdot \Delta) = 3/4$.
\end{enumerate}
\end{Lem}

\begin{proof}
Eliminating the variables $x, y$, the scheme $S \cap H = (x = y = 0)_X$ is isomorphic to the complete intersection in $\mbP (3, 4, 4, 5)$ defined by the equations
\[
z w + s t = z^2 s +  w^2 = 0.
\] 
We see that $S \cap H$ contains $\Gamma = (x = y = s = w = 0)$ as an irreducible component.
We have $S \cap H \cap (z = 0) = \{\msp_s, \msp_t\}$ and it is easy to see that $S \cap H \cap (z \ne 0)$ is the union of two irreducible and reduced curves.
It follows that $H|_S = \Gamma + \Delta$ for some irreducible and reduced curve $\Delta$.
We see that $\Gamma$ is quasismooth and $\msp \in \Gamma$.
We see that $S \cap H$ is quasismooth at $\msp$ and this implies $\msp \notin \Delta$.
Hence (1) is proved.

The scheme $S \cap H$ is quasismooth outside $\{\msp_s, \msp_t\}$.
This implies that $S$ is quasismooth along $\Gamma \setminus \{\msp_t\}$.
It is easy to see that $S$ is quasismooth at $\msp_t$, and hence $S$ is quasismooth along $\Gamma$.
We have $\Sing_{\Gamma} (S) = \{\frac{1}{3} (1, 2), \frac{1}{4} (1, 3)\}$ and hence
\[
(\Gamma^2) = -2 + \frac{2}{3} + \frac{3}{4} = - \frac{7}{12}.
\]
We have $(A \cdot \Gamma) = 1/12$ and $(A \cdot \Delta) = (A \cdot H|_S) - (A \cdot \Gamma) = 1/4$.
By taking intersection numbers of $H|_S = \Gamma + \Delta$ and $\Gamma, \Delta$, we have
\[
(\Gamma \cdot \Delta) = \frac{3}{4} \quad \text{and} \quad (\Delta^2) = -\frac{1}{4},
\]
and (2) is proved.
\end{proof}

It follows that $\msp \in \Gamma \subset S \subset X$ is a flag of type $\mathrm{IIa}$ and, by Proposition~\ref{prop:typeII}, we have
\[
\delta_{\msp} (X) \ge \min \left\{ 4, \ \frac{1}{S (V_{\bullet, \bullet}^S; C)}, \ \frac{1}{3 S (W_{\bullet, \bullet, \bullet}^{S, C}; \msp)} \right\},
\]
where we compute $S (V_{\bullet, \bullet}^S; C)$ and $S (W_{\bullet, \bullet, \bullet}^{S, C}; \msp)$ as follows:
\[
\begin{split}
S (V_{\bullet, \bullet}^S; C) 
&= 18 \int_0^1 \left( \int_0^{\frac{1}{3} (1-u)} \left(\frac{1}{6} (1-u)^2 - \frac{1}{6} (1-u)v - \frac{7}{12} v^2 \right) d v \right. \\
& \left. \hspace{4.5cm} + \int_{\frac{1}{3} (1-u)}^{\frac{1}{2} (1-u)} \frac{5}{3} \left(\frac{1}{2} (1-u) - v \right)^2 d v \right) d u \\
&= \frac{3}{16},
\end{split}
\]
\[
\begin{split}
S (W_{\bullet, \bullet, \bullet}^{S, C}; \msp) 
&=
18 \int_0^{1} \left(\int_0^{\frac{1}{3} (1-u)} \left(\frac{1}{12} (1-u) + \frac{7}{12} v \right)^2 d v  \right.\\
& \left. \hspace{1.5cm} + \int_{\frac{1}{3} (1-u)}^{\frac{1}{2} (1-u)} \left(\frac{25}{9}\right)^2 \left(\frac{1}{2}(1-u) - v \right)^2 d v \right) d u + F_{\msp} (W_{\bullet, \bullet, \bullet}^{S, C}) \\
&= \frac{7}{96}
\end{split}
\]
We have $F_{\msp} (W_{\bullet, \bullet, \bullet}^{S, C}) = 0$ since $\ord_{\msp} (\Delta|_{\Gamma}) = 0$.
Thus we obtain 
\[
\delta_{\msp} (X) \ge \min \left\{ 4, \ \frac{16}{3}, \ \frac{32}{7} \right\} = 4.
\]

\subsection{Family \textnumero~31 and $\frac{1}{4} (1, 1, 3)$ points}
\label{sec:No31delta4}

Let $X = X_{8, 10} \subset \mbP (1, 2, 3, 4, 4, 5)$ be a member of family \textnumero 31 and let $\msp$ be a $\frac{1}{4} (1, 1, 3)$ point.
We may assume that $\msp = \msp_t$.

\begin{Lem}
The defining polynomials of $X$ can be written as
\[
\begin{split}
\msF_1 &= t s + w f_3 + f_8, \\
\msF_2 &= t^2 y + t (s g_2 + g_6) + w^2 + s^2 y + s g'_6 + g_{10},
\end{split}
\]
where $f_i, g_i, g'_6 \in \mbC [x, y, z]$.
\end{Lem}

\begin{proof}
This is straightforward and we omit the proof.
\end{proof}

We set $S \coloneq (x = 0)_X$ and let $H \in |3 A|$ be a general member.
Note that $S$ is a normal surface by Lemma~\ref{lem:normqhyp} and it is easy to see that $S$ is quasismooth at $\msp$.

\subsubsection{Case: $y^4 \in f_8$}

Suppose that $y^4 \in f_8$.
We set $C \coloneq S \cap H$.

\begin{Lem}
The curve $C$ is irreducible and reduced, and it is quasismooth at $\msp$.
\end{Lem}

\begin{proof}
Rescaling $y$, we may assume that $\coeff_{f_8} (y^4) = -1$.
It is easy to see that $C$ is quasismooth at $\msp$.
We have $C = (x = z = 0)_X$ and
\[
\begin{split}
\overline{\msF}_1 &= t s - y^4, \\
\overline{\msF}_2 &= t^2 y + t (\alpha s y + \beta y^3) + w^2 + s^2 y + \gamma s y^3 + \delta y^5,
\end{split}
\]
where $\alpha, \beta, \gamma, \delta \in \mbC$ and $\overline{\msF}_i \coloneq \msF_i (0, y, 0, s, t, w)$.
We have $S \cap H \cap (s = 0) = \{\msp\}$.
Hence, by setting $s = 1$, it is enough to show that the complete intersection in $\mbA^3_{y, t, w}$ defined by the equations
\[
t - y^4 = \overline{\msF}_2 (y, 1, t, w) = 0
\]
is irreducible and reduced.
Eliminating $t = y^4$, it is then enough to show that the polynomial
\[
\overline{\msF}_2 (y, 1, y^4, w) = y^{10} + y^4 (\alpha y + \beta y^3) + w^2 + y + \gamma y^3 + \delta y^5
\]
is irreducible.
This is easily verified since the polynomial $y^{10} + y^4 (\alpha y + \beta y^3) +  y + \gamma y^3 + \delta y^5$ is not a square for any $\alpha, \beta, \gamma \in \mbC$.
\end{proof}

It follows that $\msp \in S \coloneq C \subset S \subset X$ is a flag of type $\mathrm{I}$ and by Proposition~\ref{prop:typeI}, we obtain
\[
\delta_{\msp} (X) \ge \min \left\{12, \ 4, \ \frac{4}{4 \cdot 3 \cdot 1 \cdot \frac{1}{6}} \right\} = 2.
\]

\subsubsection{Case: $y^4 \notin f_8$}

Suppose that $y^4 \notin f_8$.

\begin{Lem}
We have $H|_S = \Gamma + \Delta$, where $\Gamma$ and $\Delta$ are irreducible and reduced curves satisfying the following properties.
\begin{enumerate}
\item The curves $\Gamma$ and $\Delta$ are irreducible smooth rational curves, $\msp \in \Gamma$ and $\msp \notin \Delta$.
\item $(\Gamma^2) = (\Delta^2) = - 1/4$ and $(\Gamma \cdot \Delta) = 1/2$.
\end{enumerate}
\end{Lem}

\begin{proof}
Eliminating the variables $x, z$, the scheme $S \cap H = (x = z = 0)_X$ is isomorphic to the complete intersection in $\mbP (2, 4, 4, 5)$ defined by the equations
\[
t s = t^2 y + t (\alpha s y + \beta y^3) + w^2 + s^2 y + \gamma s y^3 + \delta y^5 = 0,
\] 
where $\alpha \coloneq \coeff_{g_2} (y)$, $\beta \coloneq \coeff_{g_6} (y^3)$, $\gamma \coloneq \coeff_{g'_6} (y^3)$ and $\delta \coloneq \coeff_{g_{10}} (y^5)$.
We set 
\[
\begin{split}
\Gamma &\coloneq (x = z = s = t^2 y + \beta t y^3 + w^2 + \delta y^5 = 0), \\
\Delta &\coloneq (x = z = t = w^2 + s^2 y + \gamma s y^3 + \delta y^5 = 0).
\end{split}
\]
By an argument similar to that in the proof of Lemma~\ref{lem:No8delta2C1S}, the curve $\Gamma$ is isomorphic to the hypersurface in $\mbP (1, 2, 2, 5)$ with homogeneous coordinates $\bar{y}, \bar{s}, \bar{t}, \bar{w}$ of weight $1, 2, 2, 5$, respectively, defined by the equation
\[
\bar{t}^2 y + \beta \bar{t} \bar{y}^3 + \bar{w} = 0.
\]
Hence $\Gamma$ is isomorphic to $\mbP^1$.
Similarly, we have $\Delta \cong \mbP^1$.
It is immediate to see that $\msp \in \Gamma$ and $\msp \notin \Delta$.
Hence (1) is proved.

We claim that $S$ is quasismooth along $\Gamma$.
The curve $\Gamma$ is quasismooth outside $\msp_y$.
The set $\Gamma \cap \Delta$ consists of one point and it is easy to see that $S$ is quasismooth at $\msp_y$ and $\Gamma \cap \Delta$.
This shows that $S$ is quasismooth along $\Gamma$.
We have $\Sing_{\Gamma} (S) = \{2 \times \frac{1}{2} (1, 1), \frac{1}{4} (1, 3)\}$.
Hence we have
\[
(\Gamma^2) = -2 + \frac{1}{2} + \frac{1}{2} + \frac{3}{4} = -\frac{1}{4}. 
\]
We have $(A \cdot \Gamma) = (A \cdot \Delta) = 1/4$.
By taking intersection numbers of $H|_S = \Gamma + \Delta$ and $\Gamma, \Delta$, we obtain
\[
(\Gamma \cdot \Delta) = \frac{1}{2} \quad \text{and} \quad (\Delta^2) = -\frac{1}{4},
\]
and (2) is proved.
\end{proof}

It follows that $\msp \in \Gamma \subset S \subset X$ is a flag of type $\mathrm{IIa}$ and, by Proposition~\ref{prop:typeII}, we have
\[
\delta_{\msp} (X) \ge \min \left\{ 4, \ \frac{1}{S (V_{\bullet, \bullet}^S; C)}, \ \frac{1}{4 S (W_{\bullet, \bullet, \bullet}^{S, C}; \msp)} \right\},
\]
where we compute $S (V_{\bullet, \bullet}^S; C)$ and $S (W_{\bullet, \bullet, \bullet}^{S, C}; \msp)$ as follows:
\[
\begin{split}
S (V_{\bullet, \bullet}^S; \Gamma) &= 18 \int_0^1 \left(\int_0^{\frac{1}{6} (1-u)} \left(\frac{1}{18} (1-u)^2 - \frac{1}{6} (1-u)v - \frac{1}{4} v^2 \right) d v \right. \\
& \left. \hspace{5cm} + \int_{\frac{1}{6} (1-u)}^{\frac{1}{3} (1-u)} \frac{9}{16} \left(\frac{1}{3} (1-u)-v\right)^2 d v \right) d u \\
&= \frac{5}{144}.
\end{split}
\]
\[
\begin{split}
S (W_{\bullet, \bullet, \bullet}^{S, C}; \msp) &=
18 \int_0^1 \left( \int_0^{\frac{1}{6} (1-u)} \left(\frac{1}{12} (1-u) + \frac{1}{4} v \right)^2 d v \right. \\
& \left. \hspace{2.5cm} + \int_{\frac{1}{6} (1-u)}^{\frac{1}{3} (1-u)} \frac{9}{16} \left(\frac{1}{3} (1-u)-v\right)^2 d v \right) d u + F_{\msp} (W_{\bullet, \bullet, \bullet}^{S, C}) \\
&= \frac{7}{576} + F_{\msp} (W_{\bullet, \bullet, \bullet}^{S, C}).
\end{split}
\]
We have $F_{\msp} (W_{\bullet, \bullet, \bullet}^{S, C}) = 0$ since $\ord_{\msp} (\Delta|_{\Gamma}) = 0$.
Thus, we obtain the estimate
\[
\delta_{\msp} (X) \ge \min \left\{ 4, \ \frac{144}{5}, \ \frac{144}{7} \right\} = 4.
\]

\subsection{Family \textnumero~45 and $\frac{1}{5} (1, 1, 4)$ points}
\label{sec:No45delta5}

Let $X = X_{10, 12} \subset \mbP (1, 2, 4, 5, 5, 6)$ be a member of family \textnumero~45 and let $\msp$ be a $\frac{1}{5} (1, 1, 4)$ point.
We may assume that $\msp = \msp_t$.

\begin{Lem}
The defining polynomials of $X$ can be written as
\[
\begin{split}
\msF_1 &= t s + w z + f_{10}, \\
\msF_2 &= t^2 y + t g_7 + w^2 + g_{12},
\end{split}
\]
where $f_{10} \in \mbC [x, y, z]$ and $g_7, g_{12} \in \mbC [x, y, z, s]$ with $y s \notin g_7$ and $\coeff_{g_{12}} (s^2 y) = \coeff_{g_{12}} (z^3) = 1$.
\end{Lem}

\begin{proof}
This is straightforward and we omit the proof.
\end{proof}

Let $S$ be a general member of the linear system $\langle w, y z, x^6\rangle \subset |6A|$ and set $H \coloneq (x = 0)_X$.
We set $C \coloneq S \cap H$.
Note that $S$ is a normal surface by Lemma~\ref{lem:normqhyp}.
It is easy to see that $S$ is quasismooth at $\msp$.

\begin{Lem}
The curve $C$ is irreducible and reduced, and it is quasismooth at $\msp$.
\end{Lem}

\begin{proof}
It is clear that $S \cap H$ is quasismooth at $\msp$ since $t s \in \msF_1$ and $t^2 y \in \msF_2$, $w$ and $x$ appear in the equations that define $S$ and $H$ in $X$, respectively.
It remains to show that $S \cap H$ is an irreducible and reduced curve.

Let $w = \lambda y z + \mu x^6$ be the equation which defines $S$ in $X$, where $\lambda, \mu \in \mbC$ are general.
By eliminating $x$ and $w$, we may identify $S \cap H$ with the subscheme 
\[
(t s + \lambda  yz^2 + y q (z, y) = t^2 y + \lambda^2 y^2 z^2 + s^2 y + c (y, z) = 0) \subset \mbP (2_y, 4_z, 5_s, 5_t),
\]
where $y q (y, z) \coloneq f_{10} (0, y, z)$ and $c (y, z) \coloneq g_{12} (0, y, z, 0)$.
We see that $(s = 0) \cap S \cap H$ is a finite set of points since $\lambda \in \mbC$ is general and $z^3 \in c (y, z)$.
By setting $s = 1$, it is enough to show that the affine curve
\[
C \coloneq (t + \lambda y z^2 + y q(y, z) = t^2 y + \lambda^2 y^2 z^2 + y + c (y, z) = 0) \subset \mbA^3_{y, z, t}
\]
is irreducible and reduced.
We write $q (y, z) = \alpha z^2 + \beta z y^2 + \gamma y^4$ and $c (y, z) = z^3 + \delta z^2 y^2 + \varepsilon z y^4 + \zeta y^6$, where $\alpha, \dots, \zeta \in \mbC$.
Eliminating $t = - \lambda y z^2 - y q (y, z)$, $C$ is isomorphic to the plane curve in $\mbA^2_{y, z}$ defined by the equation $h = 0$, where
\[  
\begin{split}
h &\coloneq y^3 (\lambda z^2 + q (y, z))^2 + \lambda^2 y^2 z^2 + y + c (y, z) \\
&= \eta^2 y^3 z^4 + (2 \eta \beta y^5 + 1) z^3 + y^2 ((2 \eta \gamma + \beta^2) y^5 + \lambda^2 + \delta) z^2 + \\
& \hspace{3cm} + y^4 (2 \beta \gamma y^5 + \varepsilon)z + y (\gamma^2 y^{10} + \zeta y^5 + 1), \\ 
&= \eta y^3 z^4 + \phi_3 z^3 + y^2 \phi_2 z^2 + y^4 \phi_1 z + y \phi_0, 
\end{split}
\]
where $\eta \coloneq (\lambda + \alpha)^2 \in \mbC$ is a nonzero constant and $\phi_i = \phi_i (y) \in \mbC [y]$ are defined by the above equation.
Note that $y \nmid \phi_3$ and $y \nmid \phi_0$.

Suppose that $h$ is reducible.
Then we can write $h = h_1 h_2$ for some $h_1, h_2 \in \mbC [y, z]$.
It is clear that $\deg_z h_i > 0$ for $i = 1, 2$.
We claim that the case $\deg_z h_1 = \deg_z h_2 = 2$ does not happen.
If it happens, then we can write $h_1 = a_2 z^2 + a_1 z + a_0$ and $h_2 = b_2 z^2 + b_1 z + b_0$ for some $a_i, b_i \in \mbC [y]$.
Comparing the coefficient of $z^4$, we have $a_2 b_2 = \eta y^3$.
If $y \mid a_2$ and $y \mid b_2$, then we have $y \mid \phi_3$, which is absurd.
Interchanging $h_1$ and $h_2$ if necessary, we may assume that $a_2 = y^3$ and $b_2 = \eta$.
Comparing the coefficients of $z^3, z^2, z^1, z^0$, we have
\begin{gather}
y^3 b_1 + \eta a_1 = \phi_3, \label{eq:Cz3}\\
y^3 b_0 + a_1 b_1 + \eta a_0 = y^2 \phi_2, \label{eq:Cz2} \\
a_1 b_0 + a_0 b_1 = y^4 \phi_1, \label{eq:Cz1} \\
a_0 b_0 = y \phi_0. \label{eq:Cz0}
\end{gather}
By \eqref{eq:Cz3}, we have $y \nmid a_1$.
By \eqref{eq:Cz0}, either $y \mid a_0$ or $y \mid b_0$.
If $y \mid a_0$, then $y \mid b_0$ by \eqref{eq:Cz1} since $y \nmid a_1$.
But this implies $y \mid \phi_0$, which is absurd.
Hence $y \nmid a_0$ and $y \mid b_0$.
By \eqref{eq:Cz2}, we have $y \nmid b_1$ since $y \nmid a_0$.
It follows that $y \mid a_1 b_0$ and $y \nmid a_0 b_1$, which is impossible by \eqref{eq:Cz1}.
Thus the claim is proved.

Hence, possibly interchanging $h_1$ and $h_2$, we may assume that $\deg_z h_1 = 1$ and $\deg_z h_2 = 3$ and we can write $h_1 = a_1 z - a_0$, $h_2 = b_3 z^3 + b_2 z^2 + b_1 z + b_0$ for some $a_i, b_i \in \mbC [y]$ with $a_1 \ne 0$ and $b_3 \ne 0$.
Plugging $z = a_0/a_1$ and clearing the denominators, we obtain
\begin{equation}
\label{eq:No45del5irr-1}
\begin{split}
& y^3 ((\lambda + \alpha) a_0^2 + \beta y^2 a_0 a_1 + \gamma y^4 a_1^2)^2 \\
& \hspace{2cm} + a_0^3 a_1 + (\lambda + \delta) y^2 a_0^2 a_1^2 + \varepsilon y^4 a_0 a_1^3 + \zeta y^6 a_1^4 + a_1^4 = 0.
\end{split}
\end{equation}
We may write $a_1 = y^m$ for some $0 \le m \le 3$ since $a_1 b_1 = (\lambda + \alpha)^2 y^3$.
Note that $a_0$ divides $y + \zeta y^6 + \gamma^2 y^{11}$.

Suppose that $m = 0$, that is, $a_1 = 1$.
By setting $y = 0$, we have $a_0 (0) \ne 0$.
Then we see that either $a_0$ is a nonzero constant or $\deg a_0 \in \{5, 10\}$ since $a_0$ divides $1 + \zeta y^5 + \gamma^2 y^{10}$.
If $a_0$ is a nonzero constant, then the degree $3$ terms of \eqref{eq:No45del5irr-1} are $(\lambda + \alpha)^2 y^3 a_0^2$ and this is nonzero since $\lambda$ is general.
This is absurd.

Suppose that $m \ge 1$.
Then $a_0$ is not divisible by $y$ because otherwise $h$ is divisible by $y$ and this is impossible.
Hence the constant term of $a_0$ is nonzero.
If $m = 1, 2$, then the least degree terms of \eqref{eq:No45del5irr-1} are $(\lambda + \alpha)^2 a_0 (0)^2 y^3$ which is nonzero.
This is absurd.
Hence $m = 3$.
As before, either $a_0$ is a nonzero constant or $\deg a_0 \in \{5, 10\}$.
If $\deg a_0 \ge 5$, then the highest degree terms of \eqref{eq:No45del5irr-1} are the highest degree terms of $(\lambda + \alpha)^2 a_0^4 y^3$ which is nonzero.
This is absurd and hence $a_0$ is a nonzero constant.
The highest degree term in \eqref{eq:No45del5irr-1} is $\gamma^2 y^7 a_1^4 = \gamma^2 y^{23}$, which implies $\gamma = 0$.
The equation \eqref{eq:No45del5irr-1} can be written as
\[
\gamma^2 y^{23} + (\zeta + 2 \beta \gamma) y^{18} + \beta^2 a_0^2 y^{13} + (\varepsilon a_0 + 1) y^{12} + ((\lambda + \delta) + 2 \beta (\lambda + \alpha)) a_0^2 y^8 + \cdots.
\] 
This shows that $\gamma = \zeta = \beta = \lambda + \delta= 0$ since $a_0 \ne 0$.
The last equality $\lambda + \delta = 0$ is impossible since $\lambda$ is general.
We obtain a contradiction and the proof is complete.
\end{proof}

It follows that $\msp \in S \coloneq C \subset S \subset X$ is a flag of type $\mathrm{I}$ and by Proposition~\ref{prop:typeI}, we obtain
\[
\delta_{\msp} (X) \ge \min \left\{ 24, \ 4, \ 2  \right\} = \frac{4}{3}.
\]

\subsection{Family \textnumero~64 and the $\frac{1}{5} (1, 2, 3)$ point}
\label{sec:No64delta5}

Let $X = X_{12, 16} \subset \mbP (1, 2, 5, 6, 7, 8)$ be a member of family \textnumero~64 and let $\msp = \msp_z$ be the $\frac{1}{5} (1, 2, 3)$ point.

\subsubsection{Case: $z^2 s \in \msF_2$}

We assume that $z^2 s \in \msF_2$.

\begin{Lem}
The defining polynomials of $X$ can be written as
\[
\begin{split}
\msF_1 &= z t + w f_4 + s^2 + s f_6 + f_{12}, \\
\msF_2 &= z^2 s + z g_{11} + w^2 + g_{16},
\end{split}
\]
where $f_i \in \mbC [x, y]$ and $g_i \in \mbC [x, y, s, t]$.
\end{Lem}

\begin{proof}
By the quasismoothness of $X$ at $\msp_t$, we have $z t \in \msF_1$.
The rest of the proof is straightforward and we omit it.
\end{proof}

We set $S \coloneq (x = 0)_X \sim A$ and let  $H \in |2 A|$ be a general member.
Note that $S$ is a normal surface by Lemma~\ref{lem:normqhyp} and it is easy to see that $S$ is quasismooth at $\msp$.
We set $C \coloneq S \cap H$.

\begin{Lem}
The curve $C$ is irreducible and reduced, and it is quasismooth at $\msp$.
\end{Lem}

\begin{proof}
We have
\[
C = (x = y = 0)_X = (x = y = z t + s^2 = z^2 s + w^2 = 0),
\]
and it is quasismooth outside $\msp_t$.
On the open set $(z \ne 0)_X$, $C$ is the $\bmu_5$-quotient of the complete intersection in $\mbA^3_{s, t, w}$ defined by the equations
\[
t + s^2 = s + w^2 = 0,
\]
which is clearly irreducible and reduced.
This shows that $C$ is irreducible and reduced since $C \cap (z = 0) = \{\msp_t\}$.
\end{proof}

It follows that $\msp \in C \subset S \subset X$ is a type $\mathrm{I}$ flag.
By Proposition~\ref{prop:typeI}, we have
\[
\delta_{\msp} (X) \ge \min \left\{ 8, 4, \frac{4}{5 \cdot 2 \cdot 1 \cdot \frac{2}{35}} \right\} = 4.
\]

\subsubsection{Case: $z^2 s \notin \msF_1$}

We assume that $z^2 s \notin \msF_1$.
Note that $\msp$ is not a maximal center by \cite[Remark~5.2]{OkadaI}.
We set $S \coloneq (x = 0)_X \sim A$.

\begin{Lem}
We have $\omult_{\msp} (S) = 2$.
\end{Lem}

\begin{proof}
We can write $\msF_2 = z^3 x + z^2 g_6 + z g_{11} + g_{16}$, where $g_i \in \mbC [x, y, s, t, w]$ with $s \notin g_6$.
Hence we can choose $s, t, w$ as local orbifold coordinates of $X$ at $\msp$.
We have $w^2 \in g_{16}$ by the quasismoothness of $X$.
This shows that $\omult_{\msp} (S) = 2$.
\end{proof}

Let $D \in |A|_{\mbQ}$ be an irreducible $\mbQ$-divisor such that $D \ne S$.
We see that $\{x, y, s, t\}$ is a $\msp$-isolating set of maximum degree $7$.
By Lemma~\ref{lem:isolomult}, we have $\omult_{\msp} (D) \le 2$.
This shows that $\lct_{\msp} (X; D) \ge 1/2$ for any $D \in |A|_{\mbQ}$, and hence we obtain 
\[
\alpha_{\msp} (X) \ge \frac{1}{2}.
\]

\subsection{Family \textnumero~76 and the $\frac{1}{7} (1, 3, 4)$ point}
\label{sec:No76delta7}

Let $X = X_{12, 20} \subset \mbP (1, 4, 5, 6, 7, 10)$ be a member of family \textnumero~76 and let $\msp = \msp_t$ be the $\frac{1}{7} (1, 3, 4)$ point.
Then we have $\delta_{\msp} (X) > 1$ by Proposition~\ref{prop:No76delta7}.

The proof of Theorem~\ref{thm:locdelta} is now complete.

\begin{proof}[Proof of \emph{Theorem~\ref{thm:Kst}}]
Any smooth member of family \textnumero~$1$ is K-stable by \cite[Theorem~1.3]{Zhuang20}.
In view of Remark~\ref{rem:KstNo60}, the assertions follow from Theorems~\ref{thm:alphasmpt}, \ref{thm:alphasingpt} and \ref{thm:locdelta}.
\end{proof}

\section{Table}
\label{sec:table}

We give the list of families of well-formed and quasismooth Fano $3$-fold WCIs of codimension $2$ and index $1$ that are indexed by $\msI^*_{\br}$ and their singular points in Table~\ref{table:BRcodim2}.

We explain the fourth column ``Singular points" that indicates the singular points of members of the corresponding family.
The symbol $\frac{1}{r} [a, r-a]$ (resp.\ $\frac{1}{2}, \frac{1}{3} (1, 1), \frac{1}{4}$) is an abbreviation for a cyclic quotient singular point of type $\frac{1}{r} (1, a, r-a)$ (resp.\ $\frac{1}{2} (1, 1, 1), \frac{1}{3} (1, 1, 2), \frac{1}{4} (1, 1, 3)$).
\begin{itemize}
\item The mark $\mathrm{QI}$ as a superscript of a singularity means that it is a $\mathrm{QI}$ center, meaning that the Kawamata blowup of a (general) member $X$ at the point initiates a Sarkisov-self link, called a quadratic involution.
\item The mark $\heartsuit$ as a subscript of a singularity means that \cite{OkadaI} proves only under some generality assumptions that either the point is not a maximal center or there exists a Sarkisov-self link initiated by the Kawamata blowup at the point.
See \S \ref{sec:known} for details.
\item The mark $\alpha$ as a subscript of a singularity means that it is not a maximal center and that $\alpha_{\msp} (X) \ge 1/2$.
This is proved in \S \ref{sec:alphasingpt}.
\item The mark $\beta$ as a subscript of a singularity means that $\delta_{\msp} (X) > 1$.
This is proved in \S \ref{sec:delta}.
\end{itemize}

\begingroup
\renewcommand{\arraystretch}{1.35}
\begin{table}[htb]
\caption{Birationally rigid Fano 3-fold WCIs of codimension 2}
\label{table:BRcodim2}
\centering
\begin{tabular}{clcl}
\toprule
\textnumero & $X_{d_1,d_2} \subset \mbP (a_0,\dots,a_5)$ & $(A^3)$ & Singular points \\ 
\midrule
8 & $X_{4, 6} \subset \mbP (1,1,2,2,2,3)$ & $1$ & $6 \ntimes \frac{1}{2}^{\QI}_{\delta}$ \\
14 & $X_{6,6} \subset \mbP (1,2,2,2,3,3)$ & $\frac{1}{2}$ & $9 \ntimes \frac{1}{2}_{\alpha}$ \\
20 & $X_{6,8} \subset \mbP (1,2,2,3,3,4)$ & $\frac{1}{3}$ & $6 \ntimes \frac{1}{2}_{\heartsuit, \alpha}, \ 2 \ntimes \frac{1}{3}^{\QI}_{\delta}$ \\
24 & $X_{6,10} \subset \mbP (1,2,2,3,4,5)$ & $\frac{1}{4}$ & $7 \ntimes \frac{1}{2}_{\alpha}, \ \frac{1}{4}^{\QI}_{\delta}$ \\
31 & $X_{8,10} \subset \mbP (1,2,3,4,4,5)$ & $\frac{1}{6}$ & $4 \ntimes \frac{1}{2}_{\alpha}, \ \frac{1}{3}^{*{\QI}}_{\heartsuit, \delta}, \ 2 \ntimes \frac{1}{4}^{\QI}_{\delta}$ \\
37 & $X_{8,12} \subset \mbP (1,2,3,4,5,6)$ & $\frac{2}{15}$ & $4 \ntimes \frac{1}{2}_{\heartsuit, \alpha}, \ 2 \ntimes \frac{1}{3}_{\heartsuit, \alpha}, \ \frac{1}{5} [1, 4]^{\QI}_{\delta}$ \\
45 & $X_{10,12} \subset \mbP (1,2,4,5,5,6)$ & $\frac{1}{10}$ & $5 \ntimes \frac{1}{2}_{\alpha}, \ 2 \ntimes \frac{1}{5} [1,4]^{\QI}_{\delta}$ \\
47 & $X_{10,12} \subset \mbP (1,3,4,4,5,6)$ & $\frac{1}{12}$ & $\frac{1}{2}_{\alpha}, \ 2 \ntimes \frac{1}{3}_{\heartsuit, \alpha}, \ 3 \ntimes \frac{1}{4}_{\alpha}$ \\
51 & $X_{10,14} \subset \mbP (1,2,4,5,6,7)$ & $\frac{1}{12}$ & $5 \ntimes \frac{1}{2}_{\alpha}, \ \frac{1}{4}_{\heartsuit, \alpha}, \ \frac{1}{6} [1,5]^{\QI}_{\delta}$ \\
59 & $X_{12,14} \subset \mbP (1,4,4,5,6,7)$ & $\frac{1}{20}$ & $2 \ntimes \frac{1}{2}_{\alpha}, \ 3 \ntimes \frac{1}{4}_{\heartsuit, \alpha}, \ \frac{1}{5} [1, 4]_{\alpha}$ \\
60 & $X_{12,14} \subset \mbP (2,3,4,5,6,7)$ & $\frac{1}{30}$ & $7 \ntimes \frac{1}{2}, \ 2 \ntimes \frac{1}{3}, \ \frac{1}{5} [2,3]$ \\
64 & $X_{12,16} \subset \mbP (1,2,5,6,7,8)$ & $\frac{2}{35}$ & $4 \ntimes \frac{1}{2}_{\alpha}, \ \frac{1}{5} [2,3]^{\QI}_{\delta}, \ \frac{1}{7} [1,6]^{\QI}_{\delta}$ \\ 
71 & $X_{14,16} \subset \mbP (1,4,5,6,7,8)$ & $\frac{1}{30}$ & $\frac{1}{2}_{\alpha}, \ 3 \ntimes \frac{1}{4}_{\heartsuit, \alpha}, \ \frac{1}{5} [2, 3]_{\alpha}, \ \frac{1}{6} [1,5]_{\alpha}$ \\
75 & $X_{14,18} \subset \mbP (1,2,6,7,8,9)$ & $\frac{1}{24}$ & $5 \ntimes \frac{1}{2}_{\alpha}, \ \frac{1}{3}_{\alpha}, \ \frac{1}{8}[1,7]^{\QI}_{\delta}$  \\
76 & $X_{12,20} \subset \mbP (1,4,5,6,7,10)$ & $\frac{1}{35}$ & $2 \ntimes \frac{1}{2}_{\alpha}, \ 2 \ntimes \frac{1}{5} [1, 4]_{\alpha}, \ \frac{1}{7} [3,4]^{\QI}_{\delta}$ \\
78 & $X_{16,18} \subset \mbP (1,4,6,7,8,9)$ & $\frac{1}{42}$ & $2 \ntimes \frac{1}{2}_{\alpha}, \ \frac{1}{3}_{\alpha}, \ 2 \ntimes \frac{1}{4}_{\alpha}, \ \frac{1}{7} [1,6]_{\alpha}$ \\
84 & $X_{18,30} \subset \mbP (1,6,8,9,10,15)$ & $\frac{1}{120}$ & $2 \ntimes \frac{1}{2}_{\alpha}, \ 2 \ntimes \frac{1}{3}_{\alpha}, \ \frac{1}{5}[1,4]_{\alpha}, \ \frac{1}{8} [1, 7]_{\alpha}$ \\
85 & $X_{24,30} \subset \mbP (1,8,9,10,12,15)$ & $\frac{1}{180}$ & $\frac{1}{2}_{\alpha}, \ \frac{1}{3}_{\alpha}, \ \frac{1}{4}_{\alpha}, \ \frac{1}{5} [2,3]_{\alpha}, \ \frac{1}{9} [1,8]_{\alpha}$ \\
\bottomrule
\end{tabular}
\end{table}
\endgroup

\bibliographystyle{alpha}
\bibliography{okada}

\end{document}